\documentclass[a4paper,11pt]{amsart}

\usepackage{hyperref}

\usepackage[margin=1 in]{geometry}

\makeatletter
\def\@secnumfont{\scshape}
\def\section{\@startsection{section}{1}%
  \z@{.7\linespacing\@plus\linespacing}{.5\linespacing}%
  {\Large\center\scshape}}
\def\subsection{\@startsection{subsection}{2}%
  \z@{.7\linespacing\@plus\linespacing}{.5\linespacing}%
  {\large\center\scshape}}
\makeatother

\makeatletter
\g@addto@macro\normalsize{
  \setlength\abovedisplayskip{8pt}
  \setlength\belowdisplayskip{8pt}
  \setlength\abovedisplayshortskip{8pt}
  \setlength\belowdisplayshortskip{8pt}
  }
\makeatother

\usepackage[english]{babel}
\usepackage{amsfonts}
\usepackage{mathrsfs}
\usepackage{bbm}
\usepackage{latexsym}
\usepackage{math dots}
\usepackage{amssymb}
\usepackage{mathtools}

\usepackage[normalem]{ulem}
\usepackage{todonotes}
\usepackage{enumitem}
\setlist{nolistsep}
\usepackage{amsthm}
\usepackage[capitalize]{cleveref}

\newtheoremstyle{plain}{3mm}{3mm}{\slshape}{}{\bfseries}{.}{.5em}{}
\newtheoremstyle{definition}{2mm}{2mm}{}{}{\bfseries}{.}{.5em}{}
\theoremstyle{plain}
	
\newtheorem{Theorem}{Theorem}

\newtheorem{Lemma}[Theorem]{Lemma}
\newtheorem{Proposition}[Theorem]{Proposition}
\newtheorem{Corollary}[Theorem]{Corollary}
\theoremstyle{definition}
\newtheorem{Definition}[Theorem]{Definition}
\newtheorem{Remark}[Theorem]{Remark}
\newtheorem{Example}[Theorem]{Example}

\theoremstyle{plain}
\newcounter{MainTheoremCounter}

\newtheorem{Maintheorem}[MainTheoremCounter]{Theorem}
\newtheorem{Maincorollary}[MainTheoremCounter]{Corollary}

\theoremstyle{plain}
\newtheorem*{namedthm}{\namedthmname}
\newcounter{namedthm}
\makeatletter
	\newenvironment{named}[2]
	{\def\namedthmname{#1}
	\refstepcounter{namedthm}
	\namedthm[#2]\def\@currentlabel{#1}}
	{\endnamedthm}
\makeatother

\usepackage{chngcntr}
\counterwithin{Theorem}{section}

\numberwithin{equation}{section}

\allowdisplaybreaks

\usepackage{xcolor}
\definecolor{Scarlet}{rgb}{0.85, 0.16 ,0.0}
\definecolor{Blue}{rgb}{0.0, 0.2, 0.6}
\definecolor{BlueC}{rgb}{0.0, 0.2, 0.6}
\definecolor{Green}{rgb}{0.0, 0.6 ,0.2}
\hypersetup{citecolor = Blue, colorlinks,
			linkcolor = black,
			urlcolor = Blue}

\addto\captionsenglish{%
}

\newcommand{\N}{\mathbb{N}}
\newcommand{\Z}{\mathbb{Z}}
\newcommand{\R}{\mathbb{R}}
\newcommand{\C}{\mathbb{C}}
\newcommand{\Q}{\mathbb{Q}}
\newcommand{\Hilb}{\mathscr{H}}

\newcommand{\Oh}{{\rm O}}
\newcommand{\oh}{{\rm o}}

\newcommand{\otherlim}[2]{\underset{#2}{{#1}\text{-}\lim}~}
\newcommand{\otherlimsup}[2]{\underset{#2}{{#1}\overline{\text{-}\lim}}~}
\newcommand{\textlim}[1]{{${#1}\text{-}\lim$}}

\newcommand{\define}[1]{\emph{#1}}

\newcommand{\lhk}{\lvert\!|\!|}
\newcommand{\rhk}{|\!|\!\rvert}

\renewcommand{\epsilon}{\varepsilon}
\renewcommand{\leq}{\leqslant}
\renewcommand{\geq}{\geqslant}
\renewcommand{\vec}[1]{{\bf #1}}
\renewcommand{\setminus}{\backslash}

\newcommand{\E}{\mathbb{E}}
\newcommand{\B}{\mathcal{B}}
\newcommand{\1}{1}

\newcommand{\xbm}{(X,\mathcal{B},\mu)}

\newcommand{\xbmt}{(X,\mathcal{B},\mu,T)}

\renewcommand{\d}{~\mathsf{d}}

\newcommand{\F}{\mathcal{F}}
\newcommand{\T}{\mathbb{T}}

\begin{document}

\title{Single and Multiple recurrence along non-polynomial sequences}
\author{Vitaly Bergelson}
\thanks{The first author gratefully acknowledges the support of the NSF under grant DMS-1500575. The second author gratefully acknowledges the support of the NSF under grant DMS-1700147. The third author is supported by the National Science Foundation under grant number DMS~1901453}
\author{Joel Moreira}
\author{Florian K.\ Richter}

\date{\small \today}

\begin{abstract}
We establish new recurrence and multiple recurrence results for a rather large family $\F$ of non-polynomial functions which contains tempered functions and (non-polynomial) functions from a Hardy field with polynomial growth.

In particular, we show that, somewhat surprisingly (and in the contrast to the multiple recurrence along polynomials), the sets of return times along functions from $\F$ are \define{thick}, i.e., contain arbitrarily long intervals.
A major component of our paper is a new result about equidistribution of sparse sequences on nilmanifolds, whose proof borrows ideas from the work of Green and Tao \cite{Green_Tao12a}.

Among other things, we show that for any $f\in\F$, any invertible probability measure preserving system $\xbmt$, any $A\in\B$ with $\mu(A)>0$, and any $\epsilon>0$, the sets of returns
$$\big\{n\in\N:\mu(A\cap T^{-\lfloor f(n)\rfloor}A)>\mu^2(A)-\epsilon\big\}$$
$$\big\{ n\in\N:
\mu\big(A\cap T^{-\lfloor f(n)\rfloor}A\cap T^{-\lfloor f(n+1)\rfloor}A\cap\cdots\cap T^{-\lfloor f(n+k)\rfloor}A\big)>0\big\}$$
are thick.

Our  recurrence theorems  imply, via Furstenberg's correspondence principle,  some new combinatorial results. For example, we show that given a set $E\subset\N$ with positive upper density, for every $k\in\N$ there are $a,n\in\N$ such that
$$\big\{a,a+\lfloor f(n)\rfloor,\cdots,a+\lfloor f(n+k)\rfloor\big\}\subset E.$$
When $f(n)=n^c$, with $c>0$ non-integer, this result provides a positive answer to a question posed by Frantzikinakis \cite[Problem 23]{Frantzikinakis11}.
\end{abstract}

\maketitle
\small
\tableofcontents
\thispagestyle{empty}
\normalsize


\section{Introduction}
\label{sec_intro}

Let $\xbmt$ be a probability measure preserving system. The classical Poincar{\'e} recurrence theorem states that for any $A\in \B$ with $\mu(A)>0$ there exists $n\in\N=\{1,2,\ldots\}$ such that $\mu(A\cap T^{-n}A)>0$.
Over the years it was revealed that the set of return times
$$
R_A\coloneqq \{n\in\N: \mu(A\cap T^{-n}A)>0\}
$$
has quite intricate combinatorial and number-theoretical properties. For example, $R_A$ contains a perfect square \cite[Proposition 1.3]{Furstenberg77}.
This fact, in turn, implies a theorem of S{\'a}rk{\"o}zy \cite{Sarkozy78} which says that for any set $E\subset\N$ with positive \define{upper Banach density}\footnote{The \define{upper Banach density} of a set $E\subset\N$ is defined as $d^*(E)\coloneqq \limsup_{N-M\to\infty} {|E\cap [M,N]|}/(N-M)$.} there exist $n\in\N$ and $a,b\in E$ such that $a-b=n^2$ (see \cite[Theorem 1.2]{Furstenberg77}).

A more general version of this result, proved by Furstenberg in \cite{Furstenberg81a,Furstenberg81}, asserts that for any $g\in\Q[x]$ with $g(\Z)\subset \Z$ and $g(0)=0$,
$$
\lim_{N\to\infty}\frac{1}{N}\sum_{n=1}^N \mu(A\cap T^{-g(n)}A) >0,
$$
which implies that the \define{upper density}\footnote{The \define{upper density} of a set $E\subset\N$ is defined as $\overline{d}(E)\coloneqq \limsup_{N\to\infty} {|E\cap [1,N]|}/{N}$.} of the set $\{n\in\N: \mu(A\cap T^{-g(n)}A)>0\}$ is positive.
One can actually show that for any $\epsilon>0$ there exists $k\in\N$ such that
$$
\lim_{N-M\to\infty}\frac{1}{N-M}\sum_{n=M}^{N-1} \mu(A\cap T^{-g(kn)}A) >\mu(A)^2-\epsilon,
$$
which gives the following polynomial version of the classical Khintchine recurrence theorem \cite{Khintchine34}. (This polynomial Khintchine recurrence theorem also follows from a stronger result obtained in \cite{Bergelson_Furstenberg_McCutcheon96}.)

\begin{Theorem}[Polynomial Khintchine recurrence theorem]
\label{thm_poly_khin}
For any invertible probability measure preserving system $\xbmt$, any $A\in\B$, any $g\in\Q[x]$ with $g(\Z)\subset \Z$ and $g(0)=0$, and any $\epsilon>0$, the set of optimal return times
$
R_{\epsilon, A}= \big\{n\in\N:\mu(A\cap T^{-g(n)}A)>\mu^2(A)-\epsilon\big\}
$
is syndetic, i.e., there exists $l\in\N$ such that $R_{\epsilon, A}$ has non-empty intersection with every interval of length bigger or equal to $l$.
\end{Theorem}

In \cite[Theorem 7.1]{Bergelson_Knutson09} a mean ergodic theorem along \define{tempered sequences}\footnote{\label{ftn_tempered}
A real-valued function $f$ defined on a half-line $[a,\infty)$ is called a \define{tempered} function if there exist $\ell\in\N$ such that $f$ is $\ell$ times continuously differentiable, $f^{(\ell)}(x)$ tends monotonically to zero as $x\to\infty$, and $\lim_{x\to\infty} x |f^{(\ell)}(x)|=\infty$. (This notion was introduced in the context of the theory of uniform distribution by Cigler in \cite{Cigler68}.)} was proved, which implies that for any tempered function $f$, any invertible probability measure preserving system $\xbmt$ and any $A\in\B$,
$$
\lim_{N\to\infty}\frac{1}{N}\sum_{n=1}^N \mu(A\cap T^{-\lfloor f(n)\rfloor}A) \geq \mu(A)^2,
$$
where $\lfloor x\rfloor$ denotes the integer part of a real number $x$ (see \cite[Corollary 7.2]{Bergelson_Knutson09}).
This gives a large class of sequences $f(n)$ for which the set of optimal return times
\begin{equation}\label{eq_intro_iii}
R_{\epsilon, A}\coloneqq \big\{n\in\N:\mu(A\cap T^{-\lfloor f(n) \rfloor}A)>\mu^2(A)-\epsilon\big\}
\end{equation}
has positive upper density.
Examples include $f(n)= b n^c \log^r(n)$, where $b\in\R\setminus\{0\}$, $c>0$ with $c\notin \N$ and $r\geq0$, and $f(n)= n^c(\cos(\log^r (n)) + 2)$, where $c > 0$ and $0 < r < 1$.

The goal of this paper is to establish new results about the set of optimal returns $R_{\epsilon, A}$ defined in \eqref{eq_intro_iii} and about the set of multiple returns
\begin{equation}
\label{eq_intro_b}
R^{(k)}_{A}\coloneqq
\big\{n\in\N:\mu(A\cap T^{-\lfloor f(n) \rfloor}A\cap T^{-\lfloor f(n+1)\rfloor}A\cap \ldots\cap T^{-\lfloor f(n+k) \rfloor}A)>0\big\}, \quad k\in\N,
\end{equation}
for a rather large class of sequences $f(n)$. As is shown in \cref{lem_nal}, this class includes:
\begin{enumerate}	
[label={\small$\circ$}
]
\item
all sequences $f(n)$, where $f\colon[1,\infty
)\to\R$ is a tempered function;
\item
all sequences $f(n)$, where $f\colon[1,\infty)\to\R$ is a function from a \define{Hardy field}\footnote{A \define{germ
at $\infty$} of a given function $f\colon[a,\infty)\to\R$ is any equivalence class of functions $g\colon[b,\infty)\to\R$ under the equivalence relationship $(f\sim g) \Leftrightarrow \big(\exists c\geq \max\{a,b\}
~\text{such that}~f(x)=g(x)~\text{for all}~x\in [c,\infty)\big)$.
Let $G$ denote the set of all germs at $\infty$ of real valued functions defined on a half-line $[a,\infty)$.
Any subfield of the ring $(G,+,\cdot)$, where $+$ and $\cdot$ denote pointwise addition and multiplication, that is closed under differentiation is called a
\define{Hardy field} \cite{Hardy12, Hardy71}. By abuse of language, we say that a function $f\colon[1,\infty)\to\R$ belongs to some Hardy field if its germ belongs to that Hardy field.} with the property that for some $\ell\in \N\cup\{0\}$ one has $\lim_{x\to\infty }f^{(\ell)}(x)=\pm\infty$ and $\lim_{x\to\infty }f^{(\ell+1)}(x)=0$.
\end{enumerate}

Before formulating our results, let us define the class of functions that we will be dealing with.

\begin{Definition}
\label{def:class_F}
Given a function $f\colon\N\to \R$ let $\Delta f$ denote its \define{first order difference} (or \define{discrete derivative}),
$
\Delta f(n)\coloneqq  f(n+1)-f(n).
$
Define
\begin{eqnarray*}
\F_0&\coloneqq &\left\{f\colon\N\to \R: \lim_{n\to\infty}f(n)=0,~f(n)~\text{eventually monotone},~\text{and}~\left|\sum_{n=1}^\infty f(n)\right|=\infty\right\};\\
\F_{\ell+1}&\coloneqq &\big\{f\colon\N\to \R: \Delta f\in\F_{\ell}\big\},\quad \text{for}~\ell\in\N\cup\{ 0\},
\end{eqnarray*}
and let $\F\coloneqq\bigcup_{\ell=1}^\infty \F_\ell$.
\end{Definition}

Observe that the function $f(n)=n^\alpha$ belongs to $\F$ if and only if $\alpha>0$ and is not an integer (in fact, in this case we have $n^\alpha\in \F_{\lfloor \alpha\rfloor+1}$).
While our class $\F$ is vastly more general than the family of functions $\{n\mapsto n^\alpha:\alpha\in(0,\infty)\setminus\N\}$, it may be convenient for the reader to keep this example in mind, as our results are already new and interesting even for this family.

We will show that for any $f\in\F$ the sets $R_{\epsilon, A}$ and $R^{(k)}_{A}$ defined in \eqref{eq_intro_iii} and \eqref{eq_intro_b} possess a somewhat unexpected combinatorial property that stands in contrast to the syndeticity featured in Khintchine's recurrence theorem and in \cref{thm_poly_khin}.
The notion in question is dual to the notion of a syndetic set:

\begin{Definition}
$R\subset\N$ is called \define{thick} if it contains arbitrarily long intervals.
\end{Definition}
Notice that a set is thick if and only if its complement is not syndetic.

While for sequences $f$ belonging to the class $\F$ the sets $R_{\epsilon, A}$ are, in general, not syndetic (cf. \cite[Remark after Theorem 2.3]{Frantzikinakis10}), they do possess a combinatorial property that we will presently introduce and which can be viewed as a generalized version of syndeticity.

\begin{Definition}\label{def_Wsyndetic}
Let $W\in \F_1\cup\{\operatorname{Id}:n\mapsto n\}$ be positive.
A set $R\subset \N$ is called \define{W-syndetic} if there exists $l \in\N$ such that for any interval $[M,N]\subset \N$ with $W(N)-W(M)\geq l$ one has
$$\sum_{n\in R\cap [M,N]} \Delta W(n)\geq 1.$$
\end{Definition}

Any $W$-syndetic set $R$ has positive \define{lower} and \define{upper $W$-density}:
$$
\underline{d}_W(R)\coloneqq \liminf_{N\to\infty}\frac{1}{W(N)}\sum_{n\in R\cap[1,N]} \Delta W(n) >0.
$$
$$
\overline{d}_W(R)\coloneqq \limsup_{N\to\infty}\frac{1}{W(N)}\sum_{n\in R\cap[1,N]} \Delta W(n) >0.
$$
Note that when $W=\operatorname{Id}$ we recover from \cref{def_Wsyndetic} the usual notion of syndeticity, and $\underline{d}_W(R)$ and $\overline{d}_W(R)$ become the usual upper and lower density, denoted simply as $\underline{d}(R)$ and $\overline{d}(R)$.
Also, for any $W\in\F_1$ the inequality $\underline{d}(R)\leq\underline{d}_W(R)\leq\overline{d}_W(R)\leq \overline{d}(R)$ holds\footnote{The middle inequality is trivial, and the left and right inequalities can be established by a method similar to that used in the proof of \cite[Lemma 1.7.1]{Kuipers_Niederreiter74}. See also \cite[Theorem 3.5]{Giuliano_Grekos07} for a uniform version.}, which implies that any $W$-syndetic set has positive upper density.

Our first main result pertains to the sets $R_{\epsilon, A}$.
\begin{Maintheorem}
\label{thm_mainsingle}
Let $\ell\in\N$, $f\in\F_{\ell}$ and define $W\coloneqq \Delta^{\ell-1} f$. Then for any invertible probability measure preserving system $(X,\B,\mu,T)$, any $A\in \B$ and any $\epsilon>0$ the set
$$
R_{\epsilon, A}= \big\{n\in\N:\mu(A\cap T^{-\lfloor f(n)\rfloor}A)>\mu^2(A)-\epsilon\big\}
$$
is thick and $W$-syndetic.
\end{Maintheorem}

We describe now the results which deal with the sets of multiple recurrence $R_A^{(k)}$ defined in \eqref{eq_intro_b}.
Let us first recall the polynomial multiple recurrence theorem obtained in \cite{Bergelson_McCutcheon96} (see also \cite{Bergelson_Leibman96, Bergelson_McCutcheon00, Bergelson_Leibman_Lesigne08}).

\begin{Theorem}[Polynomial multiple recurrence theorem]
\label{theorem:BL96}
For any $k\in\N$, any invertible probability measure preserving system $\xbmt$, any $A\in\B$ with $\mu(A)>0$ and any polynomials $g_1,\ldots,g_k\in \Q[x]$ with $g_i(\Z)\subset\Z$ and $g_i(0)=0$ for $i=1,\ldots,k$, the set
\begin{equation}\label{eq_intro1.3}
R=\big\{n\in\N:\mu(A\cap T^{-g_1(n)}A\cap \ldots\cap T^{-g_k(n)}A)>0\big\}
\end{equation}
is syndetic.
\end{Theorem}

By a classical theorem of algebra (cf. \cite{Ostrowski19} or \cite{Polya19}) any polynomial $g\in \Q[x]$ with $g(\Z)\subset\Z$ and $\deg g=d$ can be represented as an integer linear combination of binomial coefficients, $g(n)= a_d \binom{n}{d}+\ldots+a_1 \binom{n}{1}+a_0 \binom{n}{0}$. Since $\Delta \binom{n}{d}= \binom{n}{d-1}$, it follows that for any $g\in \Q[x]$ with $g(\Z)\subset\Z$ and $\deg g=d$ there exists $p(x)=a_d x^{d}+\ldots+a_1x+a_0 \in\Z[x]$ such that $g(n)=p(\Delta)\binom{n}{d}$, where the operator $p(\Delta)$ is defined as $p(\Delta) f\coloneqq a_d \Delta^{d}f+\ldots+a_1\Delta f +a_0f$.
This leads to the following equivalent formulation of \cref{theorem:BL96}.

\begin{Theorem}
\label{theorem:BL96-2}
Let $k,d\in\N$ and $\xbmt$ an invertible probability measure preserving system.
For any $A\in\B$ with $\mu(A)>0$ and any polynomials $p_1,\ldots,p_k\in \Z[x]$ with $\deg (p_i)<d$ for $i=1,\ldots,k$, the set
\begin{equation}\label{eq_intro1.4-before}
R=\Big\{n\in\N:\mu\big(A\cap T^{-p_1(\Delta) \binom{n}{d}}A\cap \ldots\cap T^{- p_k(\Delta)\binom{n}{d}}A\big)>0\Big\}
\end{equation}
is syndetic.
\end{Theorem}

Our second main result is a multiple recurrence result in the spirit of \cref{theorem:BL96-2}.

\begin{Maintheorem}
\label{thm_mainmultiple}
Let $f\in\F_{\ell}$ and define $W\coloneqq \Delta^{\ell-1} f$. For any $k\in\N$, any invertible probability measure preserving system $\xbmt$, any $A\in\B$ with $\mu(A)>0$ and any polynomials $p_1,\ldots,p_k\in \Z[x]$, the set
\begin{equation}\label{eq_intro1.4}
R=\Big\{n\in\N:\mu\big(A\cap T^{-\lfloor p_1(\Delta) f(n)\rfloor }A\cap \ldots\cap T^{-\lfloor p_k(\Delta)f(n)\rfloor}A\big)>0\Big\}
\end{equation}
is thick and $W$-syndetic.
\end{Maintheorem}

\begin{Remark}
Unlike the return-time set in \eqref{eq_intro1.4-before}, the return-time set in \eqref{eq_intro1.4} is in general not syndetic.
Nonetheless, \cref{thm_mainmultiple} shows that the set $R$ in \eqref{eq_intro1.4} is combinatorially rich in a new and different way.
\end{Remark}

\begin{Remark}\label{rmrk_comparison}
Multiple recurrence results dealing with non-polynomial functions from Hardy fields were obtained in \cite{Frantzikinakis09,Frantzikinakis10,Frantzikinakis15,Frantzikinakis_Wierdl09}.
In these papers it was established (among other things) that, for any $k\in\N$, any invertible probability measure preserving system $\xbmt$, any $A\in\B$ with $\mu(A)>0$ and for a function $f$ from a Hardy field satisfying certain growth conditions there exists $n\in\N$ such that
$$\mu(A\cap T^{\lfloor f(n)\rfloor}A\cap T^{2\lfloor f(n)\rfloor}A\cap\cdots\cap T^{k\lfloor f(n)\rfloor}A)>0.$$
Additionally, it was shown that under the same conditions, given functions $f_1,\dots,f_k$ from a Hardy field of different growth rates, there exists $n\in\N$ such that
$$\mu(A\cap T^{\lfloor f_1(n)\rfloor}A\cap T^{\lfloor f_2(n)\rfloor}A\cap\cdots\cap T^{\lfloor f_k(n)\rfloor}A)>0.$$
By way of contrast, \cref{thm_mainmultiple} applies to the functions $p_1(\Delta) f,\dots,p_k(\Delta) f$ which, in general, do not have different growth rates (in fact they may not even be linearly independent), nor are they all a scalar multiple of a common function.
These novel features are established with the help of some new ideas and techniques which are particularly patent in \cref{sec_u}.
\end{Remark}

By utilizing the identity
$$
f(n+k)=(1+\Delta)^k f(n),\qquad\forall n\in\N,~\forall k\in\N\cup\{0\},
$$
one obtains the following corollary from \cref{thm_mainmultiple}.

\begin{Maincorollary}
\label{cor_multipleconsecutiverecurrence_0}
Let $f\in\F_{\ell}$, let $\xbmt$ be an invertible probability measure preserving system and define $W\coloneqq \Delta^{\ell-1} f$. Then for any $A\in{\mathcal B}$ with $\mu(A)>0$ and any $k\in\N$ the set
\begin{equation*}
R^{(k)}_{A}=\big\{ n\in\N:
\mu\big(A\cap T^{\lfloor f(n)\rfloor}A\cap T^{\lfloor f(n+1)\rfloor}A\cap\cdots\cap T^{\lfloor f(n+k)\rfloor}A\big)>0\big\}
  \end{equation*}
is thick and $W$-syndetic.
\end{Maincorollary}

The fact that the set $R^{(k)}_{A}$ is non-empty for the special case when $f(n)=n^\alpha$ answers \cite[Problem 23]{Frantzikinakis11}.
On the other hand, we remark that \cref{cor_multipleconsecutiverecurrence_0} does \emph{not} hold when $f$ is a polynomial with integer coefficients.
To illustrate this contrast, fix $\alpha\geq1$. Then the following is true if and only if $\alpha$ is not an integer:
For every invertible probability measure preserving system $\xbmt$, for any $A\in{\mathcal B}$ with $\mu(A)>0$ and for every $k\in\N$, there exists $n\in\N$ such that
\begin{equation}\label{equation_last}
  \mu(A\cap T^{\lfloor n^\alpha\rfloor}A\cap T^{\lfloor(n+1)^\alpha\rfloor}A\cap\cdots\cap T^{\lfloor(n+k)^\alpha\rfloor}A)>0.
\end{equation}
Indeed, when $\alpha\notin\N$ the function $f(n)=n^\alpha$ belongs to $\F$ and thus \eqref{equation_last} follows from \cref{cor_multipleconsecutiverecurrence_0}.
When $\alpha$ is an integer, one can take $\xbmt$ to be the rotation on $2$ points and $k=1$ to obtain a counterexample.
In fact, one can obtain a counterexample for any system with a non-trivial Kronecker factor.

\cref{cor_multipleconsecutiverecurrence_0} implies, via Furstenberg's correspondence principle (see \cite[Section 1]{Furstenberg77} or \cite[Theorem 1.1]{Bergelson87b}) several interesting combinatorial applications.
For instance, it follows that for every $f\in\F$ and $k\in\N$, a set with positive Banach density must contain a configuration of the form
$$\{a,~ a+\lfloor f(n)\rfloor,~ a+\lfloor f(n+1)\rfloor,~\ldots,~a+\lfloor f(n+k)\rfloor\}$$
for some $a,n\in\N$.
A stronger result is \cref{cor_originalmain} below.

\begin{Remark} We would like to note that
  \cref{cor_multipleconsecutiverecurrence_0} implies Szemer\'edi's theorem \cite{Szemeredi75} which states that every set $A\subset\N$ with positive upper Banach density contains arbitrarily long arithmetic progressions.
  To verify this claim, let $A\subset\N$ have positive upper Banach density and let $k\in\N$.
  Take any function $f\in\F_2$ and let $W=\Delta f$.
  Since $W\in\F_1$, we have $\Delta W(n)\to 0$ as $n\to\infty$, and therefore $|W(n+10k^2)-W(n)|\to0$ as $n\to\infty$.
  Using \cref{cor_multipleconsecutiverecurrence_0} find $a,n\in\N$ with $n$ large enough so that $|W(n+10k^2)-W(n)|<1/(5k)$ and such that
  $$\big\{a+\lfloor f(n)\rfloor,a+\lfloor f(n+1)\rfloor,\cdots,a+\lfloor f(n+10k^2)\rfloor\big\}\subset A.$$
  Using the pigeonhole principle, find $t\in\{1,2,\dots,5k\}$ such that $\|tW(n)\|_\T<1/(10k)$, where $\|x\|_\T$ denotes the distance from the real number $x$ to the closest integer.
  Observe that the set $\big\{a+f(n),a+f(n+t),\cdots,a+f(n+2kt)\big\}$ is, up to a small error, the arithmetic progression $\{a+f(n),a+f(n)+tW(n),a+f(n)+2tW(n),\cdots,a+f(n)+2ktW(n)\}$ whose common difference $tW(n)$ is almost an integer.
  Therefore either the first $k$ or the last $k$ terms of the set
  $$\big\{a+\lfloor f(n)\rfloor,a+\lfloor f(n+t)\rfloor,\cdots,a+\lfloor f(n+2kt)\rfloor\big\}$$
  form an arithmetic progression of length $k$.
\end{Remark}

We will derive \cref{thm_mainmultiple} from a more general result dealing with weighted multiple ergodic averages along sequences from $\F$. To state this theorem we need to introduce some notation first.

\begin{Definition}
Let $a\colon\N\to\C$ be bounded, let $W\in\F_1$ and let $M<N\in\N$.
Define
$$\E_{n\in[M,N]}^W a(n)\coloneqq \frac1{W(N)-W(M)}\sum_{n=M}^N\Delta W(n)a(n).$$
We denote by
$$\otherlim{W}{n\to\infty}a(n)\coloneqq \lim_{N\to\infty}\E_{n\in[1,N]}^W a(n)$$
the \define{(simple) Riesz mean}\footnote{We follow the terminology of Kuipers and Niederreiter \cite[Example 7.3 on page 63]{Kuipers_Niederreiter74}, but use a different notation.} of $a$ with respect to $W$, whenever this limit exists, and
$$\otherlim{UW}{n\to\infty}a(n)\coloneqq \lim_{W(N)-W(M)\to\infty}\E_{n\in[M,N]}^W a(n)$$
the \define{uniform Riesz mean} of $a$ with respect to $W$, whenever this limit exists.
\end{Definition}

\begin{Remark}
When $W=\operatorname{Id}$ the
uniform Riesz mean with respect to $W$ is the same as the more well-known \define{Uniform Ces\`aro mean}, which is a natural form of convergence often used in ergodic theory and is also closely related to the classical notion of \define{well-distribution} introduced by Hlawka \cite{Hlawka55a} and Peterson \cite{Petersen56} in the 1950s (cf. \cref{def_relative-u.d.} below).
\end{Remark}

\begin{Maintheorem}
\label{thm_multipleconsecutiverecurrence_0}
Let $k\in\N$, $\ell\in\N$, $f\in\F_{\ell}$, $W\coloneqq \Delta^{\ell-1} f$ and $p_1,\ldots,p_k\in \Z[x]$. Let $\xbmt$ be an invertible probability measure preserving system.
\begin{enumerate}
\item
For any $h_1,\ldots,h_k\in L^\infty\xbm$ the limit
$$
\otherlim{UW}{n\to\infty}
T^{\lfloor p_1(\Delta)f(n)\rfloor}h_1 T^{\lfloor p_2(\Delta)f(n)\rfloor}h_2\cdot\ldots\cdot T^{\lfloor p_k(\Delta)f(n)\rfloor}h_k
$$
exists in $L^2\xbm$.
\item
For any $A\in{\mathcal B}$ with $\mu(A)>0$,
  \begin{equation}
\otherlim{UW}{n\to\infty}\mu\big(A\cap T^{-\lfloor p_1(\Delta)f(n)\rfloor}A\cap T^{-\lfloor p_2(\Delta)f(n)\rfloor}A\cap\cdots\cap T^{-\lfloor p_k(\Delta)f(n)\rfloor}A\big)>0.
  \end{equation}
\end{enumerate}
\end{Maintheorem}

\begin{Remark}
If $f$ is a tempered function (see \cref{ftn_tempered}) then we conjecture that all uniform Riesz means with respect to $W$ appearing in \cref{thm_multipleconsecutiverecurrence_0} can actually be replaced with conventional Ces{\`a}ro averages, but this is not a consequence of our results. On the other hand, there are functions in $\F$, which even grow faster than linear, for which the results in \cref{thm_multipleconsecutiverecurrence_0} do not hold when Riesz means are replaced by Ces{\`a}ro averages, even for the special case $k=1$ and $p_1(\Delta)=\text{id}$. Examples of such functions $f$ can be found in \cite[p.\ 65]{BKQW05}.
\end{Remark}

As a corollary of \cref{thm_multipleconsecutiverecurrence_0} we get the following strengthening of \cref{thm_mainmultiple}, which quantifies the appearance frequency of long intervals in the set $R$ defined in \eqref{eq_intro1.4}.

\begin{Maincorollary}\label{cor_syndeticallythick}
Take $\ell\in\N$, $f\in\F_{\ell}$, $k\in\N$, $p_1,\ldots,p_k\in \Z[x]$, $\xbmt$ an invertible probability measure preserving systems and $A\in\B$ with $\mu(A)>0$.
Define $W\coloneqq \Delta^{\ell-1} f$.
Then for any $L\in\N$ the set of $m$ such that
$$
[m,m+L]\subset \Big\{n\in\N:\mu\big(A\cap T^{-\lfloor p_1(\Delta) f(n)\rfloor }A\cap \ldots\cap T^{-\lfloor p_k(\Delta)f(n)\rfloor}A\big)>0\Big\}
$$
is $W$-syndetic.
\end{Maincorollary}

Finally, we present a combinatorial application of \cref{cor_syndeticallythick}, which can de derived using Furstenberg's correspondence principle.
\begin{Maincorollary}\label{cor_originalmain}
Take $\ell\in\N$, $f\in\F_{\ell}$, $k\in\N$, $p_1,\ldots,p_k\in \Z[x]$, and let $E\subset\N$ have positive upper Banach density.
Then $E$ contains configurations of the form
\begin{equation}\label{eq_cor_originalmain}
\big\{a,~ a+\lfloor p_1(\Delta)f(n)\rfloor,~ a+\lfloor p_2(\Delta)f(n)\rfloor,~\ldots,~a+\lfloor p_k(\Delta)f(n)\rfloor\big\}
\end{equation}
with $a,n\in\N$.
Moreover, if for each $n\in\N$ one denotes by $A(n)$ the set of $a\in\N$ for which the configuration in \eqref{eq_cor_originalmain} is contained in $E$, then for any $L\in\N$ the set of $m$ such that
$$
[m,m+L]\subset \Big\{n\in\N:d^*\big((A(n)\big)>0\Big\}
$$
is $W$-syndetic, where $W\coloneqq \Delta^{\ell-1} f$.
\end{Maincorollary}

It is natural to ask whether \cref{thm_mainsingle} can be improved in the spirit of \cref{cor_originalmain}.
Unfortunately, our proof of \cref{thm_mainsingle} (see \cref{sec_SingleRec}) does not seem to allow for such an extension
.



~

\paragraph{\textit{\textbf{Organization of the paper.}}}

To establish (multiple) recurrence it is often helpful to start by analyzing translations on tori, where recurrence results can be obtained using the theory of uniform distribution.
In \cref{sec_u.d.-along-riesz-means} we study well-distribution of orbits along sequences from $\F$ in compact abelian groups.
{To that end we employ Riesz means, which are important for two reasons.
On the one hand, they allow us to treat functions, such as $f(n)=\log n$, which are \emph{not} uniformly distributed $\bmod~1$.
On the other hand, even for functions which are uniformly distributed $\bmod~1$, such as $f(n)=\sqrt{n}$, by using Riesz means we are able to obtain a variant of well-distribution (see \cref{def_relative-u.d.} and Theorems \ref{thm_tempereduniformalongW} and \ref{thm_MainNilmanifoldEquidistribution}), which is not discernible by Ces\`aro averages and is crucially needed in our proof of \cref{thm_multipleconsecutiverecurrence_0}.}

In \cref{sec_SingleRec} we prove \cref{thm_mainsingle} using equidistribution results from \cref{sec_u.d.-along-riesz-means}; an outline of the proof is presented at the beginning of the section. 

Sections \ref{sec_4} and \ref{sec_u} are devoted to the proof of \cref{thm_multipleconsecutiverecurrence_0}.
The method of the proof 
bears some similarities to the method used by Frantzikinakis in \cite{Frantzikinakis10}, but has several differences which are partly due to the fact that we are using Riesz means. 
The first step of the proof of \cref{thm_multipleconsecutiverecurrence_0}, executed in \cref{sec_CharFact}, is to show that the nilfactor is ``characteristic'' for the pertinent expressions.
We adapt a PET induction argument introduced in \cite{Bergelson87} together with results from \cite{Host_Kra05} to fit the framework of Riesz means.
The rest of \cref{sec_4} is dedicated to 
reducing \cref{thm_multipleconsecutiverecurrence_0} to the following equidistribution result on nilmanifolds, which is of independent interest.

\begin{Theorem}\label{thm_MainNilmanifoldEquidistribution}
Let $G$ be a connected and simply connected nilpotent Lie group, let $\Gamma\subset G$ be a co-compact discrete subgroup and let $X=G/\Gamma$.
Let $k,\ell\in\N$ with $k\leq\ell$, let $f\in\F_{\ell}$, let $W=\Delta^{\ell-1}f$, let $b_0,\dots,b_{k-1}\in G$ be commuting elements, let
$$Y=\overline{\Big\{b_0^{t_0}\cdots b_{k-1}^{t_{k-1}}\Gamma:t_0,\dots,t_{k-1}\in\R\Big\}}\subset X,$$
and let $\mu_Y$ be the normalized Haar measure on $Y$.
Then for every continuous function $H\in C(X)$ we have
$$\otherlim{UW}{n\to\infty}H\left(b_0^{f(n)}b_1^{\Delta f(n)}\cdots b_{k-1}^{\Delta^{k-1}f(n)}\Gamma\right)=\int_YH(y)\d\mu_Y(y).$$
\end{Theorem}

The proof of \cref{thm_MainNilmanifoldEquidistribution} is presented in \cref{sec_u} and
is based on a new induction procedure, loosely based on the scheme used by Green and Tao in \cite{Green_Tao12a}.
Some of the technical results which are used in the proof of \cref{thm_MainNilmanifoldEquidistribution} are collected in \cref{sec_appendix} .
\\

\paragraph{\textit{\textbf{Acknowledgments.}}}
We are grateful to the referee for their helpful comments and suggestions.

\section{Well-distribution with respect to Riesz means}
\label{sec_u.d.-along-riesz-means}

In this section we establish some facts about well-distribution of sequences in compact abelian groups with respect to Riesz means. These results will be utilized in the sequel.

\subsection{Weyl's criterion and van der Corput's lemma for Riesz means}

\begin{Definition}[{cf. \cite{Drmota_Tichy88,Schatte88}}]
\label{def_relative-u.d.}
Let $G$ be a compact abelian group, let $H$ be a closed subgroup of $G$ and let $\mu_H$ denote the (normalized) Haar measure on $H$.
Let $W\in\F_1$ and let $x\colon\N\to G$.
We say that $x(n)$ is \emph{$\mu_H$-well distributed ($\mu_H$-w.d.) with respect to \textlim{W}} if for every continuous function $F\in C(G)$,
$$
\otherlim{UW}{} F\big(x(n)\big) = \lim_{W(N)-W(M)\to\infty}\frac1{W(N)-W(M)}\sum_{n=M}^N\Delta W(n)F\big(x(n)\big)=\int_G F\d\mu_H.
$$
If $H=G$ we simply say that $x(n)$ is \define{well distributed (w.d.) with respect to \textlim{W}}.
\end{Definition}


We remark that well-distribution with respect to \textlim{W} is related to $W$-syndetic sets in the same way as the conventional notion of well-distribution connects to syndetic sets.


Here are two results which illustrate \cref{def_relative-u.d.}.
They both follow directly from \cref{cor_productequidistribution} below.

\begin{Example}
Let $W(n)=\log{n}$.
Then the sequence $x(n)=(\log{n}\bmod1,\log(n+1)\bmod1)\in\T^2$ is $\mu_H$-w.d.\ with respect to \textlim{W} where $H=\{(x,x):x\in\T\}$ is the diagonal on $\T^2$.
Observe that $x(n)$ is not $\mu_H$-u.d.\ (a sequence $y(n)$ is $\mu_H$-u.d.\ if for any $f\in C(\T^2)$ we have $\lim_{N\to\infty}\frac1N\sum_{n=1}^Nf(y(n))=\int_H f\d\mu_H$).
\end{Example}

\begin{Example}
Let $H:=\{(x,y,2y-x):x,y\in\T\}\subset\T^3$ and consider the sequence $x(n)=\big(n^{3/2}\bmod1,(n+1)^{3/2}\bmod1,(n+2)^{3/2}\bmod1\big)\in\T^3$. Since $(n+2)^{3/2}= 2(n+1)^{3/2}-n^{3/2}+\oh_{n\to\infty}(1)$, the sequence $x(n)$ is $\mu_H$-u.d.\ with respect to Ces\`aro means. But $x(n)$ is not $\mu_H$-w.d.\ because if $F\colon H\to \C$ denotes the function $F(x,y,2y-x)=e^{2\pi i (x-y)}$ for all $x,y\in\T$ then it is not hard to find $M,N\in\N$ with $N-M$ arbitrarily large such that $\frac{1}{N-M}\sum_{n=M}^N F(x(n))$ is not equal to $\int_H F\d\mu_H=0$.
However, if $W(n)=\sqrt{n}$ then $x(n)$ is $\mu_H$-w.d.\ with respect to \textlim{W}.
\end{Example}

We remark that in both examples above the sequence $x(n)$ is not contained in the subgroup $H$.

We will make use of the following version of Weyl's criterion.

\begin{Proposition}[Weyl criterion]\label{prop_weylforRiesz}
Let $G$ be a compact abelian group, let $H$ be a closed subgroup of $G$, let $\widehat G$ denote the dual group of $G$ and let $\Gamma\coloneqq \{\chi\in\widehat{G}:\chi(y)=1~\text{for all}~y\in H\}$.
Let $W\in\F_1$ and let $x\colon\N\to G$.
Then $x$ is $\mu_H$-w.d.\ with respect \textlim{W} if and only if for every $\chi\in\widehat G$,
$$
\otherlim{UW}{n\to\infty} \chi(x(n)) =
\begin{cases}
1,&\text{if $\chi\in\Gamma$,}\\
0,&\text{if $\chi\notin\Gamma$.}
\end{cases}
$$
\end{Proposition}

\begin{proof}
Every continuous function on $G$ can be uniformly approximated by linear combinations of characters. So the proof follows from the fact that for each character $\chi$ we have $\int\chi\d\mu_H=1$ if $\chi\in\Gamma$ and $\int\chi\d\mu_H=0$ if $\chi\notin\Gamma$.
\end{proof}

We also need a version of van der Corput's lemma.
The proof is similar to that of \cite[Theorem 2.2]{Bergelson96}.

\begin{Proposition}[van der Corput]\label{thm_vdCforRiesz}
Let $\Hilb$ be a Hilbert space, let $W\in\F_1$,
and let $(N_t)_{t\in\N}$ and $(M_t)_{t\in\N}$ be sequences of positive integers with $W(N_t)-W(M_t)\to\infty$ as $t\to\infty$.
Let $f_1,f_2,\ldots \colon \N\to \Hilb$ satisfy $\sup_{t,n\in\N}\|f_t(n)\|<\infty$ and assume that for all but finitely many $d\in\N$ we have
$$\lim_{t\to\infty}\E_{n\in[M_t,N_t]}^W \big\langle f_t(n+d),f_t(n)\big\rangle=0.$$
Then
$$\lim_{t\to\infty}\E_{n\in[M_t,N_t]}^W f_t(n)=0\qquad\text{ in norm}.$$
\end{Proposition}

\begin{proof}
For each $d\in\N$ we have $\E_{n\in[M_t,N_t]}^W f_t(n)=\E_{n\in[M_t,N_t]}^W f_t(n+d)+\oh_{d;t\to\infty}(1)$.\footnote{Here and throughout the paper we write $\oh_{d;t\to\infty}(1)$ to denote a quantity that tends to $0$ as $t\to\infty$ and $d$ is kept fixed.} Averaging over $d$ and then using Jensen's inequality, we obtain
\begin{eqnarray*}
  \left\|\E_{n\in[M_t,N_t]}^W f_t(n)\right\|^2
  &=&
  \left\|\E_{d\in[1,D]}\E_{n\in[M_t,N_t]}^W f_t(n+d)\right\|^2+\oh_{D;t\to\infty}(1)
  \\&\leq&
  \E_{n\in[M_t,N_t]}^W\left\|\E_{d\in[1,D]} f_t(n+d)\right\|^2+\oh_{D;t\to\infty}(1)
  \\&=&
  \E_{n\in[M_t,N_t]}^W\E_{d,d'\in[1,D]}\langle f_t(n+d),f_t(n+d')\rangle+\oh_{D;t\to\infty}(1)
  \\&=&
  \E_{d,d'\in[1,D]}\E_{n\in[M_t,N_t]}^W\langle f_t(n+d-d'),f_t(n)\rangle+\oh_{D;t\to\infty}(1)
\end{eqnarray*}
When $d=d'$ the inner average equals $\|f_t(n)\|^2$ which is bounded by a constant.
When $d\neq d'$ the assumption implies that the inner average is $\oh_{D;t\to\infty}(1)$.
Therefore $\|\E_{n\in[M_t,N_t]}^W f_t(n)\|^2\leq\frac1D+\oh_{D;t\to\infty}(1)$. Letting $t\to\infty$ and then $D\to\infty$ we obtain the desired conclusion.
\end{proof}

Combining \cref{thm_vdCforRiesz} with \cref{prop_weylforRiesz} we deduce the following corollary.

\begin{Corollary}\label{cor_vdCalongW}
Let $G$ be a compact abelian group, let $u\colon\N\to G$ and let $W\in\F_1$.
  If for every $h\in\N$, the sequence $n\mapsto u(n+h)-u(n)$ is w.d.\ with respect to \textlim{W}, then the sequence $n\mapsto u(n)$ is w.d.\ with respect to \textlim{W} as well.
\end{Corollary}

\subsection{Well-distribution of sequences in \texorpdfstring{$\F$}{} with respect to Riesz means}

The following result is very similar to \cite[Theorem 1.7.16]{Kuipers_Niederreiter74}, but we don't require any smoothness and, additionally, obtain well-distribution.

\begin{Lemma}\label{lemma_unifdistrfejertorus}
  Let $\pi\colon \R\to\T\coloneqq \R/\Z$ be the canonical quotient map and let $W\in\F_1$.
  Then the sequence $\pi\big(W(n)\big)$ is w.d.\ with respect to \textlim{W}.
\end{Lemma}

\begin{proof}
When convenient, we identify $\T$ with the interval $[0,1)$ in the standard way.
Since finite linear combinations of indicator functions $1_{[0,x)}$ of intervals $[0,x)\subset\T$ with $x\in[0,1]$ are uniformly dense in $C(\T)$, it suffices to show that $$\otherlim{UW}{n\to\infty}1_{[0,x)}\Big(\pi\big(W(n)\big)\Big)=x\qquad\text{ for every }x\in[0,1].$$
Define
$A_{M,N}\coloneqq \E^W_{n\in[M,N]}1_{[0,x)}(\pi(W(n)))$.
Our goal is to show that $\lim_{W(N)-W(M)\to\infty}A_{M,N}=x$.

For each $m\in\Z$, let $S(m)$ be the smallest integer such that
$W(S(m))>m$.
Given $M,N\in\N$, let $w_M$ and $w_N$ be the largest integers for which $S(w_M)\leq M$ and $S(w_N)\leq N$ and let $s_M=S(w_M)$ and $s_N=S(w_N)$.
It suffices to verify the following claims:

\paragraph{\textsc{Claim 1:}}

$A_{s_M,s_N}=x+\oh_{W(N)-W(M)\to\infty}(1)$;

\paragraph{\textsc{Claim 2:}}

$\displaystyle|A_{M,N}-A_{s_M,s_N}|=\oh_{W(N)-W(M)\to\infty}(1)$.

\paragraph{\textit{Proof of Claim 1:}}
Fix $\epsilon>0$ and choose
$m_0$ such that $\Delta W(n)<\epsilon$ for all $n>S(m_0)$.
Let $c(m)$ be the smallest integer in the interval
$[S(m),S(m+1)]$ such that $W(c(m))\geq m+x$.
A simple calculation shows that
$$
A_{s_M,s_N}=\frac{1}{w_M-w_N}\sum_{m=s_M}^{s_N}\big[W(c(m))-W(S(m))\big]=
x+\Oh\left(\epsilon+\frac{m_0}{w_N-w_M}\right).
$$
This finishes the proof of Claim 1.

\paragraph{\textit{Proof of Claim 2:}}
Observe that $w_M\leq W(M)<w_M+1$ and $w_N\leq W(N)<w_N+1$ and hence (assuming $W(N)-W(M)>2$) we have $s_M\leq M\leq s_N\leq N$.
Therefore
\begin{eqnarray*}
|A_{M,N}-A_{s_M,s_N}|
&=&\left|
\E^W_{n\in[M,N]} 1_{[0,x)}\Big(\pi\big(W(n)\big)\Big)-\E^W_{n\in[s_M,s_N]} 1_{[0,x)}\Big(\pi\big(W(n)\big)\Big)\right|
\\
&\leq &
\left|\frac{1}{w_N-w_M}-\frac{1}{W(N)-W(M)} \right|
 \big(w_N-W(M)\big)
 \\
&&\qquad\qquad
+~\frac{W(M)-w_M}{w_N-w_M}+\frac{W(N)-w_N}{W(N)-W(M)}
\\
&\leq&
\frac{4}{W(N)-W(M)-2}.
\end{eqnarray*}
\end{proof}

The following corollary follows immediately from \cref{lemma_unifdistrfejertorus} and the definition of \textlim{W}.
\begin{Corollary}\label{cor_udtimesalpha}
  Let $W\in\F_1$ and $\alpha\in\R\setminus\{0\}$.
  Then the sequence $\alpha W(n)\bmod 1$ is w.d.\ in $[0,1)$ with respect to \textlim{W}.
\end{Corollary}
Here is another corollary which will be used later.

\begin{Corollary}\label{cor_zerodensity}
Let $W\in\F_1$ and let $f:\N\to\R$ be such that $f(n)\to0$ as $n\to\infty$.
Then the set $A:=\{n\in\N:\lfloor W(n)+f(n)\rfloor\neq\lfloor W(n)\rfloor\}$ has zero \define{Banach $W$-density}\footnote{Given a set $E\subset\N$ its \define{Banach $W$-density} is defined as $d^*_W(E)\coloneqq \otherlim{UW}{n\to\infty}\1_E(n)$ whenever this limit exists.}.
\end{Corollary}
\begin{proof}
  For each $\epsilon>0$ let $A_\epsilon=\{n\in\N:\|W(n)\|_\T<\epsilon\}$.\footnote{Here, and elsewhere in the paper, we use the notation $\|x\|_{\T}$ to denote the distance from $x\in\R$ to the closest point in $\Z$.} It follows from \cref{lemma_unifdistrfejertorus} that $d^*_W(A_\epsilon)=2\epsilon$.
  Since $A_\epsilon$ contains $A$, up to possibly finitely many exceptions, it follows that $d^*_W(A)<2\epsilon$, and hence taking $\epsilon\to0$ we obtain $d^*_W(A)=0$ as desired.
\end{proof}

The next lemma shows that the property of being w.d.\ is stable under small perturbations.
\begin{Lemma}\label{lemma_stillequidistributed}
Let $G$ be a compact abelian group, let $W\in\F_1$ and let $u\colon \N\to G$ be a sequence that is w.d.\ with respect to \textlim{W}.
Then for every $g\colon \N\to G$ such that $\lim_{n\to\infty} g(n)=0$, the sequence $u(n)+g(n)$ is w.d.\ with respect to \textlim{W}.
\end{Lemma}
\begin{proof}
  Let $\chi\colon G\to\C$ be a non-trivial character.
  We need to show that \textlim{UW}$\chi\big(u(n)+g(n)\big)=0$.
Since $\chi$ is a continuous homomorphism it satisfies $\chi\big(u(n)+g(n)\big)=\chi(u(n))\chi(g(n))$ and $\chi(g(n))=1+o_{n\to\infty}(1)$, so $\chi\big(u(n)+g(n)\big)-\chi(u(n))=o_{n\to\infty}(1)$.
This implies that $\otherlim{UW}{}\chi\big(u(n)+g(n)\big)=\otherlim{UW}{}\chi\big(u(n)\big)=0$.
\end{proof}

The following classical identity will be used often throughout this paper.

\begin{Lemma}
\label{lem_TaylorLeibnitz}
  For every function $f\colon\N\to\C$ and $h\in\N$ we have
  \begin{equation}\label{eq_TaylorLeibnitz}
  f(n+h)=\big(\Delta+1)^hf(n)=\sum_{i=0}^h\binom hi\Delta^if(n).
  \end{equation}
\end{Lemma}

\begin{Theorem}\label{thm_tempereduniformalongW}
  Let $\pi\colon \R\to\T$ be the canonical quotient map, let $\ell\geq0$ be an integer, let $f\in\F_{\ell+1}$ and let $W=\Delta^\ell f$.
  Then for any coefficients $c_0,\dots,c_\ell\in\R$, not all $0$, and any $g\colon \N\to\R$ which satisfies $\lim g(n)=0$, the function $F\coloneqq \pi\circ\big(c_0f+c_1\Delta f+\cdots+c_\ell\Delta^\ell f+g\big)$ is w.d.\ with respect to \textlim{W}.
\end{Theorem}

\begin{proof}
We prove this by induction on $\ell$.
For $\ell=0$ the conclusion follows directly from \cref{cor_udtimesalpha} and \cref{lemma_stillequidistributed}.
  Assume now that $\ell\geq1$ and the result has already been established for all smaller $\ell$.

  Let $F$ be as in the statement of the theorem.
  If $c_0=0$ the result follows by induction, hence, let us assume $c_0\neq0$.
  We will use \cref{cor_vdCalongW} to show that $F$ is w.d.\ with respect to the Riesz mean \textlim{W}.
  Hence it suffices to show that for any $h\in\N$ the sequence $n\mapsto F(n+h)-F(n)$ is w.d.\ with respect to \textlim{W}.
  In view of \cref{lem_TaylorLeibnitz} we have
  $$F(n+h)=\sum_{j=0}^\ell c_j\Delta^jf(n+h)+g(n+h)=\sum_{j=0}^\ell c_j\sum_{i=0}^h\binom hi\Delta^{j+i}f(n)+g(n+h)$$
  so $F(n+h)-F(n)=b_1\Delta f+\cdots+b_\ell\Delta^\ell f+o(1)$, where
  $$b_k=\sum_{i=0}^k c_{k-i}\binom hi-c_k=\sum_{i=1}^k c_{k-i}\binom hi.$$
  In particular $b_1=hc_0\neq0$.
  Since $\Delta f\in\F_\ell $ and $\Delta^{\ell -1}(\Delta f)=W$ we deduce from the induction hypothesis that indeed $F(n+h)-F(n)$ is w.d.\ with respect to \textlim{W}.
  \end{proof}

In view of the Weyl criterion, we can extrapolate the above result to well-distribution of sequences in higher-dimensional tori.

\begin{Lemma}\label{cor_productequidistribution}
Let $\ell\geq0$, let $f\in\F_{\ell+1}$ and let $W=\Delta^\ell f$.
Let $d\in\N$, let $\pi\colon \R^d\to\T^d$ be the canonical quotient map and, for each $i=0,\dots,\ell $, let $\alpha_{i}\in\R^d$.
Define the sequence $F\colon \N\to\R^{(\ell +1)d}$ by
$$F(n)=\big(\alpha_{i}\Delta^if(n)\big)_{i=0,\dots,\ell },$$
and consider the subgroup
$$H\coloneqq \bigotimes_{i=0}^\ell \overline{\big\{\pi(t\alpha_i):t\in\R\big\}}\subset\T^{(\ell +1)d}$$
with corresponding normalized Haar measure $\mu_H$.
Then the sequence $\pi\big(F(n)\big)$ is $\mu_H$-w.d.\ with respect to \textlim{W}.
\end{Lemma}

\begin{proof}
Let $\Gamma$ be the set of characters of $\T^{(\ell+1)d}$ that become trivial when restricted to $H$.
Given a character $\chi$ of $\T^{(\ell+1)d}$, we can find $\tau_0,\dots,\tau_\ell\in\Z^d$ such that $\chi\big(\pi(x_0,\dots,x_\ell)\big)=e\big(\langle x_0,\tau_0\rangle+\cdots +\langle x_\ell,\tau_\ell\rangle\big)$.
Then $\chi\in\Gamma$ if and only if $\langle\alpha_i,\tau_i\rangle=0$ for all $i=0,\dots,\ell$.
In view of \cref{thm_tempereduniformalongW} we have
$$\otherlim{UW}{n\to\infty}\chi\Big(\pi\big(F(n)\big)\Big)=
\otherlim{UW}{n\to\infty} e\left(\sum_{i=0}^\ell\langle\alpha_i,\tau_i\rangle\Delta^if(n)\right)
=1_\Gamma(\chi),$$
where $e(\theta)\coloneqq e^{2\pi i\theta}$,
and hence the Weyl criterion (\cref{prop_weylforRiesz}) implies that the sequence $\pi\big(F(n)\big)$ is $\mu_H$-w.d.\ with respect to \textlim{W}.
\end{proof}

Since functions $f\in\F$ are real valued, we consider the associated sequences $g(n)\coloneqq \lfloor f(n)\rfloor$.
Given a real number $x\in\R$ we denote by $\{x\}\coloneqq x-\lfloor x\rfloor$ the fractional part of $x$.
Since $\Delta$ is a linear operator we have that $\Delta^i g(n)=\Delta^i f(n)-\Delta^i\{f(n)\}$.
\begin{Lemma}\label{lemma_Deltag}
Let $f\in\F$, let $g(n)=\lfloor f(n)\rfloor$ and let $h\in\N$.
Then
\begin{equation}\label{eq_proof_lemma_Deltag}
\Delta^hg(n)=\Delta^hf(n)-\sum_{t=0}^h \binom{h}{t} (-1)^{h-t}\left\{\sum_{s=0}^t \binom{t}{s} \Delta^s f(n)\right\}
\end{equation}
\end{Lemma}

\begin{proof}
The following well known identity, which is similar to \eqref{eq_TaylorLeibnitz}, can be derived with the help of Newton's binomial formula:
\begin{equation}\label{eq_Leibnitz}
\Delta^h \tilde f(n)=\sum_{i=0}^h(-1)^{h-i}\binom h{i} \tilde f(n+i).
\end{equation}
Equation \eqref{eq_proof_lemma_Deltag} follows directly by applying \eqref{eq_Leibnitz} to the function $\tilde f(n)=\big\{f(n)\big\}$ and then using \cref{lem_TaylorLeibnitz}.
\end{proof}

We record for future reference the following inequality, which follows from \cref{lemma_Deltag}:
\begin{equation}\label{eq_floorderivativedistance}
\forall~h,n\in\N\qquad\qquad  \big|\Delta^hg(n)-\Delta^hf(n)\big|<2^h.
\end{equation}

\begin{Theorem}
\label{cor_tempfloorudalongWmultidimhigherk}
Let $\ell,d\in\N$, let $f\in\F_{\ell+1}$, let $W\coloneqq \Delta^\ell f$, let $g(n)\coloneqq \lfloor f(n)\rfloor$ and let $\pi\colon \R^{(d+1)(\ell+1)}\to\T^{(d+1)(\ell+1)}$ be the canonical projection.
Then, for all $\alpha\in\R^d$, the sequence
$$
G(n)\coloneqq \pi\Big(\big(\alpha \Delta^i g(n),\Delta^i f(n)\big)_{i=0,\dots,\ell }\Big)
$$
is $\mu_{(K\times\T)^{\ell+1}}$-w.d.\ with respect to \textlim{W}, where $K\coloneqq \overline{\{n\alpha\bmod \Z^d: n\in\N\}}\subset\T^d$.
\end{Theorem}

\begin{proof}
Let $\Gamma$ be the set of characters of $\T^{(d+1)(\ell+1)}$ which become trivial when restricted to $(K\times\T^d)^{\ell+1}$.
Let $\chi$ be a character of $\T^{(d+1)(\ell+1)}$.
For each $i=0,\dots,\ell$ there exist $\tau_i\in\Z^d$ and $h_i\in\Z$ such that
$$\chi\Big(\pi\big((x_i,y_i)_{i=0,\dots,\ell}\big)\Big)= e\left(\sum_{i=0}^\ell\langle x_i,\tau_i\rangle+h_iy_i\right).$$
Then $\chi\in\Gamma$ if and only if $\langle\alpha, \tau_i\rangle\in\Z$ and $h_i=0$ for every $i=0,\dots,\ell$.
In particular, if $\chi\in\Gamma$ then $\chi\big(G(n)\big)=1$ for every $n\in\N$.
In view of the Weyl criterion (\cref{prop_weylforRiesz}) it thus suffices to show that given $(\tau_i,h_i)\in\Z^{d+1}$, $i=0,\dots,\ell $, if there exists $i\in\{0,\ldots,\ell \}$ such that either $\alpha_i\coloneqq \langle \tau_i,\alpha\rangle\notin\Z$ or $h_i\neq0$, then one has
\begin{equation}\label{eq_proof_udgroup1-0}
\otherlim{UW}{n\to\infty}
e\left(\sum_{i=0}^{\ell }\alpha_i\Delta^i g(n)+h_i\Delta^if(n)\right)=0.
\end{equation}
Using \cref{lemma_Deltag} we can rewrite \eqref{eq_proof_udgroup1-0} as
\begin{equation}\label{eq_proof_udgroup2-0}
\otherlim{UW}{n\to\infty}
\psi\big(F(n)\big)=0
\end{equation}
where $F\colon \N\to\R^{2(\ell +1)}$ is the sequence
$$
F(n)\coloneqq \left(\Delta^if(n),~
\alpha_i\Delta^if(n)\right)_{ i=0,\dots,\ell }
$$
and $\psi\colon \R^{2(\ell +1)}\to\C$ is the map
\begin{eqnarray*}
  \psi(x_0,\dots,x_\ell ,y_0,\dots,y_\ell )&\coloneqq &e\left(\sum_{i=0}^{\ell }y_i+h_ix_i-\alpha_i\sum_{t=0}^i \binom{i}{t} (-1)^{i-t}\left\{\sum_{s=0}^t \binom{t}{s} x_s\right\}\right)\\&=&e\left(\sum_{i=0}^{\ell }y_i+h_ix_i\right)e\left(-\sum_{t=0}^\ell \left\{\sum_{s=0}^t \binom{t}{s} x_s\right\}\sum_{i=t}^\ell \alpha_i \binom{i}{t} (-1)^{i-t}\right).
\end{eqnarray*}

In view of \cref{cor_productequidistribution}, the sequence $F(n)\bmod\Z^{2(\ell +1)}$ is $\mu_H$-w.d.\ with respect to \textlim{W}, where
$$
H=\bigotimes_{i=0}^\ell  H_i\subset\T^{2(\ell +1)}
$$
and $H_i\coloneqq \overline{\{(t,\alpha_it):t\in\R\}}\subset \T^{2}$.

Since $\psi$ is Riemann integrable and periodic modulo $\Z^{2(\ell +1)}$, \eqref{eq_proof_udgroup2-0} will follow if we show that the integral $\int_H\psi\d\mu_H$ of $\psi$ over $H$ (with respect to the Haar measure $\mu_H$ on $H$) equals $0$.
For convenience, let $\psi_{\ell +1}\equiv1$ and for each $i=0,\dots,\ell $, define, recursively
\begin{equation*}
\begin{split}
&\psi_i(x_0,\dots,x_{i-1})\coloneqq \\
&\qquad\int_{H_i}e\left(y_i+h_ix_i-\left\{\sum_{s=0}^i\binom isx_s\right\}\sum_{j=i}^\ell \alpha_j\binom ji(-1)^{j-i}\right)\cdot\psi_{i+1}(x_0,\dots,x_i)\d\mu_{H_i}(x_i,y_i).
\end{split}
\end{equation*}
Notice that $\psi_0=\int_H\psi\d\mu_H$ (and in particular $\psi_0$ is a constant).

To show that $\psi_0=0$ we distinguish two cases: the case where some $\alpha_j$ is irrational, and the case where all $\alpha_i$ are rational for $i\in\{0,\ldots,\ell \}$.

For the first case, let $j\in\{0,\ldots,\ell \}$ be such that $\alpha_j$ is irrational.
In this case $H_j=\T^2$ and hence $\d\mu_{H_j}(x_j,y_j)=\d x_j\d y_j$.
Therefore, for any $x_0,\dots,x_{j-1}\in\T$,
\begin{equation*}
\begin{split}
&\psi_j(x_0,\dots,x_{j-1}) \\
&\qquad
=\int_{\T}e(y_j)\d y_j\cdot\int_\T e\left(h_ix_i-\left\{\sum_{s=0}^i\binom isx_s\right\}\sum_{j=i}^\ell \alpha_j\binom ji(-1)^{j-i}\right)\cdot\psi_{i+1}(x_0,\dots,x_i)\d x_j
\\
&\qquad=0.
\end{split}
\end{equation*}
Therefore, for every $i\leq j$ we have $\psi_i\equiv0$ and in particular $\psi_0=0$.

Next we treat the case when all $\alpha_i$ are rational.
Write $\alpha_{i}=a_{i}/b_{i}$ with $a_{i}\in\Z$ and $b_{i}\in\N$ coprime.
Then $H_{i}$ can be parameterized as $H_{i}=\{(x,\alpha_{i} x):x\in[0,b_{i})\}$ and, more generally, as $H_{i}=\{(x,\alpha_{i} x):x\in [c,c+b_i]\}$ for any $c\in\R$.
Let $c=c(x_0,\dots,x_{i-1})=-\sum_{s=0}^{i-1}\binom isx_s$ and note that

\begin{equation*}
\begin{split}
\psi_i(x_0,\dots & ,x_{i-1})
\\
&=
\int_c^{c+b_i}e\left((\alpha_i+h_i)x_i-\left\{\sum_{s=0}^i\binom isx_s\right\}\sum_{j=i}^\ell \alpha_j\binom ji(-1)^{j-i}\right)\cdot\psi_{i+1}(x_0,\dots,x_i)\d x_i\\
&=
\sum_{t=0}^{b_i-1}\int_{c+t}^{c+t+1}e\left((\alpha_i+h_i)x_i-\left\{x_i-c\right\}\sum_{j=i}^\ell \alpha_j\binom ji(-1)^{j-i}\right)\cdot\psi_{i+1}(x_0,\dots,x_i)\d x_i\\
&=
\sum_{t=0}^{b_i-1}\int_{c+t}^{c+t+1}e\left((\alpha_i+h_i)x_i-\left(x_i-c-t\right)\sum_{j=i}^\ell \alpha_j\binom ji(-1)^{j-i}\right)\cdot\psi_{i+1}(x_0,\dots,x_i)\d x_i.
\end{split}
\end{equation*}

Let $i$ be the largest index for which either $\alpha_i\notin\Z$ or $h_i\neq0$.
In order to finish the proof of \cref{cor_tempfloorudalongWmultidimhigherk} we need to following lemma:

\begin{Lemma}\label{lemma_proof_cor_tempfloorudalongWmultidimhigherk}
For every $r=i+1,\dots,\ell $
  $$\psi_r(x_0,\dots,x_{r-1})=e\left(-\sum_{s=0}^{r-1}x_s\sum_{j=r}^\ell a_j\binom js\sum_{k=r}^j\binom{j-s}{k-s}(-1)^{j-k}\right).$$
\end{Lemma}
\begin{proof}[Proof of \cref{lemma_proof_cor_tempfloorudalongWmultidimhigherk}]
 We will use backward induction on $r$.
  For $r=\ell+1$ the result is trivial. Assume now that we have established this for $r+1$.

  For each $r=i+1,\dots,\ell$ we have that $\alpha_r=a_r$, $b_r=1$ and $h_r=0$.
  Therefore
\begin{equation*}
\begin{split}
\psi_r=\int_{c}^{c+1}e\Bigg(a_rx_r-(x_r-c)\sum_{j=r}^\ell & a_j\binom jr(-1)^{j-r}
\\
& -\sum_{s=0}^rx_s\sum_{j={r+1}}^\ell a_j\binom j s\sum_{k=r+1}^j\binom {j-s} {k-s}(-1)^{j-k}\Bigg)\d x_r
\end{split}
\end{equation*}
 Observe that the coefficient of $x_r$ inside the exponential is $0$, so the integrand is in fact a constant function and hence
 \begin{eqnarray*}
   \psi_r &=& e\left(c\sum_{j=r}^\ell a_j\binom jr(-1)^{j-r}-\sum_{s=0}^{r-1}x_s\sum_{j={r+1}}^\ell a_j\binom j s\sum_{k=r+1}^j\binom {j-s} {k-s}(-1)^{j-k}\right) \\
    &=&  e\left(-\sum_{s=0}^{r-1}\binom rsx_s \sum_{j=r+1}^\ell a_j\binom jr(-1)^{j-r}-\sum_{s=0}^{r-1}x_s\sum_{j={r+1}}^\ell a_j\binom j s\sum_{k=r+1}^j\binom {j-s} {k-s}(-1)^{j-k}\right)
 \end{eqnarray*}
 Using the identity $\binom rs\binom jr=\binom js\binom{j-s}{r-s}$ we conclude that
\begin{eqnarray*}
   \psi_r &=&  e\left(-\sum_{s=0}^{r-1}x_s \sum_{j=r+1}^\ell a_j\binom js\binom{j-s}{r-s}(-1)^{j-r}-\sum_{s=0}^{r-1}x_s\sum_{j={r+1}}^\ell a_j\binom j s\sum_{k=r+1}^j\binom {j-s} {k-s}(-1)^{j-k}\right)
   \\&=&
   e\left(-\sum_{s=0}^{r-1}x_s\sum_{j=r}^\ell a_j\binom js\sum_{k=r}^j\binom{j-s}{k-s}(-1)^{j-k}\right)
 \end{eqnarray*}
 as desired.
\end{proof}

\paragraph{\textit{Continuation of the proof of \cref{cor_tempfloorudalongWmultidimhigherk}:}}
\cref{lemma_proof_cor_tempfloorudalongWmultidimhigherk} gives us a rather explicit (albeit cumbersome) expression for $\psi_i$.
Recall that $i$ was chosen as the largest index so that either $\alpha_i\notin\Z$ or $h_i\neq0$.
We further divide into two cases: $h_i=0$ (in which case $\alpha_i\notin\Z$) and $h_i\neq0$.

If $h_i=0$ then a quick computation shows that the integrand
$$e\left((\alpha_i+h_i)x_i-\left(x_i-c-t\right)\sum_{j=i}^\ell \alpha_j\binom ji(-1)^{j-i}\right)\cdot\psi_{i+1}(x_0,\dots,x_i)$$
does not depend on $x_i$, and hence we have
$$\psi_i=\sum_{t=0}^{b_i-1}e\left((c+t)\sum_{j=i}^\ell \alpha_j\binom ji(-1)^{j-i}-\sum_{s=0}^{i-1}x_s\sum_{j=i+1}^\ell a_j\binom js\sum_{k=i+1}^j\binom{j-s}{k-s}(-1)^{j-k}\right)$$
Notice that after factoring out of the sum the terms which do not depend on $t$ we end up with a simple geometric sum.
Since every $\alpha_r$ for $r>i$ is an integer (unless of course $i=\ell$ in which case there is no $r>i$), we conclude
\begin{eqnarray*}
\psi_i &=&\text{constant}\cdot \sum_{t=0}^{b_i-1}e\left(t\sum_{j=i}^\ell \alpha_j\binom ji(-1)^{j-i}\right)
\\
&=&
\text{constant}\cdot\sum_{t=0}^{b_i-1}e\left(t\alpha_i\right)
\\
&=&
\text{constant}\cdot\sum_{t=0}^{b_i-1}e\left(\frac {a_i}{b_i}\right)^t
\\
&=&0.
\end{eqnarray*}

Finally we address the case when $h_i\neq0$.
In this case, for each $t=0,\dots,b_i-1$, the integrand
$$e\left((\alpha_i+h_i)x_i-\left(x_i-c-t\right)\sum_{j=i}^\ell \alpha_j\binom ji(-1)^{j-i}\right)\cdot\psi_{i+1}(x_0,\dots,x_i)$$
is a constant multiple of $e(h_ix_i)$, so when integrated over an interval of length $1$, it vanishes.
\end{proof}

The following corollary of \cref{cor_tempfloorudalongWmultidimhigherk} will be used in the proof of \cref{thm_W-uniform-vonNeumann-along-functions-in-F} below.
\begin{Corollary}\label{cor_floorwd}
  Let $\ell\in\N$, let $f\in\F_{\ell+1}$ and let $W=\Delta^\ell f$.
  Then for every $\alpha\in\R\setminus\Z$ the sequence
  $(\alpha\lfloor f(n)\rfloor)$ is $\mu_H$-w.d.\ with respect to \textlim{W}, where $\mu_H$ is the normalized Haar measure of the closed subgroup of $\T$ defined by $H\coloneqq \overline{\{n\alpha :n\in\N\}}$.
  In particular, for every $\alpha\in\R\setminus\Z$ we have
  $$\otherlim{UW}{n\to\infty}e\big(\alpha\lfloor f(n)\rfloor\big)=0.$$
\end{Corollary}


\section{Optimal single recurrence along sequences in $\F$}\label{sec_SingleRec}
In this section we give a proof of \cref{thm_mainsingle}.
Here is the statement again, for the convenience of the reader.

\begin{named}{\cref{thm_mainsingle}}{}
Let $\ell\in\N$, $f\in\F_{\ell}$ and define $W\coloneqq \Delta^{\ell-1} f$. Then for any invertible probability measure preserving system $(X,\B,\mu,T)$, any $A\in \B$ and any $\epsilon>0$ the set
$$
R_{\epsilon, A}= \big\{n\in\N:\mu(A\cap T^{-\lfloor f(n)\rfloor}A)>\mu^2(A)-\epsilon\big\}
$$
is thick and $W$-syndetic.
\end{named}

\paragraph{\textit{\textbf{Outline of the proof:}}}
In the proof of \cref{thm_mainsingle} we will utilize the
Jacobs-de Leeuw-Glicksberg decomposition:

\begin{Theorem}[{cf. \cite[\S 2.4 ]{Krengel85} or \cite[Theorem 2.3]{Bergelson96}}]
\label{thm_jlg}
Let $\Hilb$ be a Hilbert space and $U\colon \Hilb\to\Hilb$ a unitary operator. Then any $h\in \Hilb$ can be written as $h=h_{wm} + h_c $, where $h_c \perp h_{wm}$ and
\begin{itemize}
\item $h_{wm}$ is a \define{weakly mixing element}, that is, for every $h'\in \Hilb$ we have
$$
\lim_{N-M\to\infty}\frac{1}{N-M}\sum_{n=M}^{N-1} |\langle U^n h_{wm},h'\rangle |=0.
$$
\item
$h_c$ is a \define{compact element}, i.e., the closure of $\{U^n h_c: n\in\N\}$ is a compact subset of $\Hilb$.
\end{itemize}
\end{Theorem}

In light of \cref{thm_jlg} we can write $1_A= h_{wm}+ h_c$, where $h_{wm}\in L^2\xbm$ is weakly mixing and $h_c\in L^2\xbm$ is compact.

As for the component $h_{wm}$, it will be shown in the next subsection that for any $\ell\in\N$, any $h'\in L^2\xbm$ and any $\epsilon>0$ there exists $N\in\N$ such that $\frac{1}{N}\sum_{n=1}^N\big|\langle T^{p(n)} h_{wm},h'\rangle\big|<\epsilon$ uniformly over all polynomials $p\in\Z[x]$ of degree $\ell$ and having leading coefficient belonging to a set of lower Banach density equal to $1$.

In Subsection \ref{subsec-compact} we use the equidistribution results obtained in \cref{sec_u.d.-along-riesz-means} to show that for any $\epsilon>0$ there is a thick set $E\subset\N$ with the property that for any $n\in E$ the function $T^{\lfloor f(n)\rfloor}h_c$ is $\epsilon$-close to $h_c$ in $L^2$-norm and for all $n\in E$ the Taylor expansion of $k\mapsto f(n+k)$, truncated at an appropriate level, corresponds to a polynomial whose leading coefficient belongs to the set of lower Banach density equal to $1$.

Finally, in Subsection \ref{subsec-cwm}, these results are ``glued together'' to yield a proof of \cref{thm_mainsingle}.

\subsection{The weakly mixing component}

The following remark will be utilized in the proof of \cref{prop:vdc-banachdensity1} below.

\begin{Remark}\label{remark_weakmixing}Let $(X,{\mathcal B},\mu,T)$ be a measure preserving system and let $h\in L^2(X,\mathcal{B},\mu)$.
It is not hard to see from the definition that $h$ is weakly mixing if and only if for every $\epsilon>0$ and every $h'\in L^2(X)$, there exists a set $D\subset\N$ with $d_*(D)=1$ such that $|\langle T^dh,h'\rangle|<\epsilon$ for every $d\in D$. Here, $d_*$ denotes the lower Banach density, defined by the formula
$$d_*(D)=\lim_{N\to\infty}~\inf_{M\in\N}\frac{\big|D\cap\{M,M+1,\dots,M+N\}\big|}{N}.$$
\end{Remark}

\begin{Theorem}\label{prop:vdc-banachdensity1}
Let $\ell \in\N\cup\{0\}$, let $(X,{\mathcal B},\mu,T)$ be a measure preserving system and let $h\in L^2(X,\mathcal{B},\mu)$ be a weakly mixing function with $\|h\|_2\leq1$.
Then for any $\epsilon>0$, any $h'\in L^2(X,\mathcal{B},\mu)$ with $\|h'\|_2\leq1$ and sufficiently large $N\in\N$ there exists $D\subset \N$ with $d_*(D)=1$ such that for every polynomial $p\in\Z[x]$ of degree $\ell $ with $\Delta^\ell p\in D$, one has
$$
\frac{1}{N}\sum_{n=1}^N\big|\langle T^{p(n)} h,h'\rangle\big|~\leq~\epsilon.
$$
\end{Theorem}

(Observe that $\Delta^\ell p$ is a constant function whenever $p$ is a polynomial of degree $\ell$ and the (only) value it takes is $\ell!$ times the leading coefficient of $p$.)

\begin{proof}
For $\ell =0$ the conclusion follows directly from \cref{remark_weakmixing}.
We now proceed by induction on $\ell $, assuming that $\ell >0$ and that the result has been established for $\ell -1$.
Observe that
\begin{eqnarray*}
\big|\langle T^{p(n)} h,h'\rangle\big|^2&=&\int_XT^{p(n)}h\cdot h'\d\mu\int_XT^{p(n)}\bar h\cdot\bar h'\d\mu\\&=&\int_{X\times X}(T\times T)^{p(n)}(h\otimes \bar h)\cdot(h'\otimes\bar h')\d(\mu\otimes\mu).
\end{eqnarray*}
It follows easily from the definition that if $h$ is a weak mixing function then so is $h\otimes \bar h$.
Therefore, renaming $h\otimes\bar h$ as $h$ and $T\times T$ as $T$, it suffices to prove that
$$\left\|\frac1N\sum_{n=1}^NT^{p(n)}h\right\|<\epsilon.$$
To this end we will employ a version of the van der Corput inequality (cf. \cite[Lemma 1.3.1]{Kuipers_Niederreiter74}), which states that for every Hilbert space vectors $u_1,u_2,\dots$ with norm bounded by $1$, one has
\begin{equation}
\label{eq_vdCinequality}
\left\|\frac1N\sum_{n=1}^Nu_n\right\|
\leq
\frac2{\sqrt{N}}+ \sqrt{\frac1{\sqrt{N}}+
\frac2{\sqrt{N}}\sum_{1\leq m\leq\sqrt{N}}\left|\frac1{N-m}\sum_{n=1}^{N-m}\langle u_{n+m},u_n\rangle\right|}.
\end{equation}
Therefore, it suffices to show that for every $\delta>0$ and large enough $N\in\N$ there exists a set $D\subset\N$ with $d_*(D)=1$ such that for every polynomial $p\in\Z[x]$ of degree $\ell $ with $\Delta^\ell p$ in $D$ and all $m\in\{1,\dots,\lfloor\sqrt{N}\rfloor\}$, one has
\begin{equation}\label{eq_proof_prop_finitevdC}
\left|\frac1N\sum_{n=1}^N\Big\langle T^{p(n+m)}h,T^{p(n)}h\Big\rangle\right|<\delta.
\end{equation}
Observe that
\begin{eqnarray*}
\Big\langle T^{p(n+m)}h,T^{p(n)}h\Big\rangle&=&\int_XT^{p(n+m)}h\cdot T^{p(n)}\bar h\d\mu
\\
&=&\int_XT^{p(n+m)-p(n)}h\cdot\bar h\d\mu
\end{eqnarray*}
and $p_m\colon n\mapsto p(n+m)-p(n)$ is a polynomial of degree $\ell -1$ with $\Delta^{\ell-1}p_m=m\Delta^\ell p$.

By the induction hypothesis there exists a set $D_0\subset\N$ with $d_*(D_0)=1$ such that for any polynomial $q\in\Z[x]$ of degree $\ell -1$ with $\Delta^{\ell-1}q$ in $D_0$,
$$
\frac{1}{N}\sum_{n=1}^N\big|\langle T^{q(n)} h,\bar h\rangle\big|~\leq~\delta.
$$
Note now that the set $D_0/m \coloneqq \{a\in\N:am\in D_0\}$ satisfies $d_*(D_0/m)=1$ for any $m\in\N$.
Moreover, the collection of sets with lower Banach density $1$ has the finite intersection property. It follows that the set
$$D\coloneqq \bigcap_{m=1}^{\lfloor\sqrt{N}\rfloor}D_0/m$$
has $d_*(D)=1$.

Putting everything together, if $p\in\Z[x]$ is a polynomial of degree $\ell$ with $\Delta^\ell p$ in $D$, then the polynomial $p_m(n)=p(n+m)-p(n)$ has degree $\ell-1$ and $\Delta^{\ell-1}p_m$ is in $D_0$ for every $m\in\{1,\dots,\lfloor\sqrt{N}\rfloor\}$.
This implies that \eqref{eq_proof_prop_finitevdC} holds, which concludes the proof.
\end{proof}

\subsection{The compact component}
\label{subsec-compact}
To deal with the compact component in the proof of \cref{thm_mainsingle}, we will use the equidistribution results developed in \cref{sec_u.d.-along-riesz-means}.
We start with the following observation.
\begin{Lemma}\label{cor:equi-dis-1}
Let $\ell\geq0$, let $f\in\F_{\ell+1}$, let $g(n)=\lfloor f(n)\rfloor$ and let $W=\Delta^\ell  f$. For any $\alpha\in\R^d$ and any $\delta>0$ the set\footnote{Here, and elsewhere in the paper we use the notation $\|x\|_{\T^d}$ to denote the distance from the vector $x\in\R^d$ to the closest point in the lattice $\Z^d$.}
$$
D\coloneqq\Big\{n\in\N: \big\|\Delta^i g(n)\alpha\big\|_{\T^d}\leq\delta;~\big\{\Delta^i f(n)\big\}<\delta~\text{for all }i\in\{0,1,\ldots,\ell \}\Big\}
$$
satisfies $d_W(D)>0$.
\end{Lemma}

\begin{proof}
It follows from \cref{cor_tempfloorudalongWmultidimhigherk} that the sequence
$$
n\mapsto\left(g(n)\alpha , f(n),\Delta g(n)\alpha,\Delta f(n),\ldots,\Delta^\ell  g(n)\alpha,\Delta^\ell f(n)\right)
$$
is $\mu_H$-w.d.\ with respect to \textlim{W} for some closed subgroup $H$ of $\T^{(d+1)(\ell +1)}$.
Let $F_\delta\colon (\T^d\times\T)^{(\ell +1)} \to\{0,1\}$ denote the function
$$
F_\delta(x_0,y_0,\ldots,x_\ell ,y_\ell )=
\begin{cases}
1,&\text{if }\|x_i\|_{\T^d}\leq \delta\text{ and }\{y_i\}<\delta\text{ for all $0\leq i\leq \ell $.;}
\\
0,&\text{otherwise.}
\end{cases}
$$
Notice that while $F_\delta$ is not continuous, the sets of discontinuity for $F_\delta$ as $\delta$ changes are disjoint, and hence for all but at most countably many values of $\delta$, the function $F_\delta$ is almost everywhere continuous with respect to the Haar measure $\mu_H$ of $H$.
Replacing if necessary $F_\delta$ with $F_{\delta'}$ for some $\delta'<\delta$, and noticing that $F_{\delta'}\leq F_\delta$, we can then assume that $F_\delta$ is $\mu_H$-a.e. continuous.
We now use well-distribution to conclude that
\begin{eqnarray*}
d_W(D)
&=&
\lim_{N\to\infty}\frac{1}{W(N)}\sum_{n=1}^N \Delta W(n) F\left(g(n)\alpha , f(n),\ldots,\Delta^\ell  g(n)\alpha,\Delta^\ell f(n)\right)
\\
&=&\int_{H} F\d\mu_{H}>0.
\end{eqnarray*}
\end{proof}

 \begin{Lemma}\label{lem:w-density-to-upper-density}
         Let $f\in\F_{\ell +1}$, let $g(n)=\lfloor f(n)\rfloor$, let $W=\Delta^\ell  f$ and let $D\subset\N$ be such that $d_W(D)>0$. Define $B\coloneqq \Delta^\ell g(D)=\{\Delta^\ell g(n): n\in D\}$.
         Then $\bar d(B)>0$.
     \end{Lemma}

     \begin{proof}
Let $s\leq N$ be natural numbers.
Denote by $(\Delta^\ell g)^{-1}(s)$ the set $\{n:\Delta^\ell g(n)=s\}$.
Let $M=M(N)=\max\big\{m:\Delta^\ell g(m)\leq N\big\}$ and let
         $$P(s,N)\coloneqq \frac1{W(M)}\sum_{n\in(\Delta^\ell g)^{-1}(s)}\Delta W(n).$$
         From \eqref{eq_floorderivativedistance} it follows that for large enough $N$ we have $W(M)/N\geq1/2$.
         Moreover, \eqref{eq_floorderivativedistance} also implies that if $n\in(\Delta^\ell g)^{-1}(s)$ then $W(n)\in(s-2^\ell ,s+2^\ell )$.
         Let $W^{-1}(a)\coloneqq \min\{x\in\N:W(x)\geq a\}$.
         We have
         $$\sum_{m\in(\Delta g)^{-1}(s)}\Delta W(m)\leq\sum_{m=W^{-1}(s-2^\ell )}^{W^{-1}(s+2^\ell )-1}\Delta W(m)\leq 2^{\ell +1}$$
         and hence, for large enough $N$, we have

         \begin{eqnarray*}
           \frac{d_W(D)}2
           &\leq&
           \frac1{W(M)}\sum_{m=1}^M1_D(m)\Delta W(m)
           =\frac1{W(M)}\sum_{s\in B\cap[N]}\sum_{m\in(\Delta g)^{-1}(s)}1_D(m)\Delta W(m)
           \\
           &\leq&\sum_{s\in B\cap[N]}P(s,N)
           \leq\frac{2^{\ell +1}\big|B\cap[N]\big|}{W(M)}
           \leq\frac{2^{\ell +2}\big|B\cap[N]\big|}{N}
           \leq 2^{\ell +3}\bar d(B).
         \end{eqnarray*}
 \end{proof}

\begin{Lemma}\label{lemma_joeldoingfloriansjob}
  For every $\ell,N\in\N$ there exists $\delta>0$ such that if $f\in\F_{\ell+1}$ and $a\in\N$ satisfy
  $$\forall i=0,\dots,\ell\qquad\big\{\Delta^i f(a)\big\}<\delta\qquad\text{and}\qquad\Delta^{\ell+1}f(a)<\delta$$
  then $\Delta^\ell g$ is constant in the interval $[a,a+N]$, where $g(m)=\lfloor f(m)\rfloor$.
\end{Lemma}
\begin{proof}
  Since $\Delta^ig$ is always an integer it suffices to show that $\big|\Delta^\ell g(m+1)-\Delta^\ell g(m)\big|<1$ for every $m\in[a,a+N)$.
  Then we use \cref{lemma_Deltag} and the observation that $\{x+y\}\leq\{x\}+\{y\}$ to deduce that
  \begin{equation*}
  \begin{split}
  \big|  \Delta^\ell & g(m+1)-\Delta^\ell g(m)\big|
  \\
  &\leq\Delta^{\ell+1}f(m)+\sum_{t=0}^\ell\binom \ell t\sum_{s=0}^t\binom ts\Big(2\left\{\Delta^{s}f(m)\right\}+\left\{\Delta^{s+1}f(m)\right\}\Big) \\ &\leq  \Delta^{\ell+1}f(m)+3^{\ell+1}\max_{i\geq0}\big\{\Delta^if(m)\big\}.
  \end{split}
  \end{equation*}
  Finally notice that in view of \cref{lem_TaylorLeibnitz} we have
  $$\big\{\Delta^if(m)\big\}=\left\{\sum_{j=0}^{m-a}\binom{m-a}j\Delta^{i+j}f(a)\right\}\leq 2^N\delta$$
  so we just need to choose $\delta$ so that $(1+3^{\ell+1}2^N)\delta<1$.
\end{proof}

We can now prove the main theorem of this subsection.
\begin{Theorem}\label{lemma_floriansjob}
Let $f\in\F_{\ell+1}$, let $g(n)=\lfloor f(n)\rfloor$, let $N\in\N$, $\epsilon>0$ and let $F\subset\T$ be a finite set.
Then there exists a set $B\subset\N$ with positive upper Banach density and for every $s\in B$ there exists $a\in\N$ such that for every $n\in\{0,1,\dots,N\}$
            \begin{enumerate}
                \item\label{item:fj-1} $\|g(a+n)x\|_\T<\epsilon$ for all $x\in F$,
                \item\label{item:fj-2} $\Delta^\ell g(a+n)=s$.
            \end{enumerate}
\end{Theorem}

        \begin{proof}
        Let $N\in\N$, $\epsilon>0$ and $F=\{x_1,\ldots,x_d\}\subset\T$ be given and define $\alpha\coloneqq (x_1,\ldots,x_d)$. Let $\delta_0$ be given by \cref{lemma_joeldoingfloriansjob}, set $\delta=\min(\delta_0,\tfrac{\epsilon}{(\ell+1)N^\ell})$ and let $D\subset\N$ be defined as
        \begin{equation}
        \label{eq:set-D-1}
                D\coloneqq \Big\{a\in\N: \big\|\Delta^i g(a)\alpha\big\|_{\T^d}\leq\delta;~\big\{\Delta^i f(a)\big\}<\delta~\text{for all }i\in\{0,1,\ldots,\ell\}\Big\}.
        \end{equation}
        According to \cref{cor:equi-dis-1} the set $D$ has positive $d_W$ density (where $W=\Delta^k f$).
        By chopping off a finite subset of $D$ if necessary, we can assume that $\Delta^{\ell+1}f(a)<\delta$ for every $a\in D$ (keeping in mind that removing a finite set does not affect the density $d_W(D)$).
        Define $B\coloneqq \Delta^\ell g(D)$; from \cref{lem:w-density-to-upper-density}, it follows that $B$ has positive upper density.

        Next let $s\in B$ be arbitrary. Let $a\in D$ be such that $\Delta^\ell g(a)=s$.
        We claim that conditions \eqref{item:fj-1} and \eqref{item:fj-2} are satisfied.
        In view of \cref{lemma_joeldoingfloriansjob} we have that $\Delta^\ell g(m)=s$ for every $m\in[a,a+N]$, establishing condition \eqref{item:fj-2}.

         Observe that condition \eqref{item:fj-2} also implies that $\Delta^{\ell+1}g(m)=0$ on $[a,a+N]$. Using \cref{lem_TaylorLeibnitz}, we get that
\begin{equation*}
g(a+h)=\sum_{i=0}^h\binom{h}{i}\Delta^i g(a)=\sum_{i=0}^{\min(h,\ell)} \binom{h}{i}\Delta^i g(a).
\end{equation*}
Therefore, for all $n\in\{0,\dots,N\}$, we obtain
        \begin{eqnarray*}
        \max_{x\in F} \|g(a+n)x\|_\T
        &\leq & \|g(a+n)\alpha \|_{\T^d}
        \\
        &\leq & \sum_{i=0}^\ell \binom{n}{i}\left\|\Delta^i g(a)\alpha \right\|_{\T^d}
         \\
        &< & \sum_{i=0}^\ell N^\ell\left\|\Delta^i g(a)\alpha \right\|_{\T^d}.
        \end{eqnarray*}
Finally, using \eqref{eq:set-D-1} and the fact that $a\in D$, we deduce that
$$
\sum_{i=0}^\ell N^\ell\left\|\Delta^i g(a)\alpha \right\|_{\T^d}\leq \epsilon.
$$
This finishes the proof.
        \end{proof}

        \subsection{Proof of {\cref{thm_mainsingle}}}
        \label{subsec-cwm}
        We now combine the main results of the two previous subsections to prove \cref{thm_mainsingle}.
\begin{Lemma}\label{lemma_combined}
Let $f\in\F_{\ell+1}$, let $g(n)=\lfloor f(n)\rfloor$ and $F\subset\T$ be a finite set.
Let $(X,{\mathcal B},\mu,T)$ be a measure preserving system and let $h_{wm}\in L^2(X,\mathcal{B},\mu)$ be a weakly mixing function.
Then for any $\delta>0$ and any sufficiently large $N\in\N$ there exists $a\in \N$ such that
 \begin{enumerate}
                \item\label{item:fj-1-2} $\|g(a+n)x\|_\T<\delta$ for all $x\in F$ and $n\in\{0,1,\dots,N\}$,
                \item\label{item:fj-2-2} $
\frac{1}{N}\sum_{n=1}^N\big|\langle T^{g(a+n)} h_{wm},h_{wm}\rangle\big|~\leq~\delta.
$
            \end{enumerate}
        \end{Lemma}

\begin{proof}
First, we apply \cref{prop:vdc-banachdensity1}
to find a set $D\subset \N$ with $d_*(D)=1$ such that for every polynomial $p\in\Z[x]$ of degree $\ell$ with the constant $\Delta^\ell p$ in $D$, one has
$$
\frac{1}{N}\sum_{n=1}^N\big|\langle T^{p(n)} h_{wm},h_{wm}\rangle\big|~\leq~\delta.
$$
Then we apply \cref{lemma_floriansjob} to find a set $B\subset\N$ with positive upper Banach density such that for any $s\in B$ there exists $a\in\N$ such that for every $n\in\{0,1,\dots,N\}$ we have $\sup_{x\in F}\|g(a+n)x\|_\T<\delta$ and $\Delta^\ell g(a+n)=s$.

Note that $D\cap B\neq \emptyset$ and let $s$ be in this intersection.
Since $\Delta^\ell g(a+n)=s$ for all $n\in[0,N]$, we can use \cref{lem_TaylorLeibnitz} to deduce that for all $n$ in that interval $g(a+n)=p(n)$ for some polynomial $p$ of degree $\ell$ and with $\Delta^\ell p(n)=s$.
Therefore, $
\frac{1}{N}\sum_{n=1}^N\big|\langle T^{g(a+n)} h_{wm},h_{wm}\rangle\big|~\leq~\delta$.
\end{proof}

 We also need the following well known fact whose proof we include for completeness.
Recall that a F\o lner sequence in $\N$ is a sequence $(F_N)_{N\in\N}$ of finite sets such that $\big|(F_N+x)\cap F_N\big|/|F_N|\to1$ as $N\to\infty$.
Given a F\o lner sequence $(F_N)_{N\in\N}$, the upper density of a set $E\subset\N$ relative to $(F_N)_{N\in\N}$ is
$$\bar d_{(F_N)}(E)\coloneqq \limsup_{N\to\infty}\frac{|E\cap F_N|}{F_N}.$$

\begin{Lemma}\label{lemma_thickfolner}
    Let $x_n$ be a sequence of non-negative real numbers and let $L>0$.
    If there exists a F\o lner sequence $(F_N)_{N\in\N}$ such that
    $$\lim_{N\to\infty}\frac1{|F_N|}\sum_{n\in F_N}|x_n-L|=0,$$
    then for every $\epsilon>0$ the set $E\coloneqq \{n:x_n>L-\epsilon\}$ is thick.
\end{Lemma}

\begin{proof}
    Denote by $D\coloneqq \N\setminus E$
    $$\sum_{n\in F_N}|x_n-L|=\sum_{n\in F_N\cap E}|x_n-L|+\sum_{n\in F_N\cap D}|x_n-L|\geq\epsilon|F_N\cap D|$$
    Dividing by $|F_N|$ and letting $N\to\infty$ we deduce that $\bar d_{(F_N)}(D)=0$ and hence $d_{(F_N)}(E)=1$, which implies that $E$ is thick.
\end{proof}

The following theorem is used in the proof of \cref{thm_mainsingle} to show that the set $R_{\epsilon,A}$ defined in \eqref{eq_intro_iii} is $W$-syndetic.
It can be viewed as an analogue of the classical von Neumann ergodic theorem.

\begin{Theorem}
\label{thm_W-uniform-vonNeumann-along-functions-in-F}
Let $U\colon \Hilb\to\Hilb$ be a unitary operator on a Hilbert space $\Hilb$
and let $P\colon \Hilb\to\Hilb$ denote the orthogonal projection onto the subspace of $U$-invariant elements in $\Hilb$.
Let $\ell\in\N\cup\{0\}$, $f\in\F_{\ell+1}$ and define $W\coloneqq \Delta^\ell f$. Then for any $h\in \Hilb$ we have
$$
\otherlim{UW}{n\to\infty} U^{\lfloor f(n) \rfloor}h~=~P h \qquad\text{in norm}.
$$
\end{Theorem}

\begin{proof}
Using the spectral theorem for unitary operators, we can assume without loss of generality that $\Hilb=L^2(\T,\nu)$ for some Borel probability measure $\nu$ on $\T$, and $U h= e(x)h(x)$ for all $h\in L^2(\T,\nu)$, where $e(x)=e^{2\pi i x}$.
Also, note that the orthogonal projection onto the subspace of $U$-invariant elements in $L^2(\T,\nu)$ is given by
$$
P h (x) =
\begin{cases}
h(0),&\text{if } x=0;\\
0,&\text{if }x\in\T\setminus \{0\}.
\end{cases}
$$

Fix $h\in L^2(\T,\nu)$. Then, in light of \cref{cor_floorwd}, we have for any $x\in \T\setminus\{0\}$,
$$
\otherlim{UW}{n\to\infty} U^{\lfloor f(n) \rfloor}h(x)=h(x)\left( \otherlim{UW}{n\to\infty} e(x\lfloor f(n)\rfloor)\right)=0.
$$
On the other hand, $\otherlim{UW}{n\to\infty} U^{\lfloor f(n) \rfloor}h(0)=h(0)$.
It now follows from the dominated convergence theorem that $\otherlim{UW}{n\to\infty} U^{\lfloor f(n) \rfloor}h=P h$ in $L^2$-norm and the proof is completed.
\end{proof}

We are finally in position to prove \cref{thm_mainsingle}.
\begin{proof}[Proof of \cref{thm_mainsingle}]
Let $\xbmt$ be an invertible probability measure preserving system, let $A\in{\mathcal B}$ and let $\epsilon>0$.
Also, let $\ell\in\N\cup\{0\}$, let $f\in\F_{\ell+1}$ and let $W=\Delta^\ell f$.
Using \cref{thm_W-uniform-vonNeumann-along-functions-in-F}
 and the ergodic decomposition we conclude that
\begin{equation}\label{eq_temperedkhintchine}
  \otherlim{UW}{n\to\infty}\mu(A\cap T^{-\lfloor f(n) \rfloor}A)\geq\mu^2(A).\footnote{The inequality \eqref{eq_temperedkhintchine} can be viewed as an analogue of the classical result due to Khintchine \cite{Khintchine34} which states that for every measure preserving system $(X,{\mathcal B},\mu,T)$ and every $A\in{\mathcal B}$, $\displaystyle \lim_{N-M\to\infty}\frac1{N-M}\sum_{n=M}^N\mu(A\cap T^{-n}A)\geq\mu^2(A)$.}
\end{equation}
This proves that the set $R_{\epsilon,A}$ defined by \eqref{eq_intro_iii} is $W$-syndetic.

We now move to show that $R_{\epsilon,A}$ is also thick.
In view of the Jacobs-de Leeuw-Glicksberg Decomposition, \cref{thm_jlg}, we can write $\1_A=h_{c}+h_{wm}+h_\epsilon$, where $h_c$, $h_{wm}$ and $h_{\epsilon}$ are pairwise orthogonal, $h_{wm}$ is a weakly mixing function, $h_{\epsilon}$ satisfies $\|h_\epsilon\|_{L^2}\leq \epsilon/2$ and
$$h_c=\sum_{j=1}^d c_j f_j$$
for some $d\in\N$, $c_j\in\C$ and eigenfunctions $f_j\in L^2(X,\mathcal{B},\mu)$ satisfying $Tf_j=e(\theta_j) f_j$ for some $\theta_j\in\T$.
Moreover, $h_c\geq 0$ and $\int h_c \d\mu =\mu(A)$.
In view of Cauchy-Schwarz inequality we have that $\|h_c\|_{L^2}^2\geq\mu(A)^2$.

    Observe that $R_{\epsilon,A}$ contains the set
    $$J\coloneqq \Big\{n\in\N:\langle T^{g(n)}h_c,h_c\rangle + \langle T^{g(n)}h_{wm},h_{wm}\rangle > \left\| h_c\right\|_{L^2}^2-\tfrac\epsilon2\Big\},$$
    so it suffices to show that $J$ is thick.
    In view of \cref{lemma_thickfolner} it suffices to find a F\o lner sequence $(F_N)_{N\in\N}$ of $\N$ for which
    \begin{equation}\label{eq_proof_main-01}
    \lim_{N\to\infty}\frac1{|F_N|}\sum_{n\in F_N}\left|\langle T^{g(n)}h_c,h_c\rangle + \langle T^{g(n)}h_{wm},h_{wm}\rangle -  \left\| h_c\right\|_{L^2}^2\right|=0.
    \end{equation}
As a matter of fact, we will show that there exists a F\o lner sequence $(F_N)_{N\in\N}$ such that simultaneously
\begin{equation}\label{eq_proof_main-02}
    \lim_{N\to\infty}\frac1{|F_N|}\sum_{n\in F_N}\left|\langle T^{g(n)}h_c,h_c\rangle -  \left\| h_c\right\|_{L^2}^2\right|=0,
    \end{equation}
\begin{equation}\label{eq_proof_main-03}
    \lim_{N\to\infty}\frac1{|F_N|}\sum_{n\in F_N}\left|\langle T^{g(n)}h_{wm},h_{wm}\rangle \right|=0.
    \end{equation}
It is obvious that \eqref{eq_proof_main-02} and \eqref{eq_proof_main-03} together imply \eqref{eq_proof_main-01}.
The existence of such a F\o lner sequence is equivalent to the statement that, for every $N_0\in\N$ and every $\delta>0$ there exist $a,N\in\N$ with $N\geq N_0$ such that
\begin{eqnarray}
\label{eq_proof_main-04}
\frac1N\sum_{n=0}^N\left|\langle T^{g(a+n)}h_c,h_c\rangle -  \left\| h_c\right\|_{L^2}^2\right|\leq\delta,
\\
\label{eq_proof_main-05}
\frac1N\sum_{n=0}^N
\left|\langle T^{g(a+n)}h_{wm},h_{wm}\rangle \right|\leq\delta.
\end{eqnarray}

Let $N_0\in\N$ and $\delta>0$ be arbitrary.
Recall that  $h_c=\sum_{j=1}^d c_j f_j$ where $f_j\in L^2(X,\mathcal{B},\mu)$ satisfy $Tf_j=e(\theta_j) f_j$.
Define $F\coloneqq \{\theta_1,\ldots,\theta_d\}\subset \T$.
    Applying \cref{lemma_combined} we find $a,N\in\N$ such that \eqref{eq_proof_main-05} is satisfied and $\|g(a+n)\theta_j\|_\T<\delta^2/(2\pi\|h_c\|^4_{L^2})$ for all $\theta_j\in F$ and every $n\in\{0,1,\dots,N\}$.

Since $T^{g(a+n)} f_j = e(\theta_j g(a+n)) f_j$ and since $|e(x)-1|<2\pi\|x\|_\T$ for every $x\in\T$, we deduce that $\|h_c-T^{g(a+n)}h_c\|_{L^2}\leq\delta/\|h_c\|_{L^2}$. It follows that
\begin{equation}
\label{eq_proof_main-06}
\left|\langle T^{g(a+n)}h_c,h_c\rangle -  \left\| h_c\right\|_{L^2}^2\right|\leq \delta,\qquad\forall n\in\{1,\ldots,N\}.
\end{equation}
From \eqref{eq_proof_main-06} it follows that \eqref{eq_proof_main-04} holds and this finishes the proof.
     \end{proof}

\section{Multiple recurrence along sequences in $\F$}
\label{sec_4}
In this section we provide a proof of \cref{thm_multipleconsecutiverecurrence_0}, modulo a result concerning equidistribution on nilmanifolds (see \cref{thm_MainNilmanifoldEquidistribution}) to be proved in \cref{sec_u}.
Actually, we prove a slightly stronger result (see \cref{thm_multipleconsecutiverecurrence} below), that deals with the class of sequences $\mathcal{S}(f)$ which we will presently introduce.

\begin{Definition}
\label{def_ChildrenOfTheFunction}
Let $\ell\in\N$ and $f\in\F_{\ell}$.
Denote by ${\mathcal S}(f)$ the collection
\begin{equation*}
{\mathcal S}(f)\coloneqq \left\{
\sum_{i=0}^k c_i\Delta^i f:k\geq0,c_0,\ldots,c_k\in\Q\right\}.
\end{equation*}
Note that ${\mathcal S}(f)$ is closed under pointwise addition 
and under (discrete) differentiation $\Delta$.
In view of \cref{lem_TaylorLeibnitz}, the set ${\mathcal S}(f)$ is also closed under the shift $n\mapsto n+1$.
Observe that every $g\in{\mathcal S}(f)$ tends either to $\infty$ or to $0$ as $n\to\infty$. Therefore if $g,\tilde g\in{\mathcal S}(f)$ are such that $\lim_{n\to\infty}\big| g(n)-\tilde g(n)\big|=c<\infty$, then $c=0$.

If $g\in\mathcal{S}(f)$ is such that $g(n)\to0$ as $n\to\infty$ we say that the \define{degree} of $g$ is $0$.
For every other $g\in\mathcal{S}(f)$ there exists a unique $d\in\N$ and $\tilde g\in\F_{d}$ such that $|g(n)-\tilde g(n)|\to0$ as $n\to\infty$.
We call $d$ the \define{degree} of $g$.
\end{Definition}

\begin{Theorem}
\label{thm_multipleconsecutiverecurrence}
Let $k\in\N$, $\ell\in\N$, $f\in\F_{\ell}$, $W\coloneqq \Delta^{\ell-1} f$ and $f_1,\ldots,f_k\in {\mathcal S}(f)$.
Let $\xbmt$ be an invertible probability measure preserving system.
\begin{enumerate}
\item
For any $h_1,\ldots,h_k\in L^\infty\xbm$ the limit
$$
\otherlim{UW}{n\to\infty}
T^{\lfloor f_1(n)\rfloor}h_1\cdot T^{\lfloor f_2(n)\rfloor}h_2\cdots T^{\lfloor f_k(n)\rfloor}h_k
$$
exists in $L^2\xbm$.
\item
For any $A\in{\mathcal B}$ with $\mu(A)>0$,
  \begin{equation}
\otherlim{UW}{n\to\infty}\mu\big(A\cap T^{-\lfloor f_1(n)\rfloor}A\cap T^{-\lfloor f_2(n)\rfloor}A\cap\cdots\cap T^{-\lfloor f_k(n)\rfloor}A\big)>0.
  \end{equation}
\end{enumerate}
\end{Theorem}

The structure of \cref{sec_4} is as follows.
In Subsection \ref{sec_background} we provide a summary of relevant background material. In Subsection \ref{sec_CharFact} we prove \cref{thm_characteristic_factor} which essentially shows that nilfactors are characteristic for the averages appearing in \cref{thm_multipleconsecutiverecurrence}.
In Subsection \ref{sec_reduction} we use \cref{thm_characteristic_factor} to reduce \cref{thm_multipleconsecutiverecurrence} to nilsystems; this is the content of \cref{thm_multirecnilsystem}.
We then show how \cref{thm_multirecnilsystem} follows from an equidistribution result, \cref{thm_lastuniformdistribution}.
Finally, in Subsection \ref{sec_additionalreduction}, we further reduce \cref{thm_lastuniformdistribution} to a more streamlined statement concerning equidistribution on nilmanifolds, \cref{thm_MainNilmanifoldEquidistribution}, which is proved in \cref{sec_u}.

\subsection{Background}\label{sec_background}
In this subsection we collect some facts which will be used throughout \cref{sec_4}.
\\

\paragraph{\textit{\textbf{Nilsystems:}}} Let $G$ be an $s$-step nilpotent Lie group and let $\Gamma\subset G$ be a discrete closed subgroup such that $X\coloneqq G/\Gamma$ is compact.
The space $X$ is called a \define{nilmanifold}.
A closed subset $Y\subset X$ is a \define{sub-nilmanifold} if there exists a closed subgroup $H\subset G$ such that $Y=\pi(H)$, where $\pi\colon G\to X$ is the natural projection.

For each $b\in G$ we define a translation on $X$ (which we also denote by $b$) as follows: given $g\Gamma \in X$ we let $b\cdot g\Gamma=(bg)\Gamma$.
Let us denote by $\mu_X$ the normalized Haar measure on $X$, i.e., the unique Borel probability measure invariant under the natural action of $G$ on $X$ (cf. \cite{Raghunathan72}).
Given $b\in G$, the triple $(X,\mu_X,b)$ is called an \define{$s$-step nilsystem}.
A measure preserving system is called an \define{$s$-step pro-nilsystem} if it is (isomorphic in the category of measure preserving systems to) an inverse limit of $s$-step nilsystems.
In other words, $\xbmt$ is an $s$-step pro-nilsystem if there exist $\sigma$-algebras ${\mathcal B}_1\subset{\mathcal B_2}\subset\cdots\subset{\mathcal B}$ such that ${\mathcal B}=\bigcup {\mathcal B}_n$ and for all $n$, the system $(X,{\mathcal B}_n,\mu,T)$ is isomorphic to an $s$-step nilsystem.
\\

\paragraph{\textit{\textbf{Uniformity seminorms:}}}We will make use of the uniformity seminorms introduced in \cite{Host_Kra05} for ergodic systems. As was observed in \cite{Chu_Frantzikinakis_Host11}, the ergodicity of the system is not necessary.

Given a bounded sequence $a\colon \N\to\C$ we define
the \define{Ces{\`a}ro mean} of $a$ as
$$
\otherlim{C}{n\to\infty} a(n) \coloneqq \lim_{N\to\infty}\frac{1}{N}\sum_{n=1}^N a(n)
$$
whenever this limit exists. Note that \textlim{C} coincides with \textlim{W} for $W(n)= n$.

\begin{Definition}
  Let $\xbmt$ be an invertible probability measure preserving system.
  We define the \define{uniformity seminorms} on $L^\infty(X)$ recursively as follows.
  $$\lhk h\rhk_0=\int_Xh\d\mu\qquad\text{and}\qquad\lhk h\rhk_{s}^{2^{s}}=\otherlim{C}{n\to\infty}\lhk \bar h\cdot T^nh\rhk_{s-1}^{2^{s-1}}\text{ for every }s\in\N.$$
\end{Definition}
The existence of the limits in this definition was established in \cite{Host_Kra05} for ergodic systems and in \cite[Section 2.2]{Chu_Frantzikinakis_Host11} in general.

The following result from \cite{Host_Kra05} gives the relation between the uniformity seminorms and nilsystems.
\begin{Theorem}\label{thm_hk}
  Let $\xbmt$ be an ergodic system. For each $s\in\N$ there exists a factor ${\mathcal Z}_s\subset{\mathcal B}$ with the following properties:
\begin{enumerate}
  \item The measure preserving system $(X,{\mathcal Z}_s,\mu,T)$ is an $s$-step pro-nilsystem;
  \item $\displaystyle L^\infty\big({\mathcal Z}_s\big)=\left\{h\in L^\infty({\mathcal B}):\int_Xh\cdot h'\d\mu=0\quad\forall h'\in L^\infty(X,{\mathcal B},\mu)\text{ with }\lhk h'\rhk_{s+1}=0 \right\}$.
\end{enumerate}
\end{Theorem}

\subsection{Characteristic Factors}\label{sec_CharFact}

Here is the main theorem of this subsection.

\begin{Theorem}\label{thm_characteristic_factor}
Let $\ell,k\in\N$, $f\in\F_\ell$ and $W=\Delta^{\ell-1}f$.
For every $f_1,\dots,f_k\in{\mathcal S}(f)$ with $\lim|f_i(n)-f_j(n)|=\lim|f_i(n)|=\infty$ whenever $i\neq j$, there exists $s\in\N$ such that for any measure preserving system $(X,\mu,T)$ and any functions $h_1,\dots,h_k\in L^\infty(X)$ with the property $\lhk h_i\rhk_s=0$ for some $i$, one has
$$\otherlim{UW}{n\to\infty}\prod_{i=1}^k T^{\lfloor f_i(n)\rfloor}h_i=0,
$$
in $L^2\xbm$.
\end{Theorem}

The following lemma represents a useful intermediate step in obtaining \cref{thm_characteristic_factor}.

\begin{Proposition}\label{prop_characteristic_factor}
Let $\ell,k\in\N$, let $f\in\F_\ell$ and let $W=\Delta^{\ell-1}f$.
Suppose $f_0,f_1,\dots,f_k\in{\mathcal S}(f)$ satisfy $\lim|f_i(n)-f_j(n)|=\lim|f_i(n)|=\infty$ whenever $i\neq j$ and that the degree of $f_k$ is greater or equal than the degree of any other $f_i$.
Then there exists $s\in\N$ such that for any measure preserving system $(X,\mu,T)$, any $h_k\in L^\infty(X)$ with $\lhk h_k\rhk_s=0$ and any bounded sequence $(a_n)_{n\in\N}$ in $\C$ we have
$$
\lim_{W(N)-W(M)\to\infty}~\sup_{\|h_0\|_{\infty},\ldots,\|h_{k-1}\|_{\infty}\leq 1}~\left\|\E_{n\in[M,N]}^Wa_n\prod_{i=0}^k T^{\lfloor f_i(n)\rfloor}h_i\right\|_{L^2}=0.
$$
\end{Proposition}

\begin{proof}[Proof that \cref{prop_characteristic_factor} implies \cref{thm_characteristic_factor}]

After permuting the $f_i$'s and $h_i$'s we assume that $\lhk h_k\rhk_s=0$. Also, after multiplying by a constant we will assume that $\|h_i\|_\infty\leq 1$ for all $i$.

The desired conclusion can be rephrased as follows: let $(M_t)_{t\in\N}$, $(N_t)_{t\in\N}$ be sequences in $\N$ satisfying $W(N_t)-W(M_t)\to\infty$.
Then
\begin{equation}\label{eq_proof_corol_characteristicfactor}
  \lim_{t\to\infty}\left\|\E^W_{n\in[M_t,N_t]}\prod_{i=1}^k T^{\lfloor f_i(n)\rfloor}h_i\right\|_{L^2}=0.
\end{equation}
Letting $h_0^{(t)}:=\overline{\E^W_{n\in[M_t,N_t]}\prod_{i=1}^k T^{\lfloor f_i(n)\rfloor}h_i}$, we can rewrite \eqref{eq_proof_corol_characteristicfactor} as
\begin{equation}\label{eq_proof_corol_characteristicfactor2}
  \lim_{t\to\infty}\int_X\E^W_{n\in[M_t,N_t]}h_0^{(t)}\prod_{i=1}^k T^{\lfloor f_i(n)\rfloor}h_i\d\mu=0.
\end{equation}
  Let $\tilde f\in\F_{\ell+1}$ be such that $\Delta\tilde f=f$.
  Observe that ${\mathcal S}(f)\subset{\mathcal S}(\tilde f)$.
  Since $T$ preserves the measure,
  we have
  $$\int_X\E^W_{n\in[M_t,N_t]}h_0^{(t)}\prod_{i=1}^k T^{\lfloor f_i(n)\rfloor}h_i\d\mu
  =
  \int_X\E^W_{n\in[M_t,N_t]}T^{\lfloor\tilde f(n)\rfloor}h_0^{(t)}\prod_{i=1}^k T^{\lfloor f_i(n)\rfloor+\lfloor\tilde f(n)\rfloor}h_i\d\mu.$$
  Observe that $\lfloor x\rfloor+\lfloor y\rfloor=\lfloor x+y\rfloor+e$ for some $e=e(x,y)\in\{0,1\}$.
  Therefore we can write $\lfloor f_i(n)\rfloor+\lfloor\tilde f(n)\rfloor=\lfloor f_i(n)+\tilde f(n)\rfloor+e_{i,n}$ where $e_{i,n}\in\{0,1\}$.
Let $v_n:=(e_{1,n},\dots,e_{k,n})\in\{0,1\}^k$ and, for each $v=(v_1,\dots,v_k)\in\{0,1\}^k$, let $A_v=\{n\in\N:v_n=v\}$ be the set of $n$'s for which $(e_{1,n},\dots,e_{k,n})=v$.
Since $1=\sum_{v\in\{0,1\}^k}1_{A_v}$, \eqref{eq_proof_corol_characteristicfactor2} becomes
\begin{equation*}
  \sum_{v\in\{0,1\}^k}\lim_{t\to\infty}\E^W_{n\in[M_t,N_t]}1_{A_v}(n)\int_XT^{\lfloor\tilde f(n)\rfloor}h_0^{(t)}\prod_{i=1}^k T^{\lfloor f_i(n)+\tilde f(n)\rfloor}\big(T^{v_i}h_i\big)\d\mu=0.
\end{equation*}

After replacing $h_i$ with $T^{v_i}h_i$ (and using the fact that $\lhk h_i\rhk_s=\lhk T^{v_i}h_i\rhk_s$ for every $s\in\N$) it suffices to show that, for any fixed bounded sequence $(a_n)$,
  \begin{equation}\label{eq_proof_corol_characteristicfactor3}
  \lim_{t\to\infty}\E^W_{n\in[M_t,N_t]}a_n\int_XT^{\lfloor\tilde f(n)\rfloor}h_0^{(t)}\prod_{i=1}^k T^{\lfloor f_i(n)+\tilde f(n)\rfloor}h_i\d\mu=0.
\end{equation}
  Finally, observe that now all the functions $ f_i(n)+\tilde f(n)$ have the same degree $\ell+1$, so that we satisfy all the conditions of \cref{prop_characteristic_factor}.
  Using the Cauchy-Schwarz inequality we conclude that indeed

\begin{eqnarray*}
  &&\left|\E^W_{n\in[M_t,N_t]}a_n\int_XT^{\lfloor\tilde f(n)\rfloor}h_0^{(t)}\prod_{i=1}^k T^{\lfloor f_i(n)+\tilde f(n)\rfloor}h_i\d\mu\right|
  \\&\leq &
  \left\|\E^W_{n\in[M_t,N_t]}a_nT^{\lfloor\tilde f(n)\rfloor}h_0^{(t)}\prod_{i=1}^{k}T^{\lfloor f_i(n)+\tilde f(n)\rfloor}h_i\right\|_{L^2}
  \\&\leq &
  \sup_{\|\tilde h_0\|_{\infty},\ldots,\|\tilde h_{k-1}\|_{\infty}\leq 1}~\left\|\E^W_{n\in[M_t,N_t]}a_n\prod_{i=0}^{k}T^{\lfloor f_i(n)+\tilde f(n)\rfloor}\tilde h_i\right\|_{L^2}
  \\&\to&0\qquad\text{ as }t\to\infty,
\end{eqnarray*}
where in the last step we invoked \cref{prop_characteristic_factor} (and we set $f_0\equiv0$ and $\tilde h_k:=h_k$).
\end{proof}

Let us now turn to the proof of \cref{prop_characteristic_factor}. We start with the following lemma whose simple proof is omitted.

\begin{Lemma}\label{cor_multidimRiesz0}
Let $W\in\F_1$ and let $(N_t)_{t\in\N}$ and $(M_t)_{t\in\N}$ be sequences of positive integers with $W(N_t)-W(M_t)\to\infty$ as $t\to\infty$. Then for every $\epsilon>0$ there exists $\delta>0$ and $t_0>1$ such that for all $t\geq t_0$ and $a\colon \N^2\to\C$ with $\sup_{(n,m)\in\N^2}|a(n,m)|\leq 1$, if
$$
\E_{n,m\in[M_t,N_t]}^W|a_t(n,m)|^2 \leq \delta
$$
then
$$
\E_{n,m\in[M_t,N_t]}^W|a_t(n,m)| \leq \epsilon.
$$
%
\end{Lemma}

The following lemma is based on \cite[Lemma 3.2]{Bergelson_Knutson09}; see also \cite{Frantzikinakis10}.
\begin{Lemma}\label{lemma_wmMagic}
Let $\{f_0,f_1,\dots,f_k\}\subset\F$, let $W\in\F_1$ and let $s\in\N$.
Suppose for any
measure preserving system $(X,\mu,T)$ and any $h_k\in L^\infty(X)$ with $\lhk h_k\rhk_{s-1}=0$ we have
$$
\lim_{W(N)-W(M)\to\infty}~\sup_{\|h_0\|_{\infty},\ldots,\|h_{k-1}\|_{\infty}\leq 1}~\left\|\E_{n\in[M,N]}^W\prod_{i=0}^k T^{\lfloor f_i(n)\rfloor}h_i\right\|_{L^2}=0.
$$
Then for every bounded sequence $a\colon \N\to\C$, any measure preserving system $(X,\mu,T)$ and any $h_k\in L^\infty(X)$ with $\lhk h_k\rhk_s=0$ we also have
$$
\lim_{W(N)-W(M)\to\infty}~\sup_{\|h_0\|_{\infty},\ldots,\|h_{k-1}\|_{\infty}\leq 1}~\left\|\E_{n\in[M,N]}^Wa_n\prod_{i=0}^k T^{\lfloor f_i(n)\rfloor}h_i\right\|_{L^2}=0.
$$
\end{Lemma}

\begin{proof}
Let $c$ be an upper bound on $|a(n)|$. We have
\begin{eqnarray*}
&&\left\|\E^W_{n\in[M,N]}a(n)\prod_{i=0}^k T^{\lfloor f_i(n)\rfloor}h_i\right\|_{L^2}^2
\\&=&
\int_X\E^W_{n,m\in[M,N]}a(n)\overline{a(m)}\prod_{i=0}^k T^{\lfloor f_i(n)\rfloor}h_i\cdot T^{\lfloor f_i(m)\rfloor}\bar h_i\d\mu
\\ &\leq&
c^2\E^W_{n,m\in[M,N]}\left|\int_X\prod_{i=0}^k T^{\lfloor f_i(n)\rfloor}h_i\cdot T^{\lfloor f_i(m)\rfloor}\bar h_i\d\mu\right|.
\end{eqnarray*}
In view of \cref{cor_multidimRiesz0}, in order to show that the limit as $W(N)-W(M)\to\infty$ of the expression in the previous line is $0$ uniformly over all $h_0,\ldots,h_{k-1}\in L^\infty$ with $\|h_i\|_\infty\leq 1$, it suffices to show that the limit as $W(N)-W(M)\to\infty$ of $\Psi(M,N)$ is $0$ uniformly over $h_0,\ldots,h_{k-1}$, where
$$
\Psi(M,N)\coloneqq \E^W_{n,m\in[M,N]}\left|\int_X\prod_{i=0}^k T^{\lfloor f_i(n)\rfloor}h_i\cdot T^{\lfloor f_i(m)\rfloor}\bar h_i\d\mu\right|^2.
$$
We have
\begin{eqnarray*}
\Psi(M,N)&=&
\E^W_{n,m\in[M,N]}\int_{X\times X}\prod_{i=0}^k(T\times T)^{\lfloor f_i(n)\rfloor}h_i\otimes\bar h_i\cdot(T\times T)^{\lfloor f_i(m)\rfloor}\bar h_i\otimes h_i\d(\mu\otimes\mu)
\\&=&
\int_{X\times X}\left|\E^W_{n\in[M,N]} \prod_{i=0}^k(T\times T)^{\lfloor f_i(n)\rfloor}(h_i\otimes \bar h_i)\right|^2\d(\mu\otimes\mu)
\\&=&\left\|\E^W_{n\in[M,N]} \prod_{i=0}^k(T\times T)^{\lfloor f_i(n)\rfloor}(h_i\otimes \bar h_i)\right\|_{L^2(\mu\otimes\mu)}^2.
\end{eqnarray*}
A quick induction argument gives the inequality $\lhk h\otimes\bar h\rhk_{s-1}\leq\lhk h\rhk_s^2$.
In particular, if $\lhk h_k\rhk_s=0$ then $\lhk h_k\otimes\overline{h_k}\rhk_{s-1}=0$, and so the assumptions of the lemma imply that the limit as $W(N)-W(M)\to\infty$ of this last expression is $0$ uniformly over $h_0\otimes \bar{h}_0,\ldots,h_{k-1}\otimes \bar{h}_{k-1}\in L^\infty(X\times X)$ with $\|h_i\otimes \bar{h}_i\|\leq 1$.
\end{proof}

We also need the following result.

\begin{Lemma}[{\cite[Proposition 21.7, page 327]{Host_Kra18}}]
\label{lem_multicor_uniform_zero}
Let $(X,\mu,T)$ be a measure preserving system, and $h_k\in L^\infty(X)$ with $\lhk h_k\rhk_{k}=0$.
Then for all $c_0,c_1,\ldots,c_k\in \Z$ with $c_k\neq 0$ we have
$$
\lim_{N-M\to\infty}~\sup_{\|h_0\|_{\infty},\ldots,\|h_{k-1}\|_{\infty}\leq 1}\left\|\frac1{N-M} \sum_{n=M}^{N-1} T^{c_0 n}h_0\, T^{c_1 n}h_1\cdot\ldots\cdot T^{c_k n}h_k\right\|_{L^2}~=~0.
$$
\end{Lemma}

We shall prove \cref{prop_characteristic_factor} by an induction scheme described below. The base case of this induction is covered by the following lemma.

\begin{Lemma}[Base case]\label{lemma_characbasecase}
Let $W\in\F_1$, let $f_0,f_1,\dots,f_k\colon \N\to\R$ be functions such that each $f_i(n)\to0$ as $n\to\infty$, and let $c_0,c_1,\dots,c_k\in\Z\setminus\{0\}$ be pairwise distinct.
Then for every measure preserving system $(X,\mu,T)$, every $h_k\in L^\infty(X)$ with $\lhk h_k\rhk_{k+1}=0$, and every bounded sequence $(a_n)_{n\in\N}$ in $\C$ we have
$$
\lim_{W(N)-W(M)\to\infty}~\sup_{\|h_0\|_{\infty},\ldots,\|h_{k-1}\|_{\infty}\leq 1}~\left\|\E_{n\in[M,N]}^W a_n\prod_{i=0}^k T^{\lfloor c_iW(n)+f_i(n)\rfloor}h_i\right\|_{L^2}=0.
$$
\end{Lemma}

\begin{proof}
In view of \cref{cor_zerodensity}, the set $\big\{n\in\N:\lfloor W(n)+f_i(n)\rfloor\neq\lfloor W(n)\rfloor\big\}$ has zero Banach $W$-density, and hence
$$\otherlim{UW}{n\to\infty}\left\|\prod_{i=0}^kT^{\lfloor c_iW(n)+f_i(n)\rfloor}h_i-\prod_{i=0}^kT^{\lfloor c_iW(n)\rfloor}h_i\right\|_{L^2}=0.$$
Since each $c_i$ is an integer, iterating the fact that $\lfloor x+y\rfloor=\lfloor x\rfloor+\lfloor y\rfloor+e$ for some $e=e(x,y)\in\{-1,0\}$, we can write $\lfloor c_iW(n)\rfloor=c_i\lfloor W(n)\rfloor+e_{i,n}$ where $e_{i,n}$ takes only finitely many values.
Thus the vectors $(e_{1,n},\dots,e_{k,n})$ take only finitely many values as $n$ goes through $\N$.
For each of those finitely many vectors $v=(v_1,\dots,v_k)$, let $A_v$ be the set of $n$'s for which $(e_{1,n},\dots,e_{k,n})=v$.
Since $1=\sum_v1_{A_v}$, after replacing $h_i$ with $T^{v_i}h_i$ (and using the fact that $\lhk h_k\rhk_s=\lhk T^{v_k}h_k\rhk_s$ for every $s\in\N$) it suffices to show that
$$
\lim_{W(N)-W(M)\to\infty}~\sup_{\|h_0\|_{\infty},\ldots,\|h_{k-1}\|_{\infty}\leq 1}~\left\|\E_{n\in[M,N]}^W1_{A_v}(n)a_n\prod_{i=0}^kT^{c_i\lfloor W(n)\rfloor}h_i\right\|_{L^2}=0.
$$
for every $v$.
In view of \cref{lemma_wmMagic} it thus suffices to show that
$$\lim_{W(N)-W(M)\to\infty}~\sup_{\|h_0\|_{\infty},\ldots,\|h_{k-1}\|_{\infty}\leq 1}~\left\|\E_{n\in[M,N]}^W\prod_{i=0}^kT^{c_i\lfloor W(n)\rfloor}h_i\right\|_{L^2}=0.$$
whenever $\lhk h_k\rhk_k=0$.

Let $N_m=\max\{n\in\N:\lfloor W(n)\rfloor=m\}$.
Observe that $\sum_{n=N_{m-1}+1}^{N_m}\Delta W(n)=1+\oh_{m\to\infty}(1)$.
If $N_{p-1}< M\leq N_p$ and $N_{q-1}<N\leq N_q$ then
\begin{equation*}
  \begin{split}\sum_{n=M}^N\Delta W(n)&\prod_{i=0}^kT^{c_i\lfloor W(n)\rfloor}h_i=\\ &=\sum_{j=p}^{q-1}\big(1+\oh_{W(N)-W(M)\to\infty}(1)\big)\prod_{i=0}^kT^{c_ij}h_i +\Oh_{W(N)-W(M)\to\infty}(1)\prod_{i=0}^kT^{c_im}h_i
\end{split}
\end{equation*}
and so
\begin{equation*}
  \begin{split}
  \lim_{W(N)-W(M)\to\infty}~\sup_{\|h_0\|_{\infty},\ldots,\|h_{k-1}\|_{\infty}\leq 1}~&\left\|\frac1{W(N)-W(M)}\sum_{n=M}^N\Delta W(n)\prod_{i=0}^kT^{c_i\lfloor W(n)\rfloor}h_i\right\|_{L^2}
\\
 & =\lim_{p-q\to\infty}~\sup_{\|h_0\|_{\infty},\ldots,\|h_{k-1}\|_{\infty}\leq 1}~\left\|\frac1{p-q} \sum_{j=p}^{q-1}\prod_{i=0}^kT^{c_ij}h_i\right\|_{L^2}.
  \end{split}
\end{equation*}
In view of \cref{lem_multicor_uniform_zero}, this limit is $0$.
\end{proof}

The induction that we will use to prove \cref{prop_characteristic_factor} is similar to the PET-induction scheme which was utilized in \cite{Bergelson87}.

Let $f\in\F_{\ell}$.
Given $\phi(n)=\beta(n)+c_1\Delta^{\ell-1}f+\cdots+c_{\ell-1}\Delta f +c_\ell f\in{\mathcal S}(F)$ (cf. \cref{def_ChildrenOfTheFunction}), the highest $i$ for which $c_i\neq0$ coincides with the degree of $\phi$. We call $c_i$ the \define{leading coefficient of $\phi$}.
We say that $\phi,\tilde\phi\in{\mathcal S}(f)$ are \define{equivalent} if they have the same degree and leading coefficient.
Now fix a finite set $P\subset{\mathcal S}(f)$ and, for each $j=1,\dots,\ell$, let $m_j$ denote the number of equivalence classes in $P$ of degree $j$.
The vector $(m_1,\dots,m_\ell)$ is called the \define{characteristic vector} of $P$.
We order characteristic vectors by letting $(m_1,\dots,m_\ell)<(\tilde m_1,\dots,\tilde m_\ell)$ if the maximum $j$ for which $m_j\neq\tilde m_j$ satisfies $m_j<\tilde m_j$.

\begin{proof}[Proof of {\cref{prop_characteristic_factor}}]
We prove the theorem by induction on the characteristic vector of $P=\{f_0,f_1,\dots,f_k\}$.
The case when the characteristic vector is of the form $(m,0,\dots,0)$ (i.e., all functions have degree $1$) follows from \cref{lemma_characbasecase}.
Assume now that some function in $P$ has degree at least $2$ and that the theorem has been proved for all families whose characteristic vector is strictly smaller than that of $P$.

If not all functions in $P$ have the same degree then pick $i_0\in\{0,1,\ldots,k-1\}$ such that the degree of $f_{i_0}$ is minimal among all degrees present in $P$. If all functions in $P$ have the same degree and there is more than one equivalence class in $P$ then pick $i_0\in\{0,1,\ldots,k-1\}$ such that $f_{i_0}$ is not equivalent to $f_k$.
Finally, if there is only one equivalency class in $P$ then we can pick $i_0\in\{0,1,\ldots,k-1\}$ arbitrarily.

We now use the van der Corput lemma.
  Let $u_n=a_n\prod_{i=0}^k T^{\lfloor f_i(n)\rfloor}h_i$.
In view of \cref{thm_vdCforRiesz}, instead of
$$
\lim_{W(N)-W(M)\to\infty}~\sup_{\|h_0\|_{\infty},\ldots,\|h_{k-1}\|_{\infty}\leq 1}~\left\|\E_{n\in[M,N]}^W u_n \right\|_{L^2}=0,
$$
it suffices to show that
$$
\lim_{W(N)-W(M)\to\infty}~\sup_{\|h_0\|_{\infty},\ldots,\|h_{k-1}\|_{\infty}\leq 1}~\E_{n\in[M,N]}^W \langle u_{n+m},u_n\rangle =0.
$$
We have
\begin{eqnarray*}\langle u_{n+m},u_n\rangle&=& a_{n+m}\overline{a_n}\int_X\prod_{i=0}^kT^{\lfloor f_i(n+m)\rfloor}h_i\cdot T^{\lfloor f_i(n)\rfloor}\overline{h_i}\d\mu
\\&=&
a_{n+m}\overline{a_n} \int_X\prod_{0\leq i\leq k} T^{\lfloor f_i(n+m)\rfloor-\lfloor f_{i_0}(n)\rfloor}h_i\cdot T^{\lfloor f_i(n)\rfloor-\lfloor f_{i_0}(n)\rfloor}\overline{h_i}\d\mu.
\end{eqnarray*}

We now repeat an argument that was already used in the proof of \cref{lemma_characbasecase}. Since $\lfloor a\rfloor-\lfloor b\rfloor=\lfloor a-b\rfloor+e$ for some $e\in\{0,1\}$, we can write $\lfloor f_i(n+m)\rfloor-\lfloor f_{i_0}(n)\rfloor=\lfloor f_i(n+m)- f_{i_0}(n)\rfloor+e_{i,n}$ where $e_{i,n}\in\{0,1\}$ for every $i=1,\dots,k$ and $\lfloor f_i(n)\rfloor-\lfloor f_{i_0}(n)\rfloor=\lfloor f_i(n)- f_{i_0}(n)\rfloor+\tilde e_{i,n}$ where $\tilde e_{i,n}\in\{0,1\}$ for every $i=2,\dots,k$.
For each vector $v=(v_1,\dots,v_k,\tilde v_2,\dots\tilde v_k)\in\{0,1\}^{2k-1}$, let $A_v$ be the set of $n$'s for which $(e_{1,n},\dots,e_{k,n},\tilde e_{2,n},\dots\tilde e_{k,n})=v$.
Since $1=\sum_v1_{A_v}$, we have that $\langle u_{n+m},u_n\rangle$ can be written as
$$\sum_v1_{A_v}(n)a_{n+m}\overline{a_n}\int_X\prod_{0\leq i\leq k}T^{\lfloor f_i(n+m)\rfloor-\lfloor f_{i_0}(n)\rfloor}h_i\cdot T^{\lfloor f_i(n)\rfloor-\lfloor f_{i_0}(n)\rfloor}\overline{h_i}\d\mu.$$
We will show that each summand goes to $0$ as $W(N)-W(M)\to\infty$ uniformly over $h_0$.
Fix $v\in\{0,1\}^{2k-1}$ and let
\begin{equation*}
\tilde f_i(n)\coloneqq
\begin{cases}
f_i(n)-f_{i_0}(n),&\text{if }i<i_0
\\
f_{i-1}(n)-f_{i_0}(n),&\text{if }i>i_0
\end{cases}
\quad\text{and}\quad\tilde h_i\coloneqq
\begin{cases}
T^{v_i}h_i,&\text{if }i<i_0
\\
T^{v_{i-1}}h_{i-1},&\text{if }i>i_0
\end{cases},
\end{equation*}
$$\tilde f_{k+i}\coloneqq f_i(n+m)-f_{i_0}(n) \quad\text{and}\quad\tilde h_{k+i}\coloneqq T^{\tilde v_i}\overline{h_i}\qquad\text{for each }i\in\{0,1,\dots,k\}.$$
Observe that $\lhk\tilde h_{2k}\rhk_s=\lhk h_k\rhk_s=0$.
Since $f_k$ has the highest degree within $P$ and due to our choice of $i_0$, it follows that $\tilde f_{2k}$ has the highest degree within $\tilde P\coloneqq \{\tilde f_1,\dots,\tilde f_{2k}\}$.
Indeed $\tilde f_{2k}$ either has the same degree as $f_k$ (in the case $f_{i_0}$ and $f_k$ are not equivalent) or its degree is that of $f_k$ minus $1$ (in the case that there is only one equivalence class in $P$, which implies that every function in $\tilde P$ has the same degree)\footnote{For some values of $m$ the degree of $\tilde f_k$ may be smaller than the degree of $f_k$ minus one; for instance if $f_{i_0}(n)=f_k(n+1)$ and $m=1$. Nevertheless, for all but finitely many values of $m$ the degree of $\tilde f_k$ is precisely one less than the degree of $f_k$.}.
Moreover, since the degree of $f_k$ is at least $2$, the degree of $\tilde f_{2k}$ is at least $1$.

It remains to show that
$$\lim_{W(N)-W(M)\to\infty}~\sup_{\|h_0\|_{\infty},\ldots,\|h_{k-1}\|_{\infty}\leq 1}~\E_{n\in[M,N]}^W 1_{A_v}(n)a_{n+m}\overline{a_n}\int_X\prod_{i=0}^{2k} T^{\lfloor\tilde f_i(n)\rfloor}\tilde h_i \d\mu=0$$
for all $v$ and $m$.
Replacing $\sup_{\|h_0\|_{\infty},\ldots,\|h_{k-1}\|_{\infty}\leq 1}$ with $\sup_{\|h_0\|_{\infty},\ldots,\|h_{2k-1}\|_{\infty}\leq 1}$, replacing the convergence in measure with norm convergence, and using \cref{lemma_wmMagic}, it thus suffices to prove that
\begin{equation}\label{eq_proof_petinduction1} \lim_{W(N)-W(M)\to\infty}~\sup_{\|h_0\|_{\infty},\ldots,\|h_{2k-1}\|_{\infty}\leq 1} \E_{n\in[M,N]}^W \prod_{i=0}^{2k} T^{\lfloor\tilde f_i(n)\rfloor}\tilde h_i=0
\end{equation}
whenever $\tilde h_{2k}\in L^\infty(X)$ with $\lhk \tilde{h}_{2k}\rhk_{s-1}=0$.
If some function $\tilde f_i$ in $\tilde P$ has degree $0$, then $\lim_{n\to\infty}\tilde f_i(n)=0$ and so we can remove it from $\tilde P$ and ignore the corresponding term $T^{\lfloor\tilde f_i(n)\rfloor}\tilde h_i$ in \eqref{eq_proof_petinduction1}.
We can therefore assume that $\tilde P\subset{\mathcal S}(f)$.
Thus \eqref{eq_proof_petinduction1} will follow by induction after we show that the characteristic vector $(\tilde m_1,\dots,\tilde m_\ell)$ of $\tilde P$ is strictly smaller than the characteristic vector $(m_1,\dots,m_\ell)$ of $P$.

Indeed, for each $f_i$ which is not equivalent to $f_{i_0}$, the functions $f_i(n)-f_{i_0}(n)$ and $f_i(n+m)-f_{i_0}(n)$ are equivalent to each other, and have the same degree as $f_i$.
Moreover, if $f_i$ and $f_j$ are not equivalent to $f_{i_0}$, then $f_i-f_{i_0}$ is equivalent to $ f_j-f_{i_0}$ if and only if $f_i$ is equivalent to $f_j$.
Letting $d$ be the degree of $f_{i_0}$, this shows that $\tilde m_j=m_j$ for all $j>d$.
Finally, if $f_i$ is equivalent to $f_{i_0}$, then both $f_i(n)-f_{i_0}$ and $f_{i}(n+m)-f_{i_0}$ (when $i\neq i_0$) have degree smaller than that of $f_i$.
This shows that $\tilde m_d<m_d$.

Therefore $(\tilde m_1,\dots,\tilde m_\ell)$ is strictly smaller than $(m_1,\dots,m_\ell)$ and by induction \eqref{eq_proof_petinduction1} holds.
\end{proof}

\subsection{Reduction to an equidistribution result}\label{sec_reduction}

In this subsection we prove \cref{thm_multipleconsecutiverecurrence} modulo a result about equidistribution of certain sequences in nilmanifolds.
Let $k,\ell\in\N$, let $f\in\F_{\ell}$, let $W=\Delta^{\ell-1} f$, let $f_1,\dots,f_k\in{\mathcal S}(f)$, let $(X,\mu,T)$ be an invertible probability measure preserving system and let $A\subset X$ have positive measure.
We need to show that
\begin{equation}\label{eq_proof_thm_multipleconsecutiverecurrence}
  \otherlim{UW}{n\to\infty}\mu\big(A\cap T^{\lfloor f_1(n)\rfloor}A\cap T^{\lfloor f_2(n)\rfloor}A\cap\cdots\cap T^{\lfloor f_k(n)\rfloor}A\big)>0.
\end{equation}
and that for any $h_1,\ldots,h_k\in L^\infty\xbm$ the limit
\begin{equation}\label{eq_proofmultipleconvergence}
\otherlim{UW}{n\to\infty}
T^{\lfloor f_1(n)\rfloor}h_1\cdot T^{\lfloor f_2(n)\rfloor}h_2\cdots T^{\lfloor f_k(n)\rfloor}h_k
\end{equation}
exists in $L^2\xbm$.

We start by making several standard reductions whose details are largely omitted (cf. \cite{Bergelson_Leibman_Lesigne08,Frantzikinakis10} where similar reductions were performed).
First, employing the ergodic decomposition (see, for instance, \cite[Theorem 4.8]{Einsiedler_Ward11}) we may assume that $(X,\mu,T)$ is ergodic.
Indeed, assuming that \cref{thm_multipleconsecutiverecurrence} holds for each ergodic component of $\mu$ it follows that \eqref{eq_proof_thm_multipleconsecutiverecurrence} holds (cf. \cite[Section 7.2.3]{Einsiedler_Ward11}), and that \eqref{eq_proofmultipleconvergence} holds (cf. the last paragraph of \cite[Section 21.2.3]{Host_Kra18}).
Next, let $s\in\N$ be given by \cref{thm_characteristic_factor} and let $h$ be the projection of $1_A$ onto the $(s-1)$-step nilfactor ${\mathcal Z}_{s-1}$.
In view of \cref{thm_characteristic_factor} and \cref{thm_hk}, the left hand side of \eqref{eq_proof_thm_multipleconsecutiverecurrence} is the same as
\begin{equation}\label{eq_proof_thm_multipleconsecutiverecurrence2}
  \otherlim{UW}{n\to\infty}\int_X h\cdot T^{\lfloor f_1(n)\rfloor}h\cdot T^{\lfloor f_2(n)\rfloor}h\cdots T^{\lfloor f_k(n)\rfloor}h\d\mu.
\end{equation}
Similarly, letting $\tilde h_i$ be the projection of $h_i$ onto the $(s-1)$-step nilfactor, the limit in \eqref{eq_proofmultipleconvergence} becomes
\begin{equation}\label{eq_proofmultipleconvergence2}
\otherlim{UW}{n\to\infty}
T^{\lfloor f_1(n)\rfloor}\tilde h_1\cdot T^{\lfloor f_2(n)\rfloor}\tilde h_2\cdots T^{\lfloor f_k(n)\rfloor}\tilde h_k.
\end{equation}

The $s$-step nilfactor is an inverse limit of $s$-step nilsystems.
A standard approximation argument (see \cite[Lemma 3.2]{Furstenberg_Katznelson78}) shows that it suffices to establish positivity of \eqref{eq_proof_thm_multipleconsecutiverecurrence2} and show the limit in \eqref{eq_proofmultipleconvergence2} exists under the additional assumptions that the system is a nilsystem and that $h, \tilde h_1,\dots,\tilde h_k$ are continuous functions with $h\geq0$ and $\int_X h\d\mu=\mu(A)>0$.

We have now reduced \cref{thm_multipleconsecutiverecurrence} to the following statement.
\begin{Theorem}\label{thm_multirecnilsystem}
  Let $G$ be a nilpotent Lie group, let $\Gamma\subset G$ be a co-compact discrete subgroup, let $X=G/\Gamma$ and let $b\in G$.
  Let $\ell,k\in\N$, let $f\in\F_{\ell+1}$, let $W\coloneqq \Delta^\ell f$ and let $f_1,\dots,f_k\in{\mathcal S}(f)$.
  \begin{enumerate}
    \item For any $h_1,\ldots,h_k\in C(X)$ and $x\in X$ the limit
$$
\otherlim{UW}{n\to\infty}
h_1\big(b^{\lfloor f_1(n)\rfloor} x\big)\cdot h_2\big(b^{\lfloor f_2(n)\rfloor} x\big)\cdots h_k\big(b^{\lfloor f_k(n)\rfloor}x\big)
$$
exists.
    \item For every non-negative function $h\in C(X)$ with $\int_X h\d\mu_X>0$
  $$\otherlim{UW}{n\to\infty}\int_Xh(x)\cdot h\big(b^{\lfloor f_1(n)\rfloor}x\big)\cdot h\big(b^{\lfloor f_2(n)\rfloor}x\big)\cdots h\big(b^{\lfloor f_k(n)\rfloor}x\big)\d\mu(x)>0.$$
  \end{enumerate}
\end{Theorem}

In order to prove \cref{thm_multirecnilsystem} we will need the following result, which is proved in the next subsection.

\begin{Theorem}\label{thm_lastuniformdistribution}
Let $G$ be a nilpotent Lie group, let $\Gamma\subset G$ be a co-compact discrete subgroup and let $X=G/\Gamma$.
  Let $k,\ell\in\N$, let $f\in\F_{\ell+1}$, let $W=\Delta^\ell f$ and let $f_1,\dots,f_k\in{\mathcal S}(f)$.
Then there exists a constant $C>0$ and for every $b\in G$ there exists a measure $\nu$ on $X^k$ such that the sequence
  \begin{equation}\label{eq_thm_removingfloors}
    n\mapsto\Big(b^{\lfloor f_1(n)\rfloor},b^{\lfloor f_2(n)\rfloor},\dots,b^{\lfloor f_k(n)\rfloor}\Big)\Gamma^k
  \end{equation}
  is $\nu$-w.d.\ with respect to \textlim{W}.
  Moreover, there are linear maps $\phi_1,\dots,\phi_k\colon \Z^{\ell+1}\to\Z$ such that the Haar measure $\mu_Y$ of the subnilmanifold\footnote{The fact that $Y$ is a subnilmanifold of $X$ follows from \cite{Leibman05a} or \cite{Shah98}.}
  $$
    Y\coloneqq \overline{\Big\{\big(b^{\phi_1({\vec m})}\Gamma,\dots,b^{\phi_k({\vec m})}\Gamma\big):{\vec m}\in\Z^{\ell+1}\Big\}}
$$
  satisfies $\nu\geq C\mu_Y$.
\end{Theorem}

\begin{proof}[Proof of {\cref{thm_multirecnilsystem}} conditionally on {\cref{thm_lastuniformdistribution}}]

For each $x\in X$, let $g_x\in G$ be such that $x=g_x\Gamma$.
We have
$$\prod_{i=1}^kh_i\big(b^{\lfloor f_i(n)\rfloor}x\big)
=
\prod_{i=1}^kh_i\Big(g_x\big(g_x^{-1}bg_x\big)^{\lfloor f_i(n)\rfloor}\Gamma\Big).$$
Now let $C$ be the constant given by \cref{thm_lastuniformdistribution} and let $\nu_x$ be the measure on $X^k$ given by the same theorem with $g_x^{-1}bg_x$ in place of $b$, noting crucially that $C$ does not depend on $x$.
Finally, let $H_x(x_1,x_2,\dots,x_k)=h_1(g_xx_1)h_2(g_xx_2)\cdots h_k(g_xx_k)$
We then have
\begin{equation}\label{eq_proof_thm_multirecnilsystem1}
\otherlim{W}{n\to\infty}\prod_{i=1}^kh_i\big(b^{\lfloor f_i(n)\rfloor}x\big)=\int_{X^k} H_x\d\nu_x
\end{equation}
and in particular the limit exists, proving part (1).

Let now $h_1=\cdots=h_k=h$ be a non-negative function with $\int_Xh\d\mu_X>0$.
Let
$$Y_x=\overline{\Big\{\big((g_x^{-1}bg_x)^{\phi_1({\vec m})}\Gamma,\dots,(g_x^{-1}bg_x)^{\phi_k({\vec m})}\Gamma\big):{\vec m}\in\Z^{\ell+1}\Big\}}$$
and let $\mu_{Y_x}$ be the Haar measure on $Y_x$.
Thus \cref{thm_lastuniformdistribution} and \eqref{eq_proof_thm_multirecnilsystem1} imply that
\begin{equation}\label{eq_proof_thm_multirecnilsystem2}
\otherlim{W}{n\to\infty}\prod_{i=1}^kh\big(b^{\lfloor f_i(n)\rfloor}x\big)\geq C\int_{Y_x}H_x\d\mu_{Y_x}.
\end{equation}
Observe that if $h(x)>0$, then $H_x(1_{G^k}\Gamma^k)=h(x)^k>0$.
Since $H_x$ is a continuous function and $1_{G^k}\Gamma^k\in Y_x$ for every $x\in X$, it follows that whenever $h(x)>0$, also $\int_{Y_x}H_x\d\mu_{Y_x}>0$.
Now let $\tilde X\subset X$ be the set of points $x$ for which $h(x)>0$.
Since $\int_Xh\d\mu_X>0$ we have that $\mu_X(\tilde X)>0$.
Let $I\colon \tilde X\to\R$ be the function defined by
$$I(x)\coloneqq h(x)\int_{Y_x}H_x\d\mu_{Y_x}.$$
Since $I(x)>0$ for every $x\in \tilde X$ we have $\int_{\tilde X}I(x)\d\mu_X(x)>0$.
Combining \eqref{eq_proof_thm_multirecnilsystem2} with the dominated convergence theorem we conclude that
$$\otherlim{UW}{n\to\infty}\int_Xh(x)\cdot \prod_{i=1}^kh\big(b^{\lfloor f_i(n)\rfloor}x\big)\d\mu(x)\geq \int_XI(x)\d\mu_X(x)>\int_{\tilde X}I(x)\d\mu_X(x)>0.$$

\end{proof}

\subsection{An additional reduction}\label{sec_additionalreduction}
Given a connected and simply connected nilpotent Lie group $G$, for each $b\in G$ and $t\in\R$, the element $b^t$ is well defined in $G$.
In particular, for every discrete and co-compact subgroup $\Gamma\subset G$, every nilrotation on $G/\Gamma$ is embedable in a flow.
Under these assumptions we are able to derive results about sequences of the form $b^{\lfloor f(n)\rfloor}$ from results about
(simpler) sequences $b^{f(n)}$.

In this subsection we derive \cref{thm_lastuniformdistribution} from \cref{thm_MainNilmanifoldEquidistribution}, which is proved in \cref{sec_u} and whose statement we now recall.

\begin{named}{\cref{thm_MainNilmanifoldEquidistribution}}{}
Let $G$ be a connected and simply connected nilpotent Lie group, let $\Gamma\subset G$ be a co-compact discrete subgroup and let $X=G/\Gamma$.
Let $k,\ell\in\N$ with $k\leq\ell$, let $f\in\F_{\ell}$, let $W=\Delta^{\ell-1}f$, let $b_0,\dots,b_{k-1}\in G$ be commuting elements, let
$$Y=\overline{\Big\{b_0^{t_0}\cdots b_{k-1}^{t_{k-1}}\Gamma:t_0,\dots,t_{k-1}\in\R\Big\}}\subset X,$$
and let $\mu_Y$ be the normalized Haar measure on $Y$.
Then for every continuous function $H\in C(X)$ we have
$$\otherlim{UW}{n\to\infty}H\left(b_0^{f(n)}b_1^{\Delta f(n)}\cdots b_{k-1}^{\Delta^{k-1}f(n)}\Gamma\right)=\int_YH(y)\d\mu_Y(y).$$
\end{named}

To achieve our goal we actually need the following corollary of \cref{thm_MainNilmanifoldEquidistribution}.
\begin{Corollary}\label{thm_nilequidistribution}
Let $k,\ell\in\N$, let $f\in\F_{\ell+1}$, let $W=\Delta^{\ell}f$ and let $f_1,\dots,f_k\in{\mathcal S}(f)$.
Then there are linear maps $\phi_1,\dots,\phi_k\colon \R^{\ell+1}\to\R$ with integer coefficients which satisfy the following property.
  Let $G$ be a connected and simply connected nilpotent Lie group, let $\Gamma\subset G$ be a co-compact discrete subgroup, let $b\in G$, $X=G/\Gamma$ and let $H\in C(X^k)$.

  Then we have
  \begin{equation}\label{eq_thm_nilequidistribution}
  \otherlim{UW}{n\to\infty}H\bigg(\Big(b^{f_1(n)},b^{f_2(n)},\dots,b^{f_k(n)}\Big)\Gamma^k\bigg)=\int_Y H(y)\d\mu_Y,
  \end{equation}
  where $\mu_Y$ is the normalized Haar measure on the subnilmanifold $Y\subset X^k$ defined as
$$
    Y\coloneqq \overline{\Big\{\big(b^{\phi_1({\vec t})}\Gamma,\dots,b^{\phi_k({\vec t})}\Gamma\big):{\vec t}\in\R^{\ell+1}\Big\}}.
$$
\end{Corollary}
\begin{proof}
For each $j=1,\dots,k$, write
$$f_j(n)=c_{j,0}f(n)+c_{j,1}\Delta f(n)+\cdots+c_{j,\ell}\Delta^\ell f(n)+\beta_j(n),$$
where $\beta_j(n)\to0$ as $n\to\infty$ and $(c_{j,0},\dots,c_{j,\ell})\in\Q^{\ell+1}\setminus\{\vec 0\}$.
Multiplying $f$ by a common multiple of the denominators of all the $c_{j,i}$ if necessary we can and will assume that all the $c_{j,i}$ are integers.

For each $i=0,\dots,\ell$ let $b_i\in G^k$ be the point $b_i=\Big(b^{c_{1,i}},b^{c_{2,i}},\dots,b^{c_{k,i}}\Big)$.
Also let $g(n)=(b^{\beta_1(n)},\dots,b^{\beta_k(n)})$.
We have
\begin{equation}\label{eq_proof_thm_nilequidistribution1}
\Big(b^{f_1(n)},b^{f_2(n)},\dots,b^{f_k(n)}\Big)=b_0^{f(n)}b_1^{\Delta f(n)}\cdots b_\ell^{\Delta^\ell f(n)}g(n).
\end{equation}

Since each $\beta_j(n)\to0$ as $n\to\infty$ it follows that $g(n)\to1_G$ as $n\to\infty$.
Since $X$ is compact, $H$ is uniformly continuous and hence
$$\lim_{n\to\infty}\left|H\left(\Big(b^{f_1(n)},b^{f_2(n)},\dots,b^{f_k(n)}\Big)\Gamma^{k}\right)-H\Big(b_0^{f(n)}b_1^{\Delta f(n)}\cdots b_\ell^{\Delta^\ell f(n)}\Gamma^{k}\Big)\right|=0.$$
For each $j=1,\dots,k$ and ${\vec t}=(t_0,\dots,t_\ell)\in\R^{\ell+1}$, define $\phi_j({\vec t})=c_{j,0}t_0+\cdots+c_{j,\ell}t_\ell$.
Notice that $b_0^{t_0}\cdots b_\ell^{t_\ell}=\big(b^{\phi_1({\vec t})}\Gamma,\dots,b^{\phi_k({\vec t})}\big)$.
Appealing to \cref{thm_MainNilmanifoldEquidistribution} it follows that \eqref{eq_thm_nilequidistribution} holds with
$$Y=
\overline{\Big\{b_0^{t_0}\cdots b_\ell^{t_\ell}\Gamma^{k}:{\vec t}\in\R^{\ell+1}\Big\}}
=
\overline{\Big\{\big(b^{\phi_1({\vec t})}\Gamma,\dots,b^{\phi_k({\vec t})}\Gamma\big):{\vec t}\in\R^{\ell+1}\Big\}}
\subset X^{k}.$$
\end{proof}

We are now ready to prove \cref{thm_lastuniformdistribution}.

\begin{proof}[Proof of {\cref{thm_lastuniformdistribution}}]
Every nilpotent Lie group $G$ is the quotient of a simply connected nilpotent Lie group $\tilde G$ by a discrete subgroup.
In turn every simply connected nilpotent Lie group is (isomorphic as a Lie group to) a closed subgroup of a connected and simply connected nilpotent Lie group $\hat G$ (cf. \cite[Subsection 1.11]{Leibman05a}).
Therefore it suffices to prove the theorem assuming $G$ is connected and simply connected.

Let $\tilde G=G\times\R$, let $\tilde\Gamma=\Gamma\times\Z$, let $\tilde X=\tilde G/\tilde\Gamma$ and let $a=(b,1)\in\tilde G$.
Recall the notation $\{t\}=t-\lfloor t\rfloor$ and let $\pi\colon \tilde X^k\to X^k$ be the map
$$\pi\big((g_1,t_1,\dots,g_k,t_k)\tilde\Gamma^k\big)= (b^{-\{t_1\}}g_1,\dots,b^{-\{t_k\}}g_k)\Gamma^K.$$
Observe that $\pi$ is well defined (i.e., the choice of the co-set representative does not matter in the definition of $\pi$) and that $\pi(a^{t_1},\dots,a^{t_k})=(b^{\lfloor t_1\rfloor},\dots,b^{\lfloor t_k\rfloor})$ for every $t_1,\dots,t_k\in\R$.
We warn the reader that $\pi$ is not a continuous map; indeed $\pi$ is discontinuous at the points $(g_1,t_1,\dots,g_k,t_k)\tilde\Gamma^k$ where at least one $t_i\in\Z$, but continuous elsewhere.
In particular $\pi$ is continuous almost everywhere (with respect to the Haar measure on $\tilde X^k$).

Let $\phi_1,\dots,\phi_k$ be given by \cref{thm_nilequidistribution}, let
$$\tilde Y\coloneqq \overline{\Big\{\big(a^{\phi_1(\vec t)},\dots,a^{\phi_k(\vec t)}\big)\tilde\Gamma^k:\vec t\in\R^{\ell+1}\Big\}},$$
let $\mu_{\tilde Y}$ be the Haar measure on $\tilde Y$ and let $\nu\coloneqq \pi_*\mu_{\tilde Y}$ be the pushforward measure.
Notice that for any $H\in C(X^k)$,
$$H\bigg(\Big(b^{\lfloor f_1(n)\rfloor},b^{\lfloor f_2(n)\rfloor},\dots,b^{\lfloor f_k(n)\rfloor}\Big)\Gamma^k\bigg)= H\circ\pi \bigg(\Big(a^{f_1(n)},a^{f_2(n)},\dots,a^{f_k(n)}\Big)\tilde\Gamma^k\bigg)$$
so in view of \cref{thm_nilequidistribution}, and using the fact that $\pi$ is continuous $\mu_{\tilde Y}$-a.e., it follows that the sequence defined in \eqref{eq_thm_removingfloors} is indeed $\nu$-w.d.\ with respect to \textlim{W}.

Finally we show that $\nu\geq C\mu_Y$ for some constant $C$.
Let $H\in C(X^k)$, assume $H(x)\geq0$ for every $x\in X^k$.
In view of \cite[Theorem B]{Leibman05b}
\begin{eqnarray*}
  \int_{X^k}H(x)\d\mu_{Y}(x)
  &=&
  \lim_{N\to\infty}\frac1{N^{\ell+1}}\sum_{\vec n\in[0,N]^{\ell+1}} H\Big(\big(b^{\phi_1(\vec n)},\dots,b^{\phi_k(\vec n)}\big)\Gamma^k\Big)
\end{eqnarray*}
and similarly
\begin{eqnarray*}
  \int_{X^k}H(x)\d\nu(x)
  &=&
  \int_{\tilde Y}H\circ\pi(\tilde y)\d\mu_{\tilde Y}(\tilde y)
  \\&=&
  \lim_{T\to\infty}\frac1{T^{\ell+1}}\int_{[0,T]^{\ell+1}} H\circ\pi\Big(\big(a^{\phi_1(\vec t)},\dots,a^{\phi_k(\vec t)}\big)\tilde\Gamma^k\Big)\d\vec t
  \\&=&
  \lim_{T\to\infty}\frac1{T^{\ell+1}}\int_{[0,T]^{\ell+1}} H\Big(\big(b^{\lfloor \phi_1(\vec t)\rfloor},\dots,b^{\lfloor \phi_k(\vec t)\rfloor}\big)\Gamma^k\Big)\d\vec t
  \\&=&
  \lim_{N\to\infty}\frac1{N^{\ell+1}}\sum_{\vec n\in[0,N)^{\ell+1}}\int_{[0,1]^{\ell+1}} H\Big(\big(b^{\phi_1(\vec n)+\lfloor \phi_1(\vec t)\rfloor},\dots,\phi^{\phi_k(\vec n)+\lfloor\phi_k(\vec t)\rfloor}\big)\Gamma^k\Big)\d\vec t
  \\&\geq&
  C\lim_{N\to\infty}\frac1{N^{\ell+1}}\sum_{\vec n\in[0,N)^{\ell+1}}H\Big(\big(b^{\phi_1(\vec n)},\dots,b^{\phi_k(\vec n)}\big)\Gamma^k\Big)
  \\&=&
  C\int_{X^k}H(x)\d\mu_{Y}(x),
\end{eqnarray*}
where $C>0$ is the Lebesgue measure of the set
$$\big\{\vec t\in[0,1]^{\ell+1}:\lfloor \phi_1(\vec t)\rfloor=\lfloor \phi_2(\vec t)\rfloor=\cdots=\lfloor \phi_k(\vec t)\rfloor=0\big\}.$$
\end{proof}

\section{Well-distribution with respect to Riesz means on nilmanifolds}
\label{sec_u}

In this section we prove \cref{thm_MainNilmanifoldEquidistribution}.
Results of similar nature were obtained by Frantzikinakis in \cite{Frantzikinakis09} (see \cref{rmrk_comparison} in the Introduction).

Let $G$ be a connected simply connected nilpotent Lie group, $\Gamma$ a uniform and discrete subgroup of $G$ and $X\coloneqq G/\Gamma$. We use $\pi\colon G\to X$ to denote the natural projection of $G$ onto $X$.

\begin{Definition}
\label{def:polynomial-sequneces}
Let $m\in\N$. A map $g\colon \Z^m\to G$ is called a \define{polynomial sequence} if there exist $u\in\N$, $a_1,\ldots,a_u\in G$ and $p_1,\ldots,p_u\in\R[x_1,\ldots,x_m]$ such that
$$
g(h)=a_1^{p_1(h)}a_2^{p_2(h)}\cdot\ldots\cdot a_u^{p_u(h)}, \qquad\forall h\in\Z^m.
$$
\end{Definition}

It will also be convenient to introduce the following notation: Given a bounded complex-valued sequence $a\colon\N\to\C$ we define
$$
\otherlimsup{C}{n\to\infty} a(n)\coloneqq \limsup_{N\to\infty}\left|\frac{1}{N}\sum_{n=1}^N a(n) \right|
$$
and
$$
\otherlimsup{UW}{n\to\infty} a(n)\coloneqq \sup_{(M_i,N_i)_{i\in\N}}\lim_{i\to\infty}\left|\frac{1}{W(N_i)-W(M_i)}\sum_{n=M_i}^{N_i}\Delta W(n) a(n) \right|,
$$
where the supremum is taken of all sequences $(M_i,N_i)_{i\in\N}\subset\N^2$ satsifying $M_i\leq N_i$ and $\lim_{i\to\infty} W(N_i)-W(M_i)=\infty$.

We can now state the main theorem of this section.
Notice that \cref{thm_MainNilmanifoldEquidistribution} corresponds to the special case $m=0$.
\begin{Theorem}
\label{thm:equidistribution-of-discrete-tempred-on-subnilmanifold}
Let $G$ be a connected simply connected nilpotent Lie group, $\Gamma$ a uniform and discrete subgroup of $G$ and define $X\coloneqq G/\Gamma$.
Let $\ell,u\in \N$, let $m\in\Z$ with $m\geq0$, let $f\in\F_{\ell}$ and let $g_0,g_1,\ldots,g_{\ell-1}\colon \Z^m\to G$ be polynomial sequences defined by
$$
g_i(h)\coloneqq a_{i,1}^{p_1(h)}a_{i,2}^{p_2(h)}\cdot\ldots\cdot a_{i,u}^{p_{i,u}(h)}, \qquad\forall h\in\Z^m,
$$
where $a_{i,j}$ with $(i,j)\in\{0,1,\ldots,\ell-1\}\times\{1,\ldots,u\}$ is a collection of commuting elements of $G$ and $p_{i,j}\in\R[x_1,\ldots x_m]$.
If $m=0$ then simply let $g_i$ be commuting elements of $G$ for $i=0,\dots,\ell-1$. Set
$$
g(n,h)\coloneqq g_0(h)^{f(n)}g_1(h)^{\Delta f(n)}\cdot\ldots\cdot g_{\ell-1}(h)^{\Delta^{\ell-1} f(n)}, \qquad\forall n\in\N,~\forall h\in\Z^{m}.
$$
Then for all $F\in C(X)$ we have
$$
\otherlimsup{C}{h_m\to\infty} \cdots\otherlimsup{C}{h_1\to\infty}\otherlimsup{UW}{n\to\infty}\left( F\big(g(n,h)\Gamma\big) -\int_Y F\d\mu_Y\right)=0,
$$
where $h=(h_1,\ldots,h_m)$, $W\coloneqq \Delta^{\ell-1} f$, and $\mu_Y$ is the Haar measure of the connected subnilmanifold $Y$\footnote{It is shown in \cref{lem:orbit-closure-is-subnilmanifold} that $Y$ is indeed a connected subnilmanifold.} defined by
$$
Y\coloneqq\overline{\{g_0(h)^{t_0}g_1(h)^{t_1}\cdot\ldots\cdot g_{\ell-1}(h)^{t_{\ell-1}}\Gamma: h\in\Z^m, (t_0,t_1,\ldots,t_{\ell-1})\in\R^{\ell}\}}.
$$
\end{Theorem}

\subsection{The abelian case of \cref{thm:equidistribution-of-discrete-tempred-on-subnilmanifold}}
\label{sec:abelian-case-of-equidistribution-of-discrete-tempred-on-subnilmanifold}

Before we proceed to give a proof of \cref{thm:equidistribution-of-discrete-tempred-on-subnilmanifold} in its full generality, let us establish the following special case.

\begin{Lemma}[\cref{thm:equidistribution-of-discrete-tempred-on-subnilmanifold} for the special case $G=\R^d$]
\label{lem:equidistribution-of-discrete-tempred-on-subtorus}
Let $d,\ell\in \N$, let $m\in\Z$ with $m\geq0$, let $f\in\F_{\ell}$, let $p_0,\ldots,p_{\ell-1}\colon\Z^m\to\R^d$ be polynomials (or just constants if $m=0$) and set
$$
p(n,h)\coloneqq p_0(h){f(n)}+p_1(h){\Delta f(n)}+\ldots+ p_{\ell-1}(h){\Delta^{\ell-1} f(n)},\qquad\forall n\in\N,~\forall h\in\Z^{m}.
$$
Let $\pi_1\colon \R^d\to \T^d$ denote the natural projection from $\R^d$ onto $\T^d=\R^d/\Z^d$ and define $W\coloneqq \Delta^{\ell-1} f$.
Let $Y$ denote the connected subtorus of $\T^d$ given by
$$
Y\coloneqq\overline{\Big\{\pi_1\big(p_0(h){t_0}+p_1(h){t_1}+\ldots+ p_{\ell-1}(h){t_{\ell-1}}\big): h\in\Z^m, (t_0,t_1,\ldots,t_{\ell-1})\in\R^{\ell}\Big\}}.
$$
Then for all $F\in C(\T^d)$ we have
$$
\otherlim{C}{h_m\to\infty} \cdots\otherlim{C}{h_1\to\infty}\otherlim{UW}{n\to\infty} F\big(\pi_1\big(p(n,h)\big)\big) = \int_{Y} F\d\mu_{Y},
$$
where $h=(h_1,\ldots,h_m)$.
\end{Lemma}

\begin{proof}
It suffices to show that for all continuous group characters $\chi\colon \T^d\to\{w\in\C:|w|=1\}$ with the property that $\chi$ restricted to $Y$ is non-trivial, one has
\begin{equation}
\label{eqn:torus-case-1}
\otherlim{C}{h_m\to\infty} \cdots\otherlim{C}{h_1\to\infty}
\otherlim{UW}{n\to\infty} \chi\big(\pi_1\big(p(n,h)\big)\big) = 0,
\end{equation}
because for any such character we have $\int_Y\chi\d\mu_Y=0$ and linear combinations of continuous group characters of this kind together with constants are dense in $C(Y)$.
The character $\chi$ is naturally associated with an element $\tau\in\Z^d$ such that $\chi\big(\pi(t)\big)=e(\langle t,\tau\rangle)$ for every $t\in\R^d$, where $e(x)\coloneqq e^{2\pi i x}$.
The condition that $\chi$ restricted to $Y$ is non-trivial implies that the map
$$
(h,t_0,\ldots,t_{\ell-1})\mapsto \big\langle p_0(h)t_0+p_1(h)t_1+\ldots+ p_{\ell-1}(h)t_{\ell-1},~\tau\big\rangle \bmod 1
$$
is not constant.
Therefore, for some $j\in\{0,\dots,\ell-1\}$, the map $h\mapsto \langle p_j(h),\tau\rangle$ is not identically $0$. Choose the smallest such $j$.
Since $p_j$ is a polynomial, $\langle p_j(h),\tau\rangle\neq0$ for most $h\in\Z^m$, in the sense that the set of zeros $R\coloneqq \{h\in\Z^m: \langle p_j(h),\tau\rangle=0\}$ satisfies
$$
\otherlim{C}{h_m\to\infty} \cdots\otherlim{C}{h_1\to\infty}\1_{R}(h_1,\ldots,h_m) =0.
$$
On the other hand, for each $h\in\Z^m\setminus R$, the function
$
n\mapsto\big\langle p(n,h),~\tau\big\rangle
$
belongs to $\F$ and hence, using \cref{thm_tempereduniformalongW}, we get
$$
\otherlim{UW}{n\to\infty} e\left(
\big\langle  p(n,h),~\tau\big\rangle
\right) = 0.
$$
Therefore
\begin{eqnarray*}
\Big|\otherlim{C}{h_m\to\infty} \cdots\otherlim{C}{h_1\to\infty}
\otherlim{UW}{n\to\infty} \chi\big(\pi_1\big(p(n,h)\big)\big)\Big|
&=&
\Big|\otherlim{C}{h_m\to\infty} \cdots\otherlim{C}{h_1\to\infty}
\otherlim{UW}{n\to\infty} e\big(\langle p(n,h),\tau\rangle\big)\Big|
\\&\leq&
\otherlim{C}{h_m\to\infty} \cdots\otherlim{C}{h_1\to\infty}
\Big|\otherlim{UW}{n\to\infty} e\big(\langle p(n,h),\tau\rangle\big)\Big|
\\&\leq&
\otherlim{C}{h_m\to\infty} \cdots\otherlim{C}{h_1\to\infty}
1_R(h_1,\dots,h_m)
\\&=&0.
\end{eqnarray*}
\end{proof}

\subsection{Lifting equidistribution from the horizontal torus to the nilmanifold}

\begin{Definition}Let $G$ be a connected and simply connected nilpotent Lie group and let $\Gamma\subset G$ be a uniform and discrete subgroup.
  The \define{horizontal torus} of the nilmanifold $X:=G/\Gamma$ is the compact abelian group $G/(\Gamma[G,G])$, where $[G,G]$ denotes the commutator subgroup of $G$ (i.e. the subgroup of $G$ generated by all elements $[g,h]=g^{-1}h^{-1}gh$).

  A \define{horizontal character} is a map of the form $\eta\circ\pi$, where $\pi:X\to G/(\Gamma[G,G])$ is the projection onto the horizontal torus and $\eta$ is a character of $G/(\Gamma[G,G])$.

The following theorem is the main technical result of this paper. We will show that it implies \cref{thm:equidistribution-of-discrete-tempred-on-subnilmanifold}.

\end{Definition}
\begin{Theorem}
\label{thm:qualitative-green-leibman-theorem-for-discrete-tempered}
Let $G$ be a connected simply connected nilpotent Lie group, $\Gamma$ a uniform and discrete subgroup of $G$ and define $X\coloneqq G/\Gamma$.
Let $\ell,u\in \N$, let $m\in\Z$ with $m\geq0$, let $f\in\F_{\ell}$, let $a_{i,j}$ with $(i,j)\in\{0,1,\ldots,{\ell-1}\}\times\{1,\ldots,u\}$ be a collection of commuting elements of $G$ and, for $(i,j)\in\{0,1,\ldots,\ell-1\}\times\{1,\ldots,u\}$, let $p_{i,j}\in\R[x_1,\ldots x_m]$.
Let $g_0,g_1,\ldots,g_{\ell-1}\colon \Z^m\to G$ denote the polynomial sequences defined by
$$
g_i(h)\coloneqq a_{i,1}^{p_1(h)}a_{i,2}^{p_2(h)}\cdot\ldots\cdot a_{i,u}^{p_{i,u}(h)}, \qquad\forall h\in\Z^m,
$$
(or let $g_i\in G$ be commuting constants if $m=0$) and set
$$
g(n,h)\coloneqq g_0(h)^{f(n)}g_1(h)^{\Delta f(n)}\cdot\ldots\cdot g_{\ell-1}(h)^{\Delta^{\ell-1} f(n)}, \qquad\forall n\in\N,~\forall h\in\Z^{m}.
$$
Define $W\coloneqq \Delta^{\ell-1} f$.
If for all non-trivial horizontal characters $\eta$ one has
\begin{equation}
\label{eqn:qualitative-green-leibman-theorem-for-discrete-tempered-0}
\otherlimsup{C}{h_{m}\to\infty}\cdots \otherlimsup{C}{h_1\to\infty}\otherlimsup{UW}{n\to\infty}\eta\big(g(n,h)\Gamma\big)=0,
\end{equation}
then for all $F\in C(X)$ with $\int_XF\d\mu_X=0$, we have
\begin{equation}
\label{eqn:qualitative-green-leibman-theorem-for-discrete-tempered-0-0}
\otherlimsup{C}{h_{m}\to\infty}\cdots \otherlimsup{C}{h_1\to\infty}\otherlimsup{UW}{n\to\infty}F\big(g(n,h)\Gamma\big) =0,
\end{equation}
where $h=(h_1,\ldots,h_m)$.
\end{Theorem}

{
\paragraph{\textit{\textbf{Outline of the proof of \cref{thm:qualitative-green-leibman-theorem-for-discrete-tempered}:}}}
For the proof of \cref{thm:qualitative-green-leibman-theorem-for-discrete-tempered} we have to distinguish between the cases $\ell=1$ and $\ell\geq 2$. The case $\ell=1$ is easy and taken care of in \cref{sec_case_l_one}. For the case $\ell\geq 2$, which is proved in \cref{sec_case_l_general}, we use induction on the degree of $g(n,h)$ (see \cref{def_degree_0}). The base case of this induction follows immediately from \cref{lem:equidistribution-of-discrete-tempred-on-subtorus}.
For the proof of the inductive step we use a variant of the classical van der Corput Lemma, which says that instead of \eqref{eqn:qualitative-green-leibman-theorem-for-discrete-tempered-0-0} it suffices to show
\begin{equation}
\label{eqn:qualitative-green-leibman-theorem-for-discrete-tempered-0-0-0}
\otherlimsup{C}{h_{m+1}\to\infty}\otherlimsup{C}{h_{m}\to\infty}\cdots \otherlimsup{C}{h_1\to\infty}\otherlimsup{UW}{n\to\infty}
F\big(g(n+h_{m+1},h)\Gamma\big)
\overline{F}\big(g(n,h)\Gamma\big) = 0.
\end{equation}
We can also assume without loss of generality that $F$ satisfies
\begin{equation}
\label{eqn:qualitative-green-leibman-theorem-for-discrete-tempered-2-0-0}
F(t x)=\chi(t)F(x), \qquad\forall t\in Z(G),~\forall x\in X,
\end{equation}
where $Z(G)$ is the center of $G$ and $\chi\colon Z(G)\to\{w\in\C: |w|=1\}$ is a continuous group character of $Z(G)$. Indeed, the algebra generated by functions satisfying \eqref{eqn:qualitative-green-leibman-theorem-for-discrete-tempered-2-0-0} is uniformly dense in $C(X)$.
Using ideas form \cite{Green_Tao12a}, one can show that the projection of the sequence $(g(n+h_{m+1},h),g(n,h))$ onto the quotient group $(G\times G)/Z(G)^\triangle$, where $Z(G)^\triangle=\{(g,g): g\in Z(G)\}$, has degree strictly smaller than of $g(n,h)$. Moreover, due to \eqref{eqn:qualitative-green-leibman-theorem-for-discrete-tempered-2-0-0}, the function $F\otimes \overline{F}$ is invariant under $Z(G)^\triangle$. It thus follows from the induction hypothesis that
\begin{equation*}
\otherlimsup{C}{h_{m+1}\to\infty}\otherlimsup{C}{h_{m}\to\infty}\cdots \otherlimsup{C}{h_1\to\infty}\otherlimsup{UW}{n\to\infty}
F\big(g(n+h_{m+1},h)\Gamma\big)
\overline{F}\big(g(n,h)\Gamma\big) = \int_Y F\otimes \overline{F} \d\mu_Y,
\end{equation*}
where $Y$ is a certain sub-nilmanifold of $X\times X$ which we describe explicitly in our proof. By judiciously using \eqref{eqn:qualitative-green-leibman-theorem-for-discrete-tempered-2-0-0} one can then show that $\int_Y F\otimes \overline{F} \d\mu_Y=0$, which completes the proof.
\vspace{1 em}
}

{We end this subsection by showing that \cref{thm:qualitative-green-leibman-theorem-for-discrete-tempered} implies \cref{thm:equidistribution-of-discrete-tempred-on-subnilmanifold}.
}
\begin{proof}[Proof that \cref{thm:qualitative-green-leibman-theorem-for-discrete-tempered} implies \cref{thm:equidistribution-of-discrete-tempred-on-subnilmanifold}]
We assume that $m>0$, the case $m=0$ can be proved in the same way.
We proceed by induction of the dimension $d$ of $G$.
If $d=1$ then $G=\R$, in which case \cref{thm:equidistribution-of-discrete-tempred-on-subnilmanifold} follows from \cref{lem:equidistribution-of-discrete-tempred-on-subtorus}.
Assume therefore that $d>1$ and that \cref{thm:equidistribution-of-discrete-tempred-on-subnilmanifold} has already been established for all connected simply connected nilpotent Lie groups $G'$ with dimension $d'$ smaller than $d$.

Recall that $Y$ is defined as
$$
Y\coloneqq\overline{\{g_0(h)^{t_0}g_1(h)^{t_1}\cdot\ldots\cdot g_{\ell-1}(h)^{t_{\ell-1}}\Gamma: h\in\Z^m, (t_0,t_1,\ldots,t_{\ell-1})\in\R^{\ell}\}}.
$$
Let $\pi\colon G\to X$ denote the natural projection from $G$ onto $X$ and, using \cref{lem:orbit-closure-is-subnilmanifold}, choose a closed, connected and \define{rational}\footnote{A closed subgroup $H$ of $G$ is called \define{rational} if the set $H\Gamma$ is a closed subset of $G$ (cf. \cite{Leibman06}).} subgroup $H$ of $G$ such that $Y=\pi(H)$.
Let $\vartheta\colon G\to G/[G,G]$ denote the natural projection of $G$ onto $G/[G,G]$ and define $\tilde G\coloneqq \vartheta(G)= G/[G,G]$.
If $\vartheta(H)$ is a proper subgroup of $\tilde{G}$, then $G'\coloneqq \vartheta^{-1}(\vartheta(H)) = H\cdot [G,G]$ is a closed, normal connected and rational subgroup of $G$ with $\dim (G')<\dim (G)$ and $g_0(h)^{t_0}g_1(h)^{t_1}\cdot\ldots\cdot g_{\ell-1}(h)^{t_{\ell-1}}\in G'$ for all $h\in\Z^m$ and $(t_0,t_1,\ldots,t_{\ell-1})\in\R^{\ell}$.
In this case we can replace $G$ with $G'$ and \cref{thm:equidistribution-of-discrete-tempred-on-subnilmanifold} follows from the induction hypothesis.

Let us therefore assume that $\vartheta(H)=\tilde{G}$.
Define $\tilde\Gamma\coloneqq \vartheta(\Gamma)$, $T_1\coloneqq\tilde G/\tilde\Gamma$ and, for $i=0,1,\ldots,\ell-1$, let $p_i(h)\coloneqq \vartheta(g_i(h))$.
Set
$$
p(n,h)\coloneqq p_0(h){f(n)}+p_1(h){\Delta f(n)}+\ldots+ p_{\ell-1}(h){\Delta^{\ell-1} f(n)},\qquad\forall n\in\N,~\forall h\in\Z^{m}.
$$
Note that $\tilde G\cong \R^{d'}$ and $T_1\cong \T^{d'}$ for some $d'\in\N$.
Let $\pi_1\colon \tilde G\to T_1$ denote the natural factor map from $\tilde G$ onto $T_1$.
Fix a non-trivial horizontal character $\eta\colon X\to\{w\in\C:|w|=1\}$.
Since $[G,G]\cdot\Gamma$ belongs to the kernel of $\eta$, there exists a non-trivial group character $\chi\colon T_1\to \{w\in\C:|w|=1\}$ such that
$$
\chi(\pi_1(\vartheta(a)))=\eta(a\Gamma),\qquad\forall a\in G.
$$
Since $\vartheta(H)=\tilde{G}$, the set
$$
\big\{\pi_1\big(p_0(h){t_0}+p_1(h){t_1}+\ldots+ p_{\ell-1}(h){t_{\ell-1}}\big): h\in\Z^m, (t_0,t_1,\ldots,t_{\ell-1})\in\R^{\ell}\big\}
$$
is dense in $T_1$.
It thus follows from \cref{lem:equidistribution-of-discrete-tempred-on-subtorus} that
$$
\otherlim{C}{h_{m}\to\infty}\cdots \otherlim{C}{h_1\to\infty}\otherlim{UW}{n\to\infty}\chi\big(\pi_1(p(n,h))\big)=\int_{T_1} \chi\d\mu_{T_1}=0.
$$
This implies that
$$
\otherlim{C}{h_{m}\to\infty}\cdots \otherlim{C}{h_1\to\infty}\otherlim{UW}{n\to\infty}\eta\big(g(n,h)\Gamma\big)=0
$$
and therefore, using \cref{thm:qualitative-green-leibman-theorem-for-discrete-tempered}, we obtain
$$
\otherlimsup{C}{h_{m}\to\infty}\cdots \otherlimsup{C}{h_1\to\infty}\otherlimsup{UW}{n\to\infty}F\big(g(n,h)\Gamma\big) = 0
$$
for all $F\in C(X)$ with $\int_X F\d\mu_X=0$.
It must therefore be the case that $Y=X$, which implies
$$
\otherlimsup{C}{h_{m}\to\infty}\cdots \otherlimsup{C}{h_1\to\infty}\otherlimsup{UW}{n\to\infty}F\big(g(n,h)\Gamma\big) = 0
$$
for all $F\in C(X)$ with $\int_Y F\d\mu_Y=0$ and the proof is completed.
\end{proof}

\subsection{Proof of \cref{thm:qualitative-green-leibman-theorem-for-discrete-tempered} for $\ell=1$}
\label{sec_case_l_one}
In this section we prove the case $\ell=1$ of \cref{thm:qualitative-green-leibman-theorem-for-discrete-tempered}.
\begin{Proposition}\label{prop_Fejerequidistributionnilmanifold}
Let $G$ be a connected and simply connected nilpotent Lie group, let $\Gamma\subset G$ be a uniform discrete subgroup and let $X\coloneqq G/\Gamma$.
Let $b\in G$ and let $\mu$ be the Haar measure on the subnilmanifold $\overline{\{b^t:t\in\R\}}$. Then for every $W\in\F_1$,
$$\otherlim{UW}{n\to\infty}\delta_{b^{W(n)}\Gamma}=\mu\qquad\text{in the weak}^*\text{ topology.}$$
\end{Proposition}
\begin{proof}
We need to show that for every $F\in C(X)$
\begin{equation}\label{eq_proof_prop_Fejerequidistributionnilmanifold}
\lim_{W(N)-W(M)\to\infty}\left|\frac1{W(N)-W(M)}\sum_{n=M}^N\Delta W(n)F(b^{W(n)}\Gamma)-\frac1{N-M}\int_M^NF(b^t\Gamma)\d t\right|=0.
\end{equation}
Fix $\epsilon>0$ and let $\eta=\eta(\epsilon)>0$ be such that whenever $t_1,t_2\in\R$ with $|t_1-t_2|<\eta$ then $|F(b^{t_1}\Gamma)-F(b^{t_2}\Gamma)|<\epsilon$.
We can thus replace $F(b^{W(n)}\Gamma)$ with $F(b^{\lfloor W(n)/\eta\rfloor\eta}\Gamma)$ at a cost of $\epsilon$.
Note that
\begin{equation*}
\begin{split}
\frac1{W(N)-W(M)}\sum_{n=M}^N & \Delta W(n)F(b^{\lfloor W(n)/\eta\rfloor\eta}\Gamma)\\
&=
\frac1{W(N)-W(M)}\sum_{k=\lfloor W(M)/\eta\rfloor}^{\lfloor W(N)/\eta\rfloor}~~\sum_{n:W(n)\in\big[k\eta,(k+1)\eta\big)}\Delta W(n)F(b^{k\eta}\Gamma).
\end{split}
\end{equation*}
Also, since $W\in \F_1$, we have
$$\lim_{k\to\infty}\sum_{n:W(n)\in\big[k\eta,(k+1)\eta\big)}\Delta W(n)=\eta.$$
Therefore
$$\limsup_{W(N)-W(M)\to\infty}\left|\tfrac1{W(N)-W(M)}\sum_{n=M}^N\Delta W(n)F(b^{W(n)}\Gamma)
-
\tfrac\eta{W(N)-W(M)}\sum_{k=\lfloor W(M)/\eta\rfloor}^{\lfloor W(N)/\eta\rfloor}F(b^{k\eta}\Gamma)\right|<\epsilon.$$
Since $\epsilon$ is arbitrary, \eqref{eq_proof_prop_Fejerequidistributionnilmanifold} now follows from \cref{lemma_orbitdiscretecontinuous}.

\end{proof}

\begin{Lemma}[\cref{thm:qualitative-green-leibman-theorem-for-discrete-tempered} for the special case $\ell=1$]
\label{lem:qualitative-green-leibman-theorem-for-discrete-fejer}
Let $G$ be a connected simply connected nilpotent Lie group, $\Gamma\subset G$ a uniform and discrete subgroup and define $X\coloneqq G/\Gamma$.
Let $u,m\in \N$, let $W\in\F_{1}$, let $a_1,\ldots,a_u\in G$ be commuting and let $p_1,\ldots,p_u\in\R[x_1,\ldots x_m]$.
Let
$$
g(h)=a_{1}^{p_1(h)}\cdot\ldots\cdot a_{u}^{p_{u}(h)}\qquad\text{and }\qquad g(n,h)\coloneqq g(h)^{W(n)}, \qquad\forall n\in\N,~\forall h\in\Z^{m}.
$$
If for all non-trivial horizontal characters $\eta$ one has
\begin{equation*}
\otherlim{C}{h_{m}\to\infty}\cdots \otherlim{C}{h_1\to\infty}\otherlim{UW}{n\to\infty}\eta\big(g(n,h)\Gamma\big)=0,
\end{equation*}
then the points $g(n,h)\Gamma$, where $h=(h_1,\ldots,h_m)$, equidistribute as follows:
$$
\otherlim{C}{h_{m}\to\infty}\cdots \otherlim{C}{h_1\to\infty}\otherlim{UW}{n\to\infty}\delta_{g(n,h)\Gamma} = \mu_X.
$$
\end{Lemma}
\begin{proof}
For each $h\in\N^m$, let $Y_h\subset X$ be the subnilmanifold $Y_h\coloneqq \overline{\big\{g(h)^t\Gamma:t\in\R\big\}}\subset X$ and let $\mu_h$ be the Haar measure on $Y_h$.
In view of \cref{prop_Fejerequidistributionnilmanifold},
  \begin{equation}\label{eq_proof_lemma_basecase}
  \otherlim{UW}{n\to\infty}\delta_{g(n,h)\Gamma}=\mu_h.
  \end{equation}
On the other hand, \cref{lemma_orbitdiscretecontinuous} implies that there exists a time $t\in\R$ such that
$\otherlim{C}{n\to\infty}\delta_{g(h)^{tn}\Gamma}=\mu_h$, for every $h\in\N^m$, so we need to show that
$$
\otherlim{C}{h_{m}\to\infty}\cdots \otherlim{C}{h_1\to\infty}\otherlim{C}{n\to\infty}\delta_{g(h)^{tn}\Gamma} = \mu_X.
$$
Since the map $\tilde g\colon (h_1,\dots,h_m,n)\mapsto g(h)^{tn}$ is a polynomial sequence, the result now follows directly from \cref{thm:Leibman-thmB}.

\end{proof}

\subsection{Proof of \cref{thm:qualitative-green-leibman-theorem-for-discrete-tempered} for general $\ell$}
\label{sec_case_l_general}

\begin{Definition}
Let $s\in\N$ and let $G_1,\ldots,G_s,G_{s+1}$ be subgroups of a nilpotent Lie group $G$. We call $G_\bullet\coloneqq \{G_1,\ldots,G_s,G_{s+1}\}$ an \define{$s$-step pre-filtration} of $G$ if
$$
[G_i,G_j]\subset G_{i+j},\qquad \forall j,i=\{1,\ldots,s\}~\text{with}~i+j\leq s+1,
$$
and
$
G_{s+1}=\{1_G\}.
$
A pre-filtration $G_\bullet$ is called a \define{filtration} if $G_1=G$. A \define{normal pre-filtration} is a pre-filtration consisting only of normal subgroups.
\end{Definition}

Note that for any connected simply connected nilpotent Lie group $G$ there exists a right-invariant metric $d_G\colon G\times G\to [0,\infty)$ on $G$. For any uniform and discrete subgroup $\Gamma$ the metric $d_G$ descends to a metric $d_{X}$ on $X\coloneqq G/\Gamma$ via
\begin{equation}
\label{eqn:def-of-metric}
d_{X}(x\Gamma,y\Gamma)\coloneqq \inf\{d_G(x\gamma,y\gamma'):\gamma,\gamma'\in\Gamma\}.
\end{equation}
Given a subset $S\subset G$ and a point $g\in G$ we denote $d(g,S)=\inf_{s\in S}d_G(g,s)$.

\begin{Definition}
\label{def_degree_0}
Let $\ell\in\N$, let $f\in\F_{\ell}$, let $g_0,g_1,\ldots,g_{\ell-1}\colon \Z^m\to G$ be polynomial sequences and set
$$
g(n,h)\coloneqq g_0(h)^{f(n)}g_1(h)^{\Delta f(n)}\cdot\ldots\cdot g_{\ell-1}(h)^{\Delta^{\ell-1} f(n)}, \qquad\forall n\in\N,~\forall h\in\Z^{m}.
$$
For every $n_1\in\N$ let $\Delta_{n_1}g(n,h)$ denote the \define{discrete derivative of $g(n,h)$ in direction $n_1$}, that is, $\Delta_{n_1}g(n,h)\coloneqq g(n+n_1,h)g(n,h)^{-1}$.
We define the degree of $g(\cdot,h)$ to be the smallest number $s\in\N$ such that there exists a $s$-step normal pre-filtration $G_\bullet =\{G_1,G_2,\ldots, G_{s},G_{s+1}=\{1_G\}\}$ with the property that for every fixed $h\in\N^m$ we have
$$
\lim_{n\to\infty}d_G\Big(g(n,h),G_1\Big)=0
$$
and, for all $j\in\{1,\ldots,s\}$ and all $n_1,\ldots n_j\in\N$,
$$
\lim_{n\to\infty}d_G\Big(\Delta_{n_j}\cdots \Delta_{n_1} g(n,h),G_{j+1}\Big)=0.
$$
In this case we say $G_\bullet$ is a normal pre-filtration of $G$ that \define{realizes} the step of $g(n,h)$.
If there exists no such filtration, then we say that $g\colon\N\times\Z^m\to G$ has infinite degree.
\end{Definition}

\begin{Lemma}
\label{lem_degree}
Let $\ell\in\N$, let $f\in\F_{\ell}$ and $g_0,g_1,\ldots,g_{\ell-1}\colon \Z^m\to G$ be as in \cref{thm:qualitative-green-leibman-theorem-for-discrete-tempered} and define
$$
g(n,h)\coloneqq g_0(h)^{f(n)}g_1(h)^{\Delta f(n)}\cdot\ldots\cdot g_{\ell-1}(h)^{\Delta^{\ell-1} f(n)}, \qquad\forall n\in\N,~\forall h\in\Z^{m}.
$$
Let $s$ denote the nilpotency step of $G$. Then $g\colon\N\times\Z^m\to G$ has finite degree (in fact, the degree is smaller or equal to $\ell(s+1)$).
\end{Lemma}

\begin{proof}
Let $C_\bullet \coloneqq  \{G=C_1,C_2,\ldots, C_{s},C_{s+1}=\{1_G\}\}$ denote the lower central series of $G$. Define a new filtration $G_\bullet=\{G_1, G_2,\ldots,G_{\ell(s+1)},G_{\ell(s+1)+1}=\{1_G\}\}$ by setting $G_{(j-1)\ell+i}\coloneqq C_j$ for all $j\in\{1,\ldots,s+1\}$ and $i\in\{1,\ldots,\ell\}$.

Using the fact that the polynomial sequences $g_0,\ldots,g_{\ell-1}$ commute we see that for all $n_1,\ldots n_j\in\N$ we have
$$
\Delta_{n_j}\cdots \Delta_{n_1} g(n,h)=
g_0(h)^{\Delta_{n_j}\cdots \Delta_{n_1}f(n)}g_1(h)^{\Delta_{n_j}\cdots \Delta_{n_1}\Delta f(n)}\cdot\ldots\cdot g_{\ell-1}(h)^{\Delta_{n_j}\cdots \Delta_{n_1}\Delta^{\ell-1} f(n)}.
$$
Since $f\in\F_\ell$ it follows that for all $j\geq \ell$ and all $i\geq 0$ the expression $\Delta_{n_j}\cdots \Delta_{n_1} \Delta^i f(n)$ converges to $0$ as $n\to\infty$. This implies that as long as $j\geq\ell$ then for all choices of $n_1,\ldots,n_j$ the sequence $\Delta_{n_j}\cdots \Delta_{n_1} g(n,h)$ converges to the identity $1_G$ of $G$.
Using this observation it is now straightforward to check that $G_\bullet$ is a normal filtration with the property that for all $h\in\N^m$, all $j\in\{1,\ldots,\ell(s+1)\}$ and all $n_1,\ldots, n_j\in\N$,
\begin{equation}
\label{eqn_rev_eq_1}
\lim_{n\to\infty}
d_G\Big(\Delta_{n_j}\cdots \Delta_{n_1} g(n,h),~G_{j+1}\Big)=0.
\end{equation}
Indeed, if $j<\ell$ then $G_{j+1}$ equals the entire group $G$ and so \eqref{eqn_rev_eq_1} holds trivially, whereas if $j\geq \ell$ then \eqref{eqn_rev_eq_1} holds because $\Delta_{n_j}\cdots \Delta_{n_1} g(n,h)$ converges to $1_G$.
\end{proof}

\begin{proof}[Proof of \cref{thm:qualitative-green-leibman-theorem-for-discrete-tempered}]
We proceed by induction on the degree of $g(n,h)$, which is finite according to \cref{lem_degree}. If the degree of $g(n,h)$ equals $1$ then $G$ must be abelian, in which case the claim follows from \cref{lem:equidistribution-of-discrete-tempred-on-subtorus}.
Let us therefore assume that the degree of $g(n,h)$ is equal to some $s\geq 2$ and that \cref{thm:qualitative-green-leibman-theorem-for-discrete-tempered} has already been proven for all systems $(\tilde{G},\tilde{\Gamma},f,\tilde{g}_0,\tilde{g}_1,\ldots,\tilde{g}_\ell)$ where $\tilde{g}(n,h')= \tilde{g}_0(h')^{f(n)}\tilde{g}_1(h')^{\Delta f(n)}\cdot\ldots\cdot \tilde{g}_{\ell-1}(h')^{\Delta^{\ell-1} f(n)}$ has degree smaller than $s$.

Define
\begin{equation}
\label{eqn:def-of-subnilmanifold-Z}
Z\coloneqq\overline{\{g_0(h)^{t_0}g_1(h)^{t_1}\cdot\ldots\cdot g_{\ell-2}(h)^{t_{\ell-2}}\Gamma: h\in\Z^m, (t_0,t_1,\ldots,t_{\ell-2})\in\R^{\ell-1}\}}.
\end{equation}
If $Z=\{1_X\}$ then $g_0(h)=\ldots=g_{\ell-2}(h)=1_G$ for all $h\in\Z^m$, in which case \cref{thm:qualitative-green-leibman-theorem-for-discrete-tempered} follows from \cref{lem:qualitative-green-leibman-theorem-for-discrete-fejer}.
Thus we can assume that $Z\neq\{1_X\}$.

Invoking \cref{lem:orbit-closure-is-subnilmanifold} we can find a closed rational and connected subgroup $H$ of $G$ such that $\pi(H)=Z$.
Let $L$ denote the normal closure of $H$ in $G$, i.e., the smallest normal subgroup of $G$ containing $H$.
One can show that $L$ is also connected, simply connected, rational and closed (see \cite[p.\ 844]{Leibman10}).
Also, we remark that $L$ is unique in the sense that if $H'$ is another closed rational and connected subgroup of $G$ with $\pi(H')=Z$ then the normal closure of $H'$ coincides with $L$.
Let $L_\bullet\coloneqq  \{L=L_1,L_2,\ldots, L_{r},L_{r+1}=\{1_G\}\}$ denote the lower central series of $L$ and note that all elements in $L_\bullet$ are themselves connected, simply connected, rational, closed and normal subgroups of $G$ (cf.\ \cite{Raghunathan72}).

Let $C_\bullet\coloneqq  \{G=C_1,C_2,\ldots, C_{r},C_{r+1}=\{1_G\}\}$ denote the lower central series of $G$ and
let $j_0$ denote the biggest number in $\{1,\ldots,r\}$ such that $L\cap G_{j_0}\neq \{1_G\}$. Note that $j_0$ exists, because $L\neq\{1_G\}$.
Define $V\coloneqq L\cap G_{j_0}$.

For any $a\in G$ and any $t\in V$ the element $[a,t]=a^{-1}t^{-1}at$ belongs to $C_{j_0+1}$.
Since $L$ is normal and $t\in L$, we also have $[a,t]\in L$.
Therefore $[a,t]$ belongs to $L\cap C_{j_0+1}$, which implies that $[a,t]=1_G$. This proves that $V$ is a subgroup of the center $Z(G)$ of $G$.

Define $T\coloneqq Z(G)/(Z(G)\cap \Gamma)$. Note that $T$ is a torus (i.e., isomorphic to $\T^d$, where $d\coloneqq\dim(Z(G))$).
Also, since $Z(G)\cap \Gamma$ acts trivially on $X$, the action of $Z(G)$ on $X$ naturally descends to an action of $T$ on $X$.
Recall that our goal is to show
\begin{equation}
\label{eqn:qualitative-green-leibman-theorem-for-discrete-tempered-1}
\otherlimsup{C}{h_{m}\to\infty}\cdots \otherlimsup{C}{h_1\to\infty}\otherlimsup{UW}{n\to\infty}\left(F\big(g(n,h)\Gamma\big) -\int_X F\d\mu_X\right)= 0
\end{equation}
for all all $F\in C(X)$, where $h=(h_1,\ldots,h_m)$.
Note that any continuous group character $\chi$ of $T$ lifts to a continuous group homomorphism from $Z(G)$ to $\{z\in\C:|z|=1\}$ whose kernel contains $Z(G)\cap\Gamma$. By abuse of language we will use $\chi$ to denote both those maps.

Consider the class of functions $\phi\in C(X)$ with the property that
\begin{equation}
\label{eqn:qualitative-green-leibman-theorem-for-discrete-tempered-2}
\phi(t x)=\chi(t)\phi(x), \qquad\forall t\in Z(G),~\forall x\in X,
\end{equation}
for some continuous group character $\chi$ of $T$. This class of functions separates points in $X$.\footnote{
Let $x$ and $y$ be two distinct points in $X$. To show that functions satisfying \eqref{eqn:qualitative-green-leibman-theorem-for-discrete-tempered-2} separate these two points, we distinguish the cases $y\notin Tx$ and $y\in Tx$. If $y\notin Tx$ then $Tx\neq Ty$ and so there exists a continuous function $\phi'\in C(T\backslash X)$ with $\phi'(Tx)\neq \phi'(Ty)$. Then $\phi'$ lifts to a continuous and $T$-invariant function $\phi\in C(X)$ with $\phi(x)\neq \phi(y)$. On the other hand, if $y\in Tx$ then there is $t_0\in T$, which is not the identity, such that $y=t_0x$. Let $\chi_0$ be any group character of $T$ with the property that $\chi_0(t_0)\neq 1$. Let $U$ be a small neighborhood of $x$ and let $\rho$ be a continuous function on $X$ with the property that $\rho(x)=1$ and $\rho(z)=0$ for all $z\notin U$. Now define $\phi(z)\coloneqq \int_T \rho(sz)\chi_0(s) \d\mu_{T}(s)$, where $\mu_T$ is the normalized Haar measure on $T$. It is straightforward to check that $\phi$ is a non-zero and continuous function on $X$ satisfying $\phi(t z)=\chi_0(t)\phi(z)$ for all $t\in T$ and $z\in X$. In particular $\phi(y)=\phi(t_0x)=\chi_0(t_0)\phi(x)$ is not equal to $\phi(x)$.}
Thus, by the Stone-Weierstrass theorem, the span of such functions is dense in $C(X)$. Therefore, to prove \eqref{eqn:qualitative-green-leibman-theorem-for-discrete-tempered-1}, it suffices to show
\begin{equation}
\label{eqn:qualitative-green-leibman-theorem-for-discrete-tempered-3}
\otherlimsup{C}{h_{m}\to\infty}\cdots \otherlimsup{C}{h_1\to\infty}\otherlimsup{UW}{n\to\infty} \left(\phi\big(g(n,h)\Gamma\big)  - \int_X \phi\d\mu_X\right)=0
\end{equation}
for all functions $\phi$ of this kind.

Fix $\phi\in C(X)$ and a continuous group character $\chi$ of $T$ such that \eqref{eqn:qualitative-green-leibman-theorem-for-discrete-tempered-2} is satisfied.
If $\chi$ is trivial when restricted to $V$, then, after we mod out $G$ by $V$, we have reduced the question to a nilpotent Lie group of smaller dimension. Hence, by induction on the dimension of $G$, we can assume without loss of generality that $\chi$ is non-trivial when restricted to $V$.
This also implies that $\int_X \phi\d\mu_X=0$ and hence \eqref{eqn:qualitative-green-leibman-theorem-for-discrete-tempered-3} becomes
\begin{equation}
\label{eqn:qualitative-green-leibman-theorem-for-discrete-tempered-4}
\otherlimsup{C}{h_{m}\to\infty}\cdots \otherlimsup{C}{h_1\to\infty}\otherlimsup{UW}{n\to\infty} \phi\big(g(n,h)\Gamma\big)  = 0.
\end{equation}

Using a version of van der Corput's lemma proved in the appendix (see \cref{lem:C_lim-W_lim-vdC}) we see that instead of \eqref{eqn:qualitative-green-leibman-theorem-for-discrete-tempered-4} it suffices to show
\begin{equation}
\label{eqn:qualitative-green-leibman-theorem-for-discrete-tempered-5}
\otherlimsup{C}{h_{m+1}\to\infty}\otherlimsup{C}{h_{m}\to\infty}\cdots \otherlimsup{C}{h_1\to\infty}\otherlimsup{UW}{n\to\infty}
\phi\big(g(n+h_{m+1},h)\Gamma\big)
\overline{\phi}\big(g(n,h)\Gamma\big) = 0.
\end{equation}
Using \cref{lem_TaylorLeibnitz} we can rewrite $\Delta^j f(n+h_{m+1})$ as
\begin{eqnarray*}
\Delta^j f(n+h_{m+1})
&=& \sum_{i=0}^{h_{m+1}} \binom{h_{m+1}}{i} \Delta^{i+j} f(n)
\\
&=& \sum_{i=j}^{\ell-1} \binom{h_{m+1}}{i-j} \Delta^{i} f(n)
+ \epsilon_{j}(n),
\end{eqnarray*}
where $\epsilon_{j}(n)\coloneqq\sum_{i=\ell}^{h_{m+1}} \binom{h_{m+1}}{i} \Delta^{i+j} f(n)$.
Define $\epsilon(n)\coloneqq \prod_{j=0}^\ell g_j(h_1,\ldots,h_m)^{\epsilon_{j}(n)}$.
Using the fact that the polynomial sequences $g_0,\ldots,g_{\ell-1}$ commute, we obtain
\begin{equation}
\label{eqn:qualitative-green-leibman-theorem-for-discrete-tempered-6}
\begin{split}
g(n+h_{m+1},h)
&~=~
\prod_{j=0}^{\ell-1} g_j(h)^{\Delta^jf(n+h_{m+1})}
\\
&~=~
\prod_{j=0}^{\ell-1} \prod_{i=j}^{\ell-1} g_j(h)^{\binom{h_{m+1}}{i-j}\Delta^if(n)+\epsilon_j(n)}
\\
&~=~
\epsilon(n)\prod_{i=0}^{\ell-1}\left(\prod_{j=0}^{i}
g_j(h)^{\binom{h_{m+1}}{i-j}} \right)^{\Delta^{i} f(n)}
\\
&~=~\epsilon(n)\Big(g_1'(h')^{\Delta f(n)}\cdot\ldots\cdot g_{\ell-1}'(h')^{\Delta^{\ell-1}f(n)}\Big)
\Big(g_0(h)^{f(n)}\cdot\ldots\cdot g_{\ell-1}(h)^{\Delta^{\ell-1} f(n)}\Big),
\end{split}
\end{equation}
where $h=(h_1,\ldots,h_m)$, $h'=(h_1,\ldots,h_m,h_{m+1})$ and
\begin{equation}
\label{eqn:qualitative-green-leibman-theorem-for-discrete-tempered-6.4}
g_i'(h')\coloneqq
\prod_{j=0}^{i-1}
g_j(h)^{\binom{h_{m+1}}{i-j}}\qquad\text{ for all } i=0,\dots,\ell-1,
\end{equation}
where the empty product defining $g_0'(h')$ is to be understood as $1_G$.
Define
$$
g'(n,h')\coloneqq g_1'(h')^{\Delta f(n)}\cdot\ldots\cdot g_{\ell-1}'(h')^{\Delta^{\ell-1}f(n)}\qquad\forall n\in\N,~\forall h'\in\Z^{m+1}.
$$

Thus from \eqref{eqn:qualitative-green-leibman-theorem-for-discrete-tempered-6} we get
\begin{equation}
\label{eqn:qualitative-green-leibman-theorem-for-discrete-tempered-6.5}
g(n+h_{m+1},h)= \epsilon(n) g'(n,h') g(n,h).
\end{equation}
Using the definition of the metric $d_X$ given in \eqref{eqn:def-of-metric} and the right-invariance of $d_G$ we see that
\begin{equation*}
\begin{split}
d_X\Big(\epsilon(n) g'(n,h') g(n,h)\Gamma & ,~ g'(n,h') g(n,h)\Gamma\Big)
\\
&= \inf_{\gamma,\gamma'\in\Gamma}
d_G\Big(\epsilon(n) g'(n,h') g(n,h)\gamma,~ g'(n,h') g(n,h)\gamma'\Big)
\\
&\leq d_G\Big(\epsilon(n) g'(n,h') g(n,h),~ g'(n,h') g(n,h)\Big)
\\
&=  d_G\left(\epsilon(n),~ 1_G\right).
\end{split}
\end{equation*}
Since for fixed $h_1,\ldots,h_{m+1}\in \N$ the error terms $\epsilon_{i}(n)$ converge to $0$ as $n\to\infty$ for all $i\in\{0,1,\ldots,\ell-1\}$, it follows that
$$
\lim_{n\to\infty}d_G\left(\epsilon(n),~ 1_G\right)=0
$$
and hence
\begin{equation}
\label{eqn:qualitative-green-leibman-theorem-for-discrete-tempered-7}
\lim_{n\to\infty}
d_X\Big(\epsilon(n) g'(n,h') g(n,h)\Gamma,~ g'(n,h') g(n,h)\Gamma\Big)=0.
\end{equation}
Combining \eqref{eqn:qualitative-green-leibman-theorem-for-discrete-tempered-6.5} and \eqref{eqn:qualitative-green-leibman-theorem-for-discrete-tempered-7} and using the fact that $\phi$ is continuous, we obtain
$$
\lim_{n\to\infty} \left|\phi\big(g(n+h_{m+1},h)\Gamma\big)-
\phi\big(g'(n,h')g(n,h)\Gamma\big)\right| =0.
$$
Hence \eqref{eqn:qualitative-green-leibman-theorem-for-discrete-tempered-5} is equivalent to
\begin{equation}
\label{eqn:qualitative-green-leibman-theorem-for-discrete-tempered-8}
\otherlimsup{C}{h_{m+1}\to\infty}\otherlimsup{C}{h_{m}\to\infty}\cdots \otherlimsup{C}{h_1\to\infty}\otherlimsup{UW}{n\to\infty}
\phi\big(g'(n,h')g(n,h)\Gamma\big)
\overline{\phi}\big(g(n,h)\Gamma\big) = 0,
\end{equation}
where $h=(h_1,\ldots,h_m)$ and $h'=(h_1,\ldots,h_m,h_{m+1})$.

Let $Z(G)^\triangle\coloneqq\{(a,a):a\in Z(G)\}$ and denote by $\sigma\colon G\times G\to (G\times G)/Z(G)^\triangle$ the natural projection of $G\times G$ onto $(G\times G)/Z(G)^\triangle$.
Define
\begin{eqnarray*}
G^\triangle &\coloneqq & \{(a,a):a\in G\};
\\
K&\coloneqq & G^\triangle\cdot (L\times L);
\\
\Gamma_K & \coloneqq & K\cap (\Gamma\times\Gamma)
\\
\tilde{G}&\coloneqq & \sigma(K);
\\
\tilde{\Gamma}&\coloneqq & \sigma\big(\Gamma_K\big);
\\
\tilde{X}&\coloneqq &\tilde{G}/\tilde{\Gamma};
\\
(\text{for $0\leq i \leq \ell-1$})\qquad\qquad
\tilde{g}_i(h')&\coloneqq &  \sigma\big(g'_i(h')g_i(h),g_i(h)\big);
\\
\tilde{g}(n,h')&\coloneqq &  \tilde{g}_0(h')^{f(n)}\tilde{g}_1(h')^{\Delta f(n)}\cdot\ldots\cdot \tilde{g}_{\ell-1}(h')^{\Delta^{\ell-1} f(n)}.
\end{eqnarray*}
Also, it follows from \eqref{eqn:qualitative-green-leibman-theorem-for-discrete-tempered-2} that $\phi\otimes \overline{\phi}\colon X\times X\to\C$ is invariant under the action of $Z(G)^\triangle$.
Therefore there exists a continuous function $\tilde{\phi}\colon\tilde{X}\to\C$ with the property that
$$
(\phi\otimes\overline{\phi})(a\Gamma,b\Gamma)=\tilde{\phi}(\sigma(a,b)\tilde{\Gamma}),\qquad\forall (a,b)\in K.
$$
One can check that \eqref{eqn:qualitative-green-leibman-theorem-for-discrete-tempered-8} is equivalent to
\begin{equation}
\label{eqn:qualitative-green-leibman-theorem-for-discrete-tempered-9}
\otherlimsup{C}{h_{m+1}\to\infty}\otherlimsup{C}{h_{m}\to\infty}\cdots \otherlimsup{C}{h_1\to\infty}\otherlimsup{UW}{n\to\infty}
\tilde{\phi}\big(\tilde{g}(n,h')\tilde{\Gamma}\big)= 0.
\end{equation}

We now make three claims:

\paragraph{\textsc{Claim 1:}} We have $\int_{\tilde{X}}\tilde{\phi}\d\mu_{\tilde{X}}=0$.

\paragraph{\textsc{Claim 2:}} For all non-trivial horizontal characters $\tilde{\eta}$ of $\tilde{G}/\tilde{\Gamma}$ we have
\begin{equation}
\label{eqn:qualitative-green-leibman-theorem-for-discrete-tempered-9.5}
\otherlimsup{C}{h_{m+1}\to\infty}\otherlimsup{C}{h_{m}\to\infty}\cdots \otherlimsup{C}{h_1\to\infty}\otherlimsup{UW}{n\to\infty}\tilde{\eta}\big(\tilde{g}(n,h')\tilde{\Gamma}\big)=0.
\end{equation}
\paragraph{\textsc{Claim 3:}}
The sequence $\tilde{g}(n,h')$ has degree smaller than the sequence $g(n,h)$.

~

Note that once Claims 1, 2 and 3 have been proven, the proof of \cref{thm:qualitative-green-leibman-theorem-for-discrete-tempered} is completed. Indeed, Claims 2 and 3 allow us to invoke the induction hypothesis and apply \cref{thm:qualitative-green-leibman-theorem-for-discrete-tempered} to the system $(\tilde{G},\tilde{\Gamma},f,\tilde{g}_0,\tilde{g}_1,\ldots,\tilde{g}_\ell)$ in order to obtain the identity
\begin{equation}
\label{eqn:qualitative-green-leibman-theorem-for-discrete-tempered-10}
\otherlimsup{C}{h_{m+1}\to\infty}\otherlimsup{C}{h_{m}\to\infty}\cdots \otherlimsup{C}{h_1\to\infty}\otherlimsup{W}{n\to\infty}
\left(\tilde{\phi}\big(\tilde{g}(n,h')\tilde{\Gamma}\big)- \int_{\tilde{X}}\tilde{\phi}\d\mu_{\tilde{X}}\right)=0.
\end{equation}
However, due to Claim 1 we have that \eqref{eqn:qualitative-green-leibman-theorem-for-discrete-tempered-10} implies \eqref{eqn:qualitative-green-leibman-theorem-for-discrete-tempered-9}, which finishes the proof.

\paragraph{\textit{Proof of Claim 1:}}
Define $\tilde{V}\coloneqq \sigma(V\times\{1_G\})$ and note that $\tilde{V}$ and $V$ are isomorphic.
Also, from \eqref{eqn:qualitative-green-leibman-theorem-for-discrete-tempered-2} we can derive that for all $\tilde{t}\in \tilde{V}$ and $\tilde x\in\tilde X$ one has
$$
\tilde{\phi}(\tilde{t}\tilde{x})=\tilde{\chi}(\tilde{t})\tilde{\phi}(\tilde{x}),
$$
where $\tilde{\chi}$ is defined via $\tilde{\chi}(\sigma(t,1_G))\coloneqq\chi(t)$ for all $t\in V$.
Since $\chi$ is non trivial when restricted to $V$, and there is an isomorphism from $V$ to $\tilde{V}$ taking $\chi$ to $\tilde\chi$, we conclude that $\tilde{\chi}$ is also non-trivial.
Let $\tilde t\in\tilde V$ be such that $\tilde\chi(\tilde t)\neq1$.
Since $\tilde{V}$ is a subgroup of $\tilde{G}$, we have that $\mu_{\tilde{X}}$ is invariant under $\tilde t$, which implies that
$$
\int_{\tilde{X}}\tilde{\phi}(\tilde{x})\d\mu_{\tilde{X}}(\tilde{x})=
\int_{\tilde{X}}\tilde{\phi}(\tilde{t}\tilde{x})\d\mu_{\tilde{X}}(\tilde{x})=
\tilde{\chi}(\tilde{t})
\int_{\tilde{X}}\tilde{\phi}(\tilde{x})\d\mu_{\tilde{X}}(\tilde{x}).
$$
and hence $\int_{\tilde{X}}\tilde{\phi}\d\mu_{\tilde{X}}=0$ as claimed.

\paragraph{\textit{Proof of Claim 2:}}
Since $\tilde{g}(n,h')=\sigma\big(g'(n,h') g(n,h),g(n,h)\big)$ and for any horizontal character $\tilde{\eta}$ of $\tilde{G}/\tilde{\Gamma}$ there exists a horizontal character $\eta''$ of $K/\Gamma_K$ such that $\tilde{\eta}\circ \sigma=\eta ''$, it suffices to show that for all non-trivial horizontal characters $\eta''$ of $K/\Gamma_K$ we have
\begin{equation}
\label{eqn:qualitative-green-leibman-theorem-for-discrete-tempered-10.5}
\otherlimsup{C}{h_{m+1}\to\infty}\otherlimsup{C}{h_{m}\to\infty}\cdots \otherlimsup{C}{h_1\to\infty}\otherlimsup{UW}{n\to\infty}\eta''\Big(\big(g'(n,h') g(n,h),g(n,h)\big)\Gamma_K\Big)=0.
\end{equation}
Note that $[L,G]$ is a normal subgroup of $L$.
Let $\omega\colon G/[G,G]\times L/[L,G]\to K/[K,K]$ denote the map $\omega(a[G,G],b[L,G])= (ab,a)[K,K]$. The first step in proving \eqref{eqn:qualitative-green-leibman-theorem-for-discrete-tempered-10.5} is to show that $\omega$ is a well defined continuous and surjective homomorphism from $G/[G,G]\times L/[L,G]$ onto $K/[K,K]$.
It is straightforward to check that the map $(a,b)\mapsto (ab,a)[K,K]$ is a continuous and surjective group homomorphism.
So it only remains to show that $[G,G]\times[L,G]$ belongs to the kernel of the homomorphism $(a,b)\mapsto (ab,a)[K,K]$. This, however, follows immediately from \cref{lem:description-of-commutator-subgroup-of-K}.

Let $\vartheta_1\colon G\to G/[G,G]$ denote the natural factor map from $G$ onto $G/[G,G]$. Since $G/[G,G]$ is a connected, simply connected and abelian Lie group and $\vartheta_1(\Gamma)$ a uniform and discrete subgroup of $G/[G,G]$, there exists $d_1\in\N$ such that  $G/[G,G]$ is isomorphic to $\R^{d_1}$ and $\vartheta_1(\Gamma)$ is isomorphic to $\Z^{d_1}$.
Similarly, if $\vartheta_2\colon L\to L/[L,G]$ denotes the natural factor map from $L$ onto $L/[L,G]$ and $\vartheta_3\colon K\to K/[K,K]$ the natural factor map from $K$ onto $K/[K,K]$, then we can identify $L/[L,G]$ with $\R^{d_2}$ and $\vartheta_2(L\cap\Gamma)$ with $\Z^{d_2}$ as well as $K/[K,K]$ with $\R^{d_3}$ and $\vartheta_3(\Gamma_K)$ with $\Z^{d_3}$, for some numbers $d_2,d_3\in\N$.

Observe that $L\cdot\Gamma$ is a subgroup of $G$ and that $L$ is the connected component of the identity of $L\cdot\Gamma$. Also, since $\pi(H)=Z$ and $H\subset L$, we have that $\pi^{-1}(Z)\subset L\cdot\Gamma$ and hence, using the definition of $Z$ (see. \eqref{eqn:def-of-subnilmanifold-Z}), we conclude that the set $\{g_0(h)^{t_0}g_1(h)^{t_1}\cdot\ldots\cdot g_{\ell-2}(h)^{t_{\ell-2}}: h\in\Z^m, (t_0,t_1,\ldots,t_{\ell-2})\in\R^{\ell-1}\}$ is a subset of $L\cdot\Gamma$. However, the set $\{g_0(h)^{t_0}g_1(h)^{t_1}\cdot\ldots\cdot g_{\ell-2}(h)^{t_{\ell-2}}: h\in\Z^m, (t_0,t_1,\ldots,t_{\ell-2})\in\R^{\ell-1}\}$ is connected and contains the identity. It follows that
$$
g_i(h)^t\in L,\qquad \forall h\in\Z^{m},~\forall t\in\R,~\forall i\in\{0,1,\ldots,\ell-1\}
$$
and therefore, using \eqref{eqn:qualitative-green-leibman-theorem-for-discrete-tempered-6.4},
$$
g_i'(h')\in L,\qquad \forall h'\in\Z^{m},~\forall i\in\{0,1,\ldots,\ell-1\}.
$$

For $i=0,1,\ldots,\ell-1$, define $p_i(h)\coloneqq \vartheta_1(g_i(h))$, for $i=0,1,\ldots,\ell-2$, define $q_i(h)\coloneqq \vartheta_2(g_i(h))$ and, for $i=1,\ldots,\ell-1$, define $q_i'(h')\coloneqq \vartheta_2(g_i'(h'))$. It follows from \eqref{eqn:qualitative-green-leibman-theorem-for-discrete-tempered-6.4} (also cf.\ \eqref{eq:equidistribution-on-product-1}) that
$$
q_i'(h_1,\ldots,h_{m+1})\coloneqq \sum_{j=1}^i \binom{h_{m+1}}{j} q_{i-j}(h_1,\ldots,h_{m}).
$$
Set
$$
p(n,h)\coloneqq p_0(h){f(n)}+p_1(h){\Delta f(n)}+\ldots+ p_{\ell-1}(h){\Delta^{\ell-1} f(n)},\qquad\forall n\in\N,~\forall h\in\Z^{m}.
$$
and
$$
q'(n,h')\coloneqq q_1'(h'){\Delta f(n)}+\ldots+ q_{\ell-1}'(h'){\Delta^{\ell-1} f(n)},\qquad\forall n\in\N,~\forall h'\in\Z^{m+1}.
$$
Denote by $T_1$ the torus $(G/[G,G])/\vartheta_1(\Gamma)$ and let $\pi_1\colon G/[G,G]\to (G/[G,G])/\vartheta_1(\Gamma)$ denote the corresponding factor map.
Similarly, let $T_2$ and $T_3$ be the tori $T_2\coloneqq (L/[L,G])/\vartheta_2(\Gamma_L)$ and $T_3\coloneqq (K/[K,K])/\vartheta_3(\Gamma_K)$ and let $\pi_2$ and $\pi_3$ be the corresponding factor maps.
Using \cref{lem:normal-closure-as-product}, we can identify $T_2$ with $(H\cap[L,G])\backslash Z$. In particular, we have that
$$
\pi_2\big(\big\{q_0(h){t_0}+q_1(h){t_1}+\ldots+ q_{\ell-2}(h){t_{\ell-2}}): h\in\Z^m, (t_0,t_1,\ldots,t_{\ell-2})\in\R^{\ell-1}\big\}\big)
~\text{is dense in $T_2$}.
$$
Moreover, the hypothesis \eqref{eqn:qualitative-green-leibman-theorem-for-discrete-tempered-0} implies that
$$
\pi_1\big(\big\{p_0(h){t_0}+p_1(h){t_1}+\ldots+ p_{\ell-1}(h){t_{\ell-1}}: h\in\Z^m, (t_0,t_1,\ldots,t_{\ell-1})\in\R^{\ell}\big\}\big)~\text{is dense in $T_1$}.
$$
It is therefore guaranteed by \cref{lem:equidistribution-on-product} that
$$
(\pi_1\otimes \pi_2)\big(\big\{\tilde{p}_0(h'){t_0}+\ldots+ \tilde{p}_{\ell-1}(h'){t_{\ell-1}}: h'\in\Z^{m+1}, (t_0,\ldots,t_{\ell-1})\in\R^{\ell}\big\}\big)~\text{is dense in $T_1\times T_2$},
$$
where $\tilde{p}_0(h_1,\ldots,h_{m+1})\coloneqq(p_0(h_1,\ldots,h_m),0)$ and
$$
\tilde{p}_i(h_1,\ldots,h_{m+1})\coloneqq(p_i(h_1,\ldots,h_m),q_i'(h_1,\ldots,h_{m+1}))
$$
for $i=1,\ldots,\ell-1$.
Using \cref{lem:equidistribution-of-discrete-tempred-on-subtorus} we deduce that
\begin{equation}
\label{eqn:qualitative-green-leibman-theorem-for-discrete-tempered-11}
\otherlim{C}{h_{m+1}\to\infty}\otherlim{C}{h_{m}\to\infty}\cdots \otherlim{C}{h_1\to\infty}\otherlim{W}{n\to\infty}
F_1\big(\pi_1\big(p(n,h)\big)\big)F_2\big(\pi_2\big(q'(n,h')\big)\big)=\int_{T_1} F_1\d\mu_{T_1}\int_{T_2} F_2\d\mu_{T_2}
\end{equation}
for all $F_1\in C(T_1)$ and $F_2\in C(T_2)$.

Fix now a non-trivial horizontal characters $\eta''$ of $K/\Gamma_K$. Since $[K,K]\cdot\Gamma_K$ belongs to the kernel of $\eta''$, there exists a character $\chi''\colon T_3\to \{w\in\C:|w|=1\}$ such that $\eta''=\chi''\circ\pi_3\circ\vartheta_3$.
Next, using the fact that $\omega\colon G/[G,G]\times L/[L,G]\to K/[K,K]$ is a continuous and surjective homomorphism satisfying $\omega(\vartheta_1(\Gamma)\times \vartheta_2(\Gamma\cap L))\subset\vartheta_3(\Gamma_K)$, we can find two other characters $\chi\colon T_1\to \{w\in\C:|w|=1\}$ and $\chi'\colon T_2\to \{w\in\C:|w|=1\}$, such that
$$
(\chi''\circ\pi_3)\circ \omega =(\chi\circ\pi_1 )\otimes(\chi'\circ\pi_2) .
$$
Also, we have
$$
\vartheta_3 \big(g'(n,h') g(n,h),g(n,h)\big)=\omega(p(n,h), q'(n,h'));
$$
Hence \eqref{eqn:qualitative-green-leibman-theorem-for-discrete-tempered-10.5} follows from

\begin{equation}
\label{eqn:qualitative-green-leibman-theorem-for-discrete-tempered-12}
\otherlim{C}{h_{m+1}\to\infty}\otherlim{C}{h_{m}\to\infty}\cdots \otherlim{C}{h_1\to\infty}\otherlim{UW}{n\to\infty} \chi\big(\pi_1(p(n,h))\big)\chi'\big(\pi_2(q'(n,h'))\big)=0.
\end{equation}
But \eqref{eqn:qualitative-green-leibman-theorem-for-discrete-tempered-12} follows immediately from \eqref{eqn:qualitative-green-leibman-theorem-for-discrete-tempered-11}, which finishes the proof of Claim 2.

\paragraph{\textit{Proof of Claim 3:}}
Let $G_\bullet\coloneqq  \{G_1,G_2,\ldots, G_{s},G_{s+1}=\{1_G\}\}$ be a normal pre-filtration of $G$ that realizes the step of $g(n,h)$.
Define a new normal pre-filtration $K_\bullet= \{K_1,K_2,\ldots, K_{s},K_{s+1}=\{1_G\}\}$ via
$$
K_i\coloneqq K\cap\Big( G_i^\triangle\cdot (G_{i+1}\times G_{i+1})\Big)=G_i^\triangle\cdot\big((L\cap G_{i+1})\times(L\cap G_{i+1})\big).
$$
It is shown in \cite[Proposition 7.2]{Green_Tao12a} that $[K_i,K_j]\subset K_{i+j}$; in particular, $K_i$ is normal in $K$ and hence $K_\bullet$ is a normal pre-filtration. 

From \eqref{eqn:qualitative-green-leibman-theorem-for-discrete-tempered-6.5} we deduce that
$$
\epsilon(n)g'(n,h')=\Delta_{h_{m+1}}g(n,h)
$$
and hence
$$
\lim_{n\to\infty} \inf_{a\in G_2} d_G ( g'(n,h'),a)=0.
$$
Using the metric $d_K\big((a,b),(c,d)\big)\coloneqq d_G(a,c)+d_G(b,d)$, it follows that
\begin{eqnarray*}
\lim_{n\to\infty} & \displaystyle\inf_{(a,b)\in K_1} &d_{K} \Big(\big( g'(n,h')g(n,h),g(n,h)\big),(a,b)\Big)
\\&\leq&
\lim_{n\to\infty} \inf_{a\in G_2} d_{K} \Big(\big( g'(n,h')g(n,h),g(n,h)\big),\big(ag(n,h),g(n,h)\big)\Big)
\\&=&
\lim_{n\to\infty} \inf_{a\in G_2} d_{G} \big(g'(n,h'),a\big)
\\&=&0.
\end{eqnarray*}
A similar argument shows that
for all $j\in\{1,\ldots,s\}$ and all $n_1,\ldots n_j\in\N$,
$$
\lim_{n\to\infty} ~\inf_{(a,b)\in K_{j+1}}
d_K \big(\Delta_{n_j}\cdots \Delta_{n_1}(g'(n,h')g(n,h),g(n,h)),(a,b)\big)=0.
$$
To finish the proof of Claim 3 define $\tilde{G_i}\coloneqq \sigma(K_i)$. Since $K_s=G_s^\triangle$ and $G_s$ is a subset of $Z(G)$, we have $K_s\subset Z(G)^\triangle$. In particular, $\tilde{G}_s=\sigma(K_s)$ is trivial. Therefore $\tilde{G}_\bullet =\{\tilde{G}_1,\tilde{G}_2,\ldots, \tilde{G}_{s-1},\tilde{G}_{s}=\{1_{\tilde{G}}\}\}$ is a $(s-1)$-step normal pre-filtration of $\tilde{G}$ with the property that for every fixed $h'\in\N^{m+1}$ we have
$$
\lim_{n\to\infty}
d_{\tilde G}\Big(\tilde{g}(n,h'),~\tilde{G}_{1}\Big)=0,
$$
where $d_{\tilde G}(x,y)\coloneqq \inf\{d_K(a,b):a,b\in K, \sigma(a)=x,\sigma(b)=y\}$ and, for all $j\in\{1,\ldots,s-1\}$ and all $n_1,\ldots n_j\in\N$,
$$
\lim_{n\to\infty}
d_G\Big(\Delta_{n_j}\cdots \Delta_{n_1} \tilde{g}(n,h'),~\tilde{G}_{j+1}\Big)=0.
$$
This shows that the degree of $\tilde{g}(n,h')$ is smaller than or equal to $s-1$.
\end{proof}

\appendix
\section{Some auxiliary Lemmas}\label{sec_appendix}

\begin{Lemma}
\label{lem_nal}
If $f\colon[1,\infty
)\to\R$ is either a tempered function or
a function from a Hardy field satisfying $\lim_{x\to\infty }f^{(\ell)}(x)=\pm\infty$ and $\lim_{x\to\infty }f^{(\ell+1)}(x)=0$ for some $\ell\in \N\cup\{0\}$, then $f$ restricted to $\N$ belongs to $\F$.
\end{Lemma}

\begin{proof}
If we are in the case when $f$ is tempered, then by definition there exists $\ell\in\N\cup\{0\}$ such that $f$ is $(\ell+1)$-times continuously differentiable, $f^{(\ell+1)}(x)$ tends monotonically to $0$, and $\lim_{x\to\infty }x f^{(\ell+1)}(x)=\pm\infty$.
We can rewrite the last condition as $\lim_{x\to\infty }|f^{(\ell+1)}(x)|\succ1/x$, which implies $\lim_{x\to\infty } f^{(\ell)}(x)=\pm\infty$.
On the other hand, since $f^{(\ell+1)}(x)$ tends  to $0$,
$f^{(\ell)}(N)\sim \sum_{n=1}^N f^{(\ell+1)}(n)$, and hence we can conclude that $f^{(\ell+1)}$ restricted to $\N$ belongs to $\F_0$.
Since $\big(\Delta (f^{(\ell)})\big)(n)=f^{(\ell+1)}(n)\big(1+\oh_{n\to\infty}(1)\big)$, we conclude that $f^{(\ell)}$ restricted to $\N$ belongs to $\F_1$. Repeating this arguments shows that $f^{(i)}$ restricted to $\N$ belongs to $\F_{\ell+1-i}$; in particular $f$ restricted to $\N$ belongs to $\F_{\ell+1}\subset\F$.

If we are in the case where $f$ belongs to a Hardy field, then $f$ is also $(\ell+1)$-times continuously differentiable, $f^{(\ell+1)}(x)$ tends monotonically to $0$, and $\lim_{x\to\infty } f^{(\ell)}(x)=\pm\infty$. Hence, similar arguments can be used to show that $f$ restricted to $\N$ belongs to $\F_{\ell+1}$.
\end{proof}

\begin{Lemma}
\label{lem:normal-closure-as-product}
Let $G$ be a group, let $H$ be a subgroup of $G$ and let $L$ denote the normal closure of $H$ in $G$. Define $N\coloneqq [L,G]$. Then $L=H\cdot N$.
\end{Lemma}

\begin{proof}
Since $L$ is a normal subgroup of $G$, the group $N=[L,G]$ is a subgroup of $L$ and therefore $H\cdot N\subset L$.
To show that $H\cdot N\supset L$ it suffices to show that $H\cdot N$ is a normal subgroup of $G$, because any normal subgroup of $G$ containing $H$ must also contain $L$, because $L$ is the normal closure of $H$.

First, let us show that $H\cdot [L,G]$ is a group. Since $L$ is a normal subgroup of $G$, the commutator group $N=[L,G]$ is also a normal subgroup of $G$. It is then easy to check that the product $H\cdot N$ of a group $H$ and a normal subgroup $N$ is itself a group.

Finally, we show that $H\cdot N$ is normal. Let $g\in G$ and $h\in H$.
Then
$$
g h N g^{-1}~=~ hg g^{-1}h^{-1}gh N g^{-1}~=~ h g [g,h] N g^{-1}.
$$
Note that $[g,h]\in N$ and that $g N g^{-1}=N$. Therefore
$$
h g [g,h] N g^{-1}~=~ hN.
$$
It follows that for all $g\in G$ one has
$$
g (H\cdot N) g^{-1}=g\left(\bigcup_{h\in H} h N\right)g^{-1}=\bigcup_{h\in H}g h N g^{-1}=\bigcup_{h\in H} h N =H\cdot N.
$$
This proves that $H\cdot N$ is normal.
\end{proof}

\begin{Lemma}
\label{lem:description-of-commutator-subgroup-of-K}
Let $G$ be a group and denote by $G^\triangle$ the diagonal $\{(g,g):g\in G\}\subset G^2$.
Let $L$ be a normal subgroup of $G$ and define $K\coloneqq G^{\triangle}\cdot (L\times L)\subset G\times G$.
Then $[K,K]\supset [L,G]\times[L,G]$.
\end{Lemma}

\begin{proof}
We will show that $[L,G]\times \{1_K\}\subset [K,K]$. From this it will follow by symmetry that $\{1_K\}\times [L,G]\subset [K,K]$, which implies that $[L,G]\times[L,G]\subset [K,K]$.

$[L,G]$ is the group generated by the set $\{[l,g]:l\in L,g\in G\}$. It thus suffices to show that $([l,g],1)\in [K,K]$ for all $l\in L$ and $g\in G$.
But observe that
$$
([l,g],1)=\big[(l,1),(g,g)\big].
$$
Since $(g,g)\in G^\triangle\subset K$ and $(l,1)\in L\times L\subset K$, we conclude that $\big[(l,1),(g,g)\big]\in [K,K]$ and hence $([l,g],1)\in [K,K]$.
\end{proof}

The following lemma is similar in spirit to \cite[Lemma 1.1]{Bergelson_Leibman15}.
\begin{Lemma}\label{lemma_iteratedlim}
  Let $m\in\N$, let $K$ be a compact subset of a Banach space and let $a\colon \N^m\to K$ be a function.
  Assume that for every $h_1,\dots,h_m\in\N$, for every $i=1,\dots,m$, and every F\o lner sequence $(F_N)$ on $\N^i$, the limit
  $$\lim_{N\to\infty}\frac1{|F_N|}\sum_{h'\in F_N}a(h',h_{i+1},\dots,h_m)$$
  exists.
  Then for every F\o lner sequence $(F_N)_{N\in\N}$ on $\N^m$,
  $$\otherlim{C}{h_m\to\infty}\otherlim{C}{h_{m-1}\to\infty}\cdots \otherlim{C}{h_1\to\infty}a(h_1,\dots,h_m)= \lim_{N\to\infty}\frac1{|F_N|}\sum_{h\in F_N}a(h).$$
\end{Lemma}
\begin{proof}
First observe that since the limit $\lim_{N\to\infty}\frac1{|F_N|}\sum_{h\in F_N}a(h)$ exists for every F\o lner sequence, its value does not depend on the choice of F\o lner sequence.
Indeed, if two F\o lner sequences witnessed different limits, then one could form a third F\o lner sequence by alternating the sets in each of the given sequences, thus forming a F\o lner sequence for which the limit would not exist.

  We now prove the lemma by induction on $m$.
  For $m=1$ the statement becomes trivial.
  Next assume that $m>1$ and that the conclusion has been established for $m-1$.
  Fix a F\o lner sequence $(F_N)_{N\in\N}$ on $\N^{m-1}$.
  By induction hypothesis, for every $h_m\in\N$,
  $$\otherlim{C}{h_{m-1}\to\infty}\cdots \otherlim{C}{h_1\to\infty}a(h_1,\dots,h_m)= \lim_{N\to\infty}\frac1{|F_N|}\sum_{h'\in F_N}a(h',h_m).$$
  For each $H\in\N$, let $N(H)\in\N$ be such that for every $h_m\in\{1,\dots,H\}$,
  $$\left\|\lim_{N\to\infty}\frac1{|F_N|}\sum_{h'\in F_N}a(h',h_m)~-~\frac1{|F_{N(H)}|}\sum_{h'\in F_{N(H)}}a(h',h_m)\right\|<\frac1H,$$
where $\|.\|$ is the norm on the Banach space.
  Letting $\tilde F_H=F_{N(H)}\times\{1,\dots,H\}\subset\N^m$ we have that $\big(\tilde F_H)_{H\in\N}$ form a F\o lner sequence.
  Averaging the previous inequality over all $h_m\in\{1,\dots,H\}$ and taking $H\to\infty$ we conclude that
  $$\otherlim{C}{h_m\to\infty}\otherlim{C}{h_{m-1}\to\infty}\cdots \otherlim{C}{h_1\to\infty}a(h_1,\dots,h_m)= \lim_{N\to\infty}\frac1{|\tilde F_N|}\sum_{h\in\tilde F_N}a(h).$$
\end{proof}

Combining the previous lemma with Theorems B and B$^*$ from \cite{Leibman05b} we obtain the following.
\begin{Theorem}
\label{thm:Leibman-thmB}
Let $G$ be a connected simply connected nilpotent Lie group, $\Gamma$ a uniform and discrete subgroup of $G$ and define $X\coloneqq G/\Gamma$.
Let $g\colon \Z^m\to G$ be a polynomial sequence. The following are equivalent:
\begin{itemize}
\item
$g(\Z^m)\Gamma$ is dense in $X$;
\item
$\otherlim{C}{h_{m}\to\infty}\cdots \otherlim{C}{h_1\to\infty} F(g(h_1,\ldots,h_\ell)\Gamma)=\int_X F\d\mu_X$ for all $F\in C(X)$;
\item
$\otherlim{C}{h_{m}\to\infty}\cdots \otherlim{C}{h_1\to\infty} \eta(g(h_1,\ldots,h_\ell)\Gamma)=0$ for all non-trivial horizontal characters $\eta$ of $G/\Gamma$.
\end{itemize}
\end{Theorem}

\begin{Remark}
  \cref{lemma_iteratedlim} and \cref{thm:Leibman-thmB} have more general versions involving uniform Ces\`aro averages, which coincide with the $UW$-$\lim$ when $W(n)=n$.
  More precisely, every \textlim{C} appearing in either statement could be replaced with \textlim{UC}, which is defined as
  $$\otherlim{UC}{n\to\infty}a(n):=\lim_{N-M\to\infty}\frac1{N-M}\sum_{n=M}^Na(n).$$
\end{Remark}

\begin{Lemma}[cf. {\cite[Lemmas 4.3 and 5.2]{Frantzikinakis09}}]\label{lemma_orbitdiscretecontinuous}
Let $G$ be a connected simply connected nilpotent Lie group, $\Gamma$ a uniform and discrete subgroup of $G$ and define $X\coloneqq G/\Gamma$.
For every $g_0,\dots,g_\ell\in G$ there exists a set $R\subset\R^{\ell+1}$ of full measure such that for every $(\xi_0,\dots,\xi_\ell)\in R$,
$$\overline{\big\{g_0^{\xi_0 n_0}\cdots g_\ell^{\xi_\ell n_\ell}\Gamma:n_0,\dots,n_\ell\in\Z\big\}}
=
\overline{\big\{g_0^{t_0}\cdots g_\ell^{t_\ell}\Gamma:t_0,\dots,t_\ell\in\R\big\}}.$$
\end{Lemma}

\begin{proof}
Ratner's theorem (\cite[Theorem A]{Ratner91a}) implies that $\overline{\big\{g_0^{t_0}\cdots g_\ell^{t_\ell}\Gamma:t_0,\dots,t_\ell\in\R\big\}}=\tilde G/\tilde \Gamma$, where $\tilde G$ is a connected and simply connected closed subgroup of $G$ containing $g_0,\dots,g_\ell$ and $\tilde\Gamma=\Gamma\cap\tilde G$.
Let $T=\tilde G/(\tilde \Gamma[\tilde G,\tilde G])$ and let $\vartheta\colon \tilde G/\tilde \Gamma\to T$ be the canonical projection map.
In view of \cite[Theorem C]{Leibman05b} we have $\overline{\big\{g_0^{\xi_0 n_0}\cdots g_\ell^{\xi_\ell n_\ell}\Gamma:n_0,\dots,n_\ell\in\Z\big\}}=\tilde G/\tilde \Gamma$ if and only if
\begin{equation}\label{eq_proof_lemma_orbitdiscretecontinuous}
\overline{\{\vartheta(g_0^{\xi_0 n_0}\cdots g_\ell^{\xi_\ell n_\ell}\Gamma):n_0,\dots,n_\ell\in\Z\}}=T.
\end{equation}
However $T$ is a torus and $\vartheta$ is a homomorphism, so a routine argument shows that the set of $\xi=(\xi_0,\dots,\xi_\ell)\in\R^{\ell+1}$ for which \eqref{eq_proof_lemma_orbitdiscretecontinuous} holds has full measure.
\end{proof}


\begin{Lemma}
\label{lem:orbit-closure-is-subnilmanifold}
Let $G$ be a connected simply connected nilpotent Lie group, $\Gamma$ a uniform and discrete subgroup of $G$ and define $X\coloneqq G/\Gamma$.
Let $\ell,m\in \N$, let $g_0,g_1,\ldots,g_\ell\colon \Z^m\to G$ be polynomial sequences and define
$$
Y\coloneqq\overline{\{g_0(h)^{t_0}g_1(h)^{t_1}\cdot\ldots\cdot g_{\ell}(h)^{t_\ell}\Gamma: h\in\Z^m, (t_0,t_1,\ldots,t_\ell)\in\R^{\ell+1}\}}.
$$
Then $Y$ is a connected subnilmanifold of $X$.
\end{Lemma}

\begin{proof}
For each $h\in\Z^m$, let
$$Y_h\coloneqq \overline{\{g_0(h)^{t_0}\cdots g_{\ell}(h)^{t_\ell}\Gamma:(t_0,t_1,\ldots,t_\ell)\in\R^{\ell+1}\}}.$$
Using \cref{lemma_orbitdiscretecontinuous} we can find $\xi_0,\dots,\xi_\ell\in\R$ such that for every $h\in\Z^m$,
$$
Y_h=
\overline{\big\{g_0(h)^{\xi_0 n_0}\cdots g_\ell(h)^{\xi_\ell n_\ell}\Gamma:n_0,\dots,n_\ell\in\Z\big\}}.$$
Thus
$$Y=\overline{\bigcup_{h\in\Z^m}Y_h}=
\overline{\{g_0(h)^{\xi_0n_0}\cdots g_{\ell}(h)^{\xi_\ell n_\ell}\Gamma: h\in\Z^m, (n_0,\ldots,n_\ell)\in\Z^{\ell+1}\}}.$$
In particular, $Y$ is the closure of the set $\tilde g(\Z^m\times\Z^{\ell+1})$ for some polynomial map $\tilde g\colon \Z^m\times\Z^{\ell+1}\to X$.
Therefore Leibman's theorem \cite{Leibman05b} implies that $Y$ is a finite union of subnilmanifolds.
However, from the definition, it is clear that $Y$ is connected and therefore it must be a connected subnilmanifold as claimed.
\end{proof}

Next, we derive some useful Lemmas needed for the proof of \cref{thm:qualitative-green-leibman-theorem-for-discrete-tempered}.

\begin{Lemma}[Van der Corput lemma]
\label{lem:C_lim-W_lim-vdC}
Let $m\in\N\cup\{0\}$, let $W\in\F_1$ and suppose that $x\colon \N^{m+1}\to\C$ is a bounded complex-valued sequence.
If
$$
\otherlimsup{C}{h_{m+1}\to\infty}\otherlimsup{C}{h_{m}\to\infty}\cdots \otherlimsup{C}{h_1\to\infty}\otherlimsup{UW}{n\to\infty} x(n+h_{m+1},h_1,\ldots,h_m)\overline{x(n,h_1,\ldots,h_m)} = 0
$$
then
$$
\otherlimsup{C}{h_{m}\to\infty}\cdots \otherlimsup{C}{h_1\to\infty}\otherlimsup{UW}{n\to\infty}x(n,h_1,\ldots,h_m)= 0.
$$
\end{Lemma}

\begin{proof}
Notice that for any bounded sequence $a\colon \N\to\C$ and $H\in\N$ we have
\begin{eqnarray*}
\frac1{H^2}\sum_{h_1,h_2=1}^Ha(h_2-h_1)
&=&
\frac1{H^2}\sum_{h_1=1}^H\sum_{d=1}^{H-h_1}a(d)+a(-d)+\Oh(1/H)
\\&=&
\frac1{H^2}\sum_{d=1}^{H-1}(H-d)\big(a(d)+a(-d)\big)+\Oh(1/H)
\end{eqnarray*}
where the $O(1/H)$ term accounts for the case $h_1=h_2$. Hence,
\begin{equation}\label{eq_proof_vanderCorput1}
\frac1{H^2}\sum_{h_1,h_2=1}^Ha(h_2-h_1)
=
\E_{d\in[H]}\left(1-\frac{d}H\right)\big(a(d)+a(-d)\big)+\Oh(1/H).
\end{equation}
Next, let $(u_n)_{n\in\N}$ be a bounded sequence of complex numbers.
Notice that, since $\Delta W$ is eventually decreasing, for every $H\leq N-M$
$$\E_{n\in[M,N]}^Wu_n=\E_{n\in[M,N]}^W\E_{h\in[H]} u_{n+h}+\Oh\left(\frac{H}{W(N)-W(M)}\right)$$
and so the Cauchy-Schwarz inequality implies that
\begin{eqnarray*}
\left|\E_{n\in[M,N]}^W u_n\right|^2
&\leq&
\E_{n\in[M,N]}^W\frac1{H^2}\sum_{h_1,h_2=1}^H u_{n+h_1}\overline{u_{n+h_2}}+\Oh\left(\frac{H}{W(N)-W(M)}\right)
\\&=&
\frac1{H^2}\sum_{h_1,h_2=1}^H\E_{n\in[M,N]}^Wu_{n}\overline{u_{n+h_2-h_1}}+\Oh\left(\frac{H}{W(N)-W(M)}\right)
\end{eqnarray*}
where the last equality follows from the general fact that $\E_{n\in[M,N]}^Wa(n+h)=\E_{n\in[M,N]}^Wa(n)+\Oh(h/(W(N)-W(M)))$ for any bounded sequence $a:\N\to\C$.
Combining the above with \eqref{eq_proof_vanderCorput1} for $a(d)=\E_{n\in[M,N]}^Wu_{n}\overline{u_{n+d}}$ and observing that $a(-d)=\overline{a(d)}+\Oh(d/(W(N)-W(M)))$ we get
\begin{eqnarray*}
\left|\E_{n\in[M,N]}^W u_n\right|^2
&\leq&
\E_{d\in[H]}\left(1-\frac{d}H\right)2\text{Re} \E_{n\in[M,N]}^Wu_{n}\overline{u_{n+d}}+\Oh\left(\frac H{W(N)-W(M)}+\frac1H\right)
\\&\leq&
2\E_{d\in[H]}\Big|\E_{n\in[M,N]}^Wu_{n}\overline{u_{n+d}}\Big|+\Oh\left(\frac H{W(N)-W(M)}+\frac1H\right).
\end{eqnarray*}

Now, denote by $h=(h_1,\dots,h_m)$ and let $x(n,h;h_{m+1})\coloneqq x(n+h_{m+1},h)\overline{x(n,h)}$.
We get

$$
\otherlimsup{UW}{n\to\infty}x(n,h)
\leq
\sqrt{2\E_{d\in[H]}\otherlimsup{UW}{n\to\infty}x(n,h;d)+\Oh\left(\frac1H\right)}
$$
and hence applying again the Cauchy-Schwarz inequality we obtain
$$
\otherlimsup{C}{h_m\to\infty}\cdots\otherlimsup{C}{h_1\to\infty}\otherlimsup{UW}{n\to\infty}x(n,h)
\leq
\sqrt{2\E_{d\in[H]}\otherlimsup{C}{h_m\to\infty}\cdots\otherlimsup{C}{h_1\to\infty}\otherlimsup{UW}{n\to\infty}x(n,h;d)+\Oh\left(\frac1H\right)}.
$$
Since the left hand side does not depend on $H$, we can let $H\to\infty$ in the right hand side and use the hypothesis to conclude that
$$
\otherlimsup{C}{h_m\to\infty}\cdots\otherlimsup{C}{h_1\to\infty}\otherlimsup{UW}{n\to\infty}x(n,h)
\leq
\sqrt{2\lim_{H\to\infty}\E_{d\in[H]}\otherlimsup{C}{h_m\to\infty}\cdots\otherlimsup{C}{h_1\to\infty}\otherlimsup{UW}{n\to\infty}x(n,h;d)}=0.
$$
\end{proof}

\begin{Remark}
  We observe that the ``uniform Ces\`aro'' version of \cref{lem:C_lim-W_lim-vdC} (where every instance of $\otherlim{C}{}$ is replaced with $\otherlim{UC}{}$) is also true and can be proved in a very similar way.
\end{Remark}

\begin{Lemma}
\label{lem:equidistribution-on-product}
Suppose $d_1,d_2,\ell,m\in \N$, $T_1\coloneqq \R^{d_1}/\Z^{d_1}$ and $T_2\coloneqq \R^{d_2}/\Z^{d_2}$. Let $\pi_1\colon \R^{d_1}\to T_1$ and $\pi_2\colon \R^{d_2}\to T_2$ denote the corresponding natural factor maps and let $p_0,p_1,\ldots,p_\ell\colon\Z^m\to\R^{d_1}$ and $q_0,q_1,\ldots,q_{\ell-1}\colon\Z^m\to\R^{d_2}$ be polynomial sequences such that
$$
\pi_1\big(\big\{p_0(h){t_0}+p_1(h){t_1}+\ldots+ p_{\ell}(h){t_\ell}: h\in\Z^m, (t_0,t_1,\ldots,t_\ell)\in\R^{\ell+1}\big\}\big)~\text{is dense in $T_1$},
$$
and
$$
\pi_2\big(\big\{q_0(h){t_0}+q_1(h){t_1}+\ldots+ q_{\ell-1}(h){t_{\ell-1}}): h\in\Z^m, (t_0,t_1,\ldots,t_{\ell-1})\in\R^{\ell}\big\}\big)
~\text{is dense in $T_2$}.
$$
Define new polynomial sequences
$q_1',\ldots,q_{\ell}'\colon\Z^{m+1}\to\R^{d_2}$ as
\begin{equation}
\label{eq:equidistribution-on-product-1}
q_i'(h_1,\ldots,h_{m+1})\coloneqq \sum_{j=1}^i \binom{h_{m+1}}{j} q_{i-j}(h_1,\ldots,h_{m}).
\end{equation}
Set $\tilde{p}_0(h_1,\ldots,h_{m+1})\coloneqq(p_0(h_1,\ldots,h_m),0)$ and, for $i=1,\ldots,\ell$,
$$
\tilde{p}_i(h_1,\ldots,h_{m+1})\coloneqq(p_i(h_1,\ldots,h_m),q_i'(h_1,\ldots,h_{m+1})).
$$
Then
$$
(\pi_1\otimes \pi_2)\big(\big\{\tilde{p}_0(h'){t_0}+\tilde{p}_1(h'){t_1}+\ldots+ \tilde{p}_{\ell}(h'){t_\ell}: h'\in\Z^{m+1}, (t_0,t_1,\ldots,t_\ell)\in\R^{\ell+1}\big\}\big).
$$
is dense in $T_1\times T_2$
\end{Lemma}

\begin{proof}
Let
$$T_3\coloneqq \overline{(\pi_1\otimes \pi_2)\big(\big\{\tilde{p}_0(h'){t_0}+\ldots+ \tilde{p}_{\ell}(h'){t_\ell}: h'\in\Z^{m+1}, (t_0,\ldots,t_\ell)\in\R^{\ell+1}\big\}\big)}.$$
According to \cref{lem:orbit-closure-is-subnilmanifold}, $T_3$ is a subtorus of $T_1\times T_2$. If $T_3$ is a proper subtorus of $T_1\times T_2$ then there exists a non-trivial continuous group character $\chi\colon T_1\times T_2\to\{w\in\C: |w|=1\}$ whose kernel contains $T_3$. So in order to show that $T_3=T_1\times T_2$ it suffices to show that the only continuous group characters $\chi$ of $T_1\times T_2$ whose kernel contains $T_3$ is the trivial character. Equivalently, it suffices to show that if the function
$
\psi_\chi\colon \Z^{m+1}\times \R^{\ell+1} \to \C
$
given by
$$
\psi_\chi(h',t_0,\ldots,t_{\ell})=\chi \circ (\pi_1\otimes \pi_2)\big(\tilde{p}_0(h'){t_0}+\tilde{p}_1(h'){t_1}+\ldots+ \tilde{p}_{\ell}(h'){t_\ell}\big)
$$
is constant equal to $1$, then $\chi$ is the trivial character.

Any character $\chi$ of $T_1\times T_2$ can be written as $\chi=\chi_1\otimes\chi_2$, where $\chi_1$ is a continuous group character of $T_1$ and $\chi_2$ is a continuous group character of $T_2$. Hence
\begin{eqnarray*}
\psi_\chi(h',t_0,\ldots,t_{\ell})
&=&
\big((\chi_1\circ \pi_1)\otimes (\chi_2\circ\pi_2)\big)\big(\tilde{p}_0(h'){t_0}+\tilde{p}_1(h'){t_1}+\ldots+ \tilde{p}_{\ell}(h'){t_\ell}\big)
\\
&=&
(\chi_1\circ \pi_1)\big(p_0(h){t_0}+\ldots+ p_{\ell}(h){t_\ell}\big)
(\chi_2\circ \pi_2)\big(q_1'(h'){t_1}+\ldots+ q_{\ell}'(h'){t_\ell}\big),
\end{eqnarray*}
where $h'=(h_1,\ldots,h_{m+1})$ and $h=(h_1,\ldots,h_m)$.
Pick $\tau^{(1)}\in\Z^{d_1}$ such that $\chi_1\circ\pi_1(x)=e^{2\pi i \langle x,\tau^{(1)}\rangle}$ for all $x\in \R^{d_1}$ and, similarly, pick $\tau^{(2)}\in\Z^{d_2}$ such that $\chi_2\circ\pi_2(x)=e^{2\pi i \langle x,\tau^{(2)}\rangle}$ for all $x\in \R^{d_2}$. Since $\psi_\chi(h',t_0,\ldots,t_{\ell})=1$ for all $h'=(h_1,\ldots,h_{m+1})\in\Z^{m+1}$ and $(t_0,\ldots,t_{\ell})\in\R^{\ell+1}$, we have that
$$
\big\langle p_0(h){t_0}+p_1(h){t_1}+\ldots+ p_{\ell}(h){t_\ell},\tau^{(1)} \big\rangle + \big\langle q_1'(h'){t_1}+\ldots+ q_{\ell}'(h'){t_\ell},\tau^{(2)} \big\rangle \in\Z
$$
for all $h'=(h_1,\ldots,h_{m+1})\in\Z^{m+1}$ and all $(t_0,\ldots,t_{\ell})\in\R^{\ell+1}$.
It follows that
\begin{equation}
\label{eqn:char-is-trivial-1}
\langle p_0(h),\tau^{(1)}\rangle =0
\end{equation}
and
\begin{equation}
\label{eqn:char-is-trivial-2}
\langle p_i(h),\tau^{(1)}\rangle + \langle q_i'(h'),\tau^{(2)}\rangle=0,\qquad\forall i\in\{1,\ldots,\ell\}.
\end{equation}
Note that \eqref{eq:equidistribution-on-product-1} and \eqref{eqn:char-is-trivial-2} together imply
\begin{equation}
\label{eqn:char-is-trivial-3}
\langle p_{i}(h),\tau^{(1)}\rangle =  \sum_{j=1}^{i} \binom{h_{m+1}}{j} \langle - q_{i-j}(h),\tau^{(2)}\rangle,\qquad \forall h\in\Z^m,~ \forall h_{m+1}\in \Z.
\end{equation}
By comparing coefficients of the variable $h_{m+1}$ we see that \eqref{eqn:char-is-trivial-3} can only hold if
\begin{equation}
\label{eqn:char-is-trivial-4}
\langle p_i(h),\tau^{(1)}\rangle =0 \qquad\text{and}\qquad \langle q_{i-j}(h),\tau^{(2)}\rangle=0,\qquad\forall i\in\{1,\ldots,\ell\},\forall j\in\{1,\ldots,i\}.
\end{equation}
It follows that $\langle p_i(h),\tau^{(1)}\rangle =0 $ for all $i=0,1,\ldots,\ell$ and $\langle q_{i}(h),\tau^{(2)}\rangle=0$ for all $i=0,1,\ldots,\ell-1$. Since $\pi_1\big(\big\{p_0(h){t_0}+p_1(h){t_1}+\ldots+ p_{\ell}(h){t_\ell}: h\in\Z^m, (t_0,t_1,\ldots,t_\ell)\in\R^{\ell+1}\big\}\big)$ is dense in $T_1$, we conclude that $\tau^{(1)}=0$ and, likewise, since $\pi_2\big(\big\{q_0(h){t_0}+q_1(h){t_1}+\ldots+ q_{\ell-1}(h){t_{\ell-1}}): h\in\Z^m, (t_0,t_1,\ldots,t_{\ell-1})\in\R^{\ell}\big\}\big)$ is dense in $T_2$, we have that $\tau^{(2)}=0$. This proves that $\chi=\chi_1\otimes\chi_2$ is the trivial character.
\end{proof}

\bibliography{refs-joel}
\bibliographystyle{plain}


\bigskip
\footnotesize

\noindent
Vitaly\ Bergelson\\
\textsc{The Ohio State University}\par\nopagebreak
\noindent
\href{mailto:bergelson.1@osu.edu}
{\texttt{bergelson.1@osu.edu}}
\\

\noindent
Joel\ Moreira\\
\textsc{University of Warwick}\par\nopagebreak
\noindent
\href{mailto:joel.moreira@warwick.ac.uk}
{\texttt{joel.moreira@warwick.ac.uk}}
\\

\noindent
Florian K.\ Richter\\
\textsc{Northwestern University}\par\nopagebreak
\noindent
\href{mailto:fkr@northwestern.edu}
{\texttt{fkr@northwestern.edu}}

\end{document}